%% file: main.tex
\documentclass[final]{siamart220329}
\usepackage{graphicx} 
\usepackage[english]{babel}
\usepackage[numbers,square,sort&compress]{natbib}

\usepackage{amsmath,amssymb,graphicx,theorem}
\usepackage{xr}
\usepackage{xcolor}
\usepackage{csquotes}
\newcommand{\cdummy}{\cdot}
\newcommand{\infixand}{\text{ and }}
\newcommand{\nobracket}{}
\newcommand{\tmtextbf}[1]{\text{{\bfseries{#1}}}}
\newcommand{\nosymbol}{}
\newcommand{\tmop}[1]{\ensuremath{\operatorname{#1}}}

\newcommand{\assign}{:=}
\newcommand{\tmem}[1]{{\em #1\/}}
\newcommand{\infixor}{\text{ or }}
\newcommand{\nin}{\not\in}

{\theorembodyfont{\rmfamily}\newtheorem{question}{Assumption}}
{\theorembodyfont{\rmfamily}\newtheorem{remark}{Remark}}
\newcommand{\mathi}{\mathrm{i}}
\newtheorem{case}{Case}

\newcommand{\bigverbar}{\bigg|}

\numberwithin{equation}{section}

\title{Asymptotic analysis for Bloch electrons with Weyl nodes\thanks{Zhennan Zhou is supported by the National Key R\&D Program of China, Project Number 2021YFA1001200, and the NSFC, grant Number 12031013, 12171013. Changhe Yang is supported by the NSF Grant DMS-2205590. }}

\author{Jianfeng Lu\thanks{Departments of Mathematics, Physics, and Chemistry, Duke University, Durham NC 27708, USA (\email{jianfeng@math.duke.edu})} 
\and Changhe Yang\thanks{Computing and Mathematical Sciences (CMS) Department, California Institute of Technology, Pasadena CA 91125, USA} 
\and 
Zhennan Zhou\thanks{Beijing International Center for Mathematical Research, Peking University, Beijing, 100871, China (\email{zhennan@bicmr.pku.edu.cn})}}

\begin{document}
\maketitle

\begin{abstract}
In this paper, we study the semiclassical behavior of Bloch electrons in the presence of Weyl nodes, which are singular points in the band structure of certain materials. We carry out asymptotic analysis
and present a rigorous derivation of the semiclassical asymptotic expansion of the current of Bloch electrons with the presence of Weyl nodes. The analysis shows that the current contains two parts, one independent of the Weyl nodes and the other a contribution from the singular points. This work provides a theoretical foundation towards a rigorous justification of recent scientific discoveries in Weyl semimetals. The main innovation of this paper is a new strategy to deal with the singular points with quantitative estimates, which may have broader applications in multiscale models with singularities.
\end{abstract}

\begin{keywords}
 Semiclassical Analysis, Bloch Electrons, Weyl Nodes.
\end{keywords}

\begin{AMS} 81Q20; 35B40.
\end{AMS}

\input{Chapter_1}
\input{Chapter_2}
\input{Chapter_3}
\input{Chapter_4}
\input{Chapter_5}

\input{Chapter_6}

\appendix
\input{Appendix_A}

\input{Appendix_B}
\input{Appendix_C}
\input{Appendix_D}
\input{Appendix_E}

\bibliographystyle{siamplain}
\bibliography{Crystal}

\end{document}

%% file: Chapter_1.tex
\section{Introduction}

In this paper, we consider the semiclassical asymptotic behavior of Bloch
electrons in an external electric field as a perturbation. The motion of Bloch
electrons is one of the fundamental problems in solid-state physics, which is
concerned with the quantum dynamics of electrons in a periodic crystal
lattice, and is found to have connections with many intriguing physical phenomena, such as the Hall effect
{\cite{hoddeson_development_1987,bloch_uber_1929,peierls_zur_1929,reed_analysis_1978}}.
It is worth emphasizing that the Berry phase, first introduced by
{\cite{berry_quantal_1984}}, plays an important role in Bloch dynamics. (see
e.g.
{\cite{bloch_ultracold_2005,culcer_transport_2020,gao_field_2014,xiao_berry_2010}}).

There are two basic scenarios in the Bloch theory: the adiabatic and the
non-adiabatic regimes. In the adiabatic regime, the energy bands are isolated
from one another and the electrons effectively remain in the same energy bands
as they move through the crystal. 
It has been thoroughly investigated in the mathematical
literature (see e.g.
{\cite{teufel_adiabatic_2003,spohn_adiabatic_2001,teufel_semiclassical_2002,hovermann_semiclassical_2001}}).
In contrast, in the non-adiabatic regime, the crossings between energy bands
occur and transitions between different energy levels become significant. In
spite of many scientific applications due to band crossing phenomenon, the
mathematical understanding of such models are comparatively less complete.

Weyl equation was first proposed to describe massless fermions with a definite chirality \cite{weyl1929gravitation}. When degeneracies occur in electronic band structures, the dispersion in the vicinity of these 3D band touching points  is
generically linear and resembles the Weyl equation \cite{herring1937accidental}, and thus these touching points are referred as \emph{Weyl nodes}. Weyl nodes lead to exotic surface states in the form of Fermi arcs, which manifests the topological aspects of Weyl fermions \cite{wan2011topological,armitage2018weyl}.  Weyl nodes can occur in certain types of materials, such as topological
insulators and semimetals {\cite{wan_topological_2011,armitage_weyl_2018}},
where the band structure exhibits a non-trivial topology. 

The Weyl nodes have been studied in the mathematical literature from the aspect of spectrum theory. For example, in
{\cite{fefferman_bifurcations_2016,fefferman_honeycomb_2012,fefferman_edge_2016}},
the authors illustrated the properties of Dirac points, as the 2D analogy of
Weyl nodes, in honeycomb structures and the existence of topologically
protected edge states. For Weyl nodes in 3D, the mathematical study is still sparse, while there are also some recent progress. For example, in \cite{giuliani2021anomaly}, the authors proved the universality of the quadratic response of the quasi-particle flow in some 3D lattice models for Weyl semimetals.

In this paper, we take the usual ``independent electron'' assumption and
neglect the interaction between electrons. Then the dynamic of Bloch electrons
is described by the following semiclassical Schr\"{o}dinger equation  with a
periodic lattice potential and a slow-varying scalar potential,
\begin{equation}
  \label{Schrodinger equation 0} \mathi \varepsilon \frac{\partial}{\partial
  t} \psi^{\varepsilon} = (- \frac{\varepsilon^2}{2} \Delta + V (\frac{x}{\varepsilon}) -
  \varepsilon E (t) \cdummy x) \psi^{\varepsilon},\quad  x \in \mathbb{R}^3,\quad  t > 0.
\end{equation}
Here, $\psi^{\varepsilon}$ is the wave function and $V$ is the lattice potential, which
will be formally defined in Section \ref{1}, and may result in Weyl nodes in
the band structure. $\varepsilon \ll 1$ is referred to as the semiclassical parameter, which indicates the typical oscillation scale of the wave function as well as the size of the periodic lattice. When $\varepsilon \to 0$, the semiclassical analysis produces an asymptotic solution that captures the classical behavior of Bloch electrons. Many works are
devoted to this direction of research based on
the WKB ansatz {\cite{carles_semiclassical_2004,lu_bloch_2022}} or the Wigner transform
{\cite{e_asymptotic_2013,markowich_wignerfunction_1994,gerard_homogenization_1997,sparber_wigner_2003}}, which are tailored to treat the periodic lattice potential.

The optical properties of a material can be described by its response to
external electromagnetic fields. In this work, we assume that the Bloch dynamics
are subject to a small external electric field $\varepsilon E (t)$, which
results in an $O (\varepsilon)$ perturbation, $\varepsilon E (t) \cdummy x$,
to the potential field. Here, we suppose that the electric field is
homogenuous in space and periodic in time. Such electronic systems, subjected to periodic driving forces, are
also known as Floquet-driven models. Floquet-driven models have been
used to study a wide range of phenomena in physics and chemistry, like the
laser systems {\cite{drese_floquet_1999}}, the photoelectric
effect {\cite{zhang_anomalous_2014}} and the Stark
effect {\cite{sie_valley-selective_2018}}.

We aim to carry out asymptotic analysis for this model due to two
aspects of motivations. Mathematically, there are a lot of works on Bloch
electrons in spatially homogenuous electric field, such as
{\cite{wannier_wave_1960,krieger_time_1986,krieger_quantum_1987}}. However, to
the best of our knowledge, none of the existing of works is able to explicitly
quantify the effect of the Weyl nodes in Bloch dynamics. Besides, in
semiclassical regime, the small parameter $\varepsilon$ makes the wave
function highly oscillatory in time and space. Therefore, 
rigorous asymptotic analysis for such a multiscale model with
singularities, i.e. the Weyl nodes, is even more challenging. 

This work is also partially inspired by some recent works in the physics literature {\cite{de_juan_quantized_2017}}, which shows that in Weyl semimetals,
the injection contribution to the CPGE (circular photogalvanic effect) trace is
effectively quantized in terms of the fundamental constants with no
material-dependent parameters. However, the mathematical explanation of this
phenomenon is still absent. Our analysis of the semiclassical Bloch model with
Weyl nodes is developed to serve as a theoretical foundation towards a rigorous
justification of the results.

The presence of Weyl nodes and the semiclassical regime endow the model
\eqref{Schrodinger equation 0} with a multi-scale nature, which we will
elaborate later in Section \ref{3}. It calls for novel expansion treatments
and convergence estimates in developing a complete asymptotic analysis result.
The main innovation of this paper is that we develop a new strategy to deal
with the the singular points with quantitative estimates. More precisely, we
first use the Bloch transform to convert the semiclassical Schr\"{o}dinger
equation into an equivalent dynamical system for the density of the Bloch
functions \eqref{c equation}, which agrees with the result in the physics
literature {\cite{sipe_second-order_2000}}. Then we divide the Brillouin zone
into three parts: the regular part, the singular part and the intermediate part. We deal
with each part with different strategies and finally combine them together to
derive an asymptotic expansion with error estimates..

The main result of this paper is to give a derivation of the semiclassical
asymptotic behavior of the current of Bloch electrons with the presence of  Weyl nodes. In fact, we can also apply the asymptotic theory to a large family
of physical observables.  Our final result can be briefly summarized as follows. In the presence of the $O (\varepsilon)$ external electric field, the current $J(t)$ has the following asymptotic expansion,
\begin{equation}
    J(t) = J_1(t) + J_2(t) + o(\varepsilon^3).
\end{equation}
It contains two
parts $J_1$ and $J_2$. $J_1$ depends on the whole Brillouin zone, which is
independent of the presence of Weyl nodes. When Weyl nodes are present, the third order term of $J_1$ becomes singular. As a result, $J_2$, the contribution from Weyl nodes, appears. Because of the singular nature of Weyl nodes, the leading
order contribution of $J_2$ is $O
(\varepsilon^3 \log \varepsilon)$ instead of $O (\varepsilon^3)$. In some special cases, the zeroth to second order terms of $J_1$ may vanish, and then $J_2$ becomes the leading order term. The precise statement of the results is presented in Section \ref{3}.

In {\cite{de_juan_quantized_2017}}, the
authors focused on the second part and derived a quantized value of the system.
Therefore, the final results justified the linear approximation used in
{\cite{de_juan_quantized_2017}} from a mathematical aspect. We also note that the influence of Weyl nodes appears only in high order terms. There are several works on the high order expansions of semiclassical Schrodinger equation with periodic potentials, such as \cite{lu_bloch_2022,gaim2021higher}. Our result is a generalization of these works in the case when Weyl nodes are present.

This paper is organized as following: In Section \ref{1}, we transform the
original equation into the density of Bloch functions representation \eqref{c
equation}, which facilitates further analysis. In Section \ref{2}, we derive
some results about Weyl nodes with a similar argument to
{\cite{fefferman_honeycomb_2012}}. In Section \ref{3}, we present our main
result and lay out the proof strategy. In Section \ref{4}, we finish the proof
of the main result.

%% file: Chapter_2.tex
\section{Preliminaries}\label{1}

Recall the Schr{\"o}dinger equation with a
periodic lattice potential and a spatial uniform electric field 
\begin{equation}
  \label{Schrodinger equation} \mathi \varepsilon \frac{\partial}{\partial t}
  \psi = \bigl(- \frac{\varepsilon^2}{2} \Delta + V (\frac{x}{\varepsilon}) -
  \varepsilon E (t) \cdummy x \bigr) \psi,\quad x \in \mathbb{R}^3, \quad t > 0,
\end{equation}
with initial data given by
\begin{equation}
  \psi (t = 0, x) = \psi^{i n} (x) . \label{initial data}
\end{equation}
For simplicity, we suppress all dependence on $\varepsilon$ in the notation. In this paper, we only study
initial data with some special structures, which we will explain in more details in
Section \ref{1.2} (see \eqref{id ansatz}). $V (z)$ is the periodic lattice potential, and $- \varepsilon E (t) \cdummy x$ is the external electric potential.

\subsection{Bloch decomposition}

As we mentioned before, $\varepsilon$ indicates the typical size of the lattice. It is natural to use a change of variable $z = x / \varepsilon$. Then the Hamiltonian operator becomes
\begin{equation}
  - \frac{\varepsilon^2}{2} \Delta_x + V (\frac{x}{\varepsilon}) - \varepsilon
  x \cdot E (t) = - \frac{1}{2} \Delta_z + V (z) - \varepsilon^2 z \cdot E (t)
\end{equation}
In the absence of the external potential, the Hamiltonian is given by
\begin{equation}
  H_{\tmop{per}} = - \frac{1}{2} \Delta_z + V (z) .
\end{equation}
It is translational invariant with respect to the lattice $\mathbb{L}$. As a
result, the spectrum of the Hamiltonian can be understood by the Bloch-Floquet
theory. In particular, we have the periodic Bloch wave functions $\Psi^0_n (z,
p)$, given as the eigenvectors of
\begin{equation}
  \label{eigen01} H^0 (p) \Psi_n^0 (z, p) \assign \left( \frac{1}{2}  (-
  \mathi \nabla_z + p)^2 + V (z) \right) \Psi_n^0 (z, p) = \lambda_n (p)
  \Psi^0_n (z, p)
\end{equation}
on $\Gamma$ with periodic boundary conditions. Here $\Gamma$ is the
{\tmem{unit cell}} of lattice $\mathbb{L}$ and $p \in \Gamma^{\ast}$ is the
crystal momentum, where $\Gamma^{\ast}$ denotes the {\tmem{first Brillouin
zone}} (unit cell of the reciprocal lattice). For each fixed $p \in
\Gamma^{\ast}$, the Bloch Hamiltonian $H^0 (p)$ is a self-adjoint operator
with compact resolvent, the spectrum of which is given by
\[ \sigma (H^0 (p)) = \{ \lambda_n (p) \mid n \in \mathbb{Z}_+ \} \subset
   \mathbb{R}, \]
where the eigenvalues $\lambda_n (p)$ (counting multiplicity) are increasingly
ordered $\lambda_1 (p) \le \cdots \le \lambda_n (p) \le \lambda_{n + 1} (p)
\le \cdots$. The spectrum of $H^0$ is exactly the same as that of $H_{per}$. It is shown by Nenciu {\cite{nenciu_dynamics_1991}} that for any
$n \in \mathbb{Z}_+$ there exists a closed set $A_n \subset \Gamma^{\ast}$ of
measure zero such that
\begin{equation}
  \lambda_{n - 1} (p) < \lambda_n (p) < \lambda_{n + 1} (p), \quad p \in
  \Gamma^{\ast} \setminus A_n,
\end{equation}
and moreover $\lambda_n (p)$ and $\Psi^0_n (\cdummy, p)$ are analytic on $p
\in \Gamma^{\ast} \setminus A_n$.

Given $f \in L^2 (\mathbb{R}^d)$, we recall the {\tmem{Bloch transform}},
which is an isometry from $L^2 (\mathbb{R}^d)$ to $L^2  (\Gamma \times
\Gamma^{\ast})$
\begin{equation}
  \label{psiphi} \mathfrak{B}f (z, p) = \frac{1}{\lvert \Gamma^{\ast}
  \rvert^{1 / 2}}  \sum_{X \in \mathbb{L}} f (z + X) e^{- \mathi p \cdot (z +
  X)},
\end{equation}
where $\lvert \Gamma \rvert$ denotes the volume of the unit cell of the
lattice $\mathbb{L}$. The inverse transform is given by
\begin{equation}
  \label{phipsi} f (z) = \frac{1}{\lvert \Gamma^{\ast} \rvert^{1 / 2}} 
  \int_{\Gamma^{\ast}} e^{\mathi p \cdot z} \mathfrak{B}f (z, p)
  \hspace{0.17em} \mathrm{d} p.
\end{equation}
With the external potential $\varepsilon E(t)\cdot x$, our analysis needs perturbation of the Bloch
waves. For this, let us recall the perturbation theory in the context of Bloch
wave functions.  Let a family of Hamiltonian $H^\varepsilon$  denote the Bloch transform of the original Hamiltonian. We consider the asymptotic expansion for $H^{\varepsilon}(t)$ as
\begin{equation}
  H^{\varepsilon}  (t) = H^0(t) + \varepsilon H^1  (t) +
  \varepsilon^2 H^2  (t) +\mathcal{O} (\varepsilon^3) .
\end{equation}
where the zeroth order term is the Bloch Hamiltonian $H^0 = H^0 (p) = \tfrac{1}{2} (- \mathi \nabla_z
+ p)^2 + V (z)$, which is time-independent. The slow-varying potential $\varepsilon E(t)\cdot x = \varepsilon^2 E(t)\cdot z$ corresponds a second order
correction to the unperturbed Hamiltonian.
More specifically, 
we note that using Bloch transformation, we have
\[ zf (z) = \int_{\Gamma^{\ast}} ze^{\mathi p \cdot z} \mathfrak{B}f (z, p)
   \hspace{0.17em} \mathrm{d} p = \int_{\Gamma^{\ast}} - \mathi \nabla_p
   e^{\mathi p \cdot z} \mathfrak{B}f (z, p) \hspace{0.17em} \mathrm{d} p =
   \int_{\Gamma^{\ast}} \mathi e^{\mathi p \cdot z} \nabla_p \mathfrak{B}f (z,
   p) \hspace{0.17em} \mathrm{d} p. \]
Hence, we arrived at the second order correction to the Hamiltonian, which is independent of $x$
\begin{equation}
  \label{correction H1} H^2  (t) = - \mathi E (t) \cdot \nabla_p .
\end{equation}

\subsection{Berry connection}\label{1.2}

Let us now apply Bloch transform to \eqref{Schrodinger equation}. 
For a given wave function $\psi (t, x) \in L^3 (\mathbb{R}^3)$, we use an $L^2$-invariant rescaling,
\begin{equation}
  \label{rescaling} \psi_r (t, z) = \varepsilon^{3 / 2} \psi (t, \varepsilon
  z),
\end{equation}
and
\[ \psi^{i n}_r (t, z) = \varepsilon^{3 / 2} \psi^{i n} (t, \varepsilon z). \]
It is easy to see that $\| \psi_r (t, z) \|_{L^2
(\mathbb{R}^3)} = \| \psi (t, x) \|_{L^2 (\mathbb{R}^3)}$. Recall that the Bloch transform defined by \eqref{psiphi} is
\begin{equation}
  \mathfrak{B} \psi_r (t, z, p) = \frac{1}{\lvert \Gamma^{\ast} \rvert^{1 /
  2}}  \sum_{X \in \mathbb{L}} \psi_r (t, z + X) e^{- \mathi p \cdot (z + X)}
  .
\end{equation}
Apply the rescaling \eqref{rescaling} and the Bloch transform to \eqref{Schrodinger equation} and get
  \begin{align}
    \mathi \varepsilon \frac{\partial}{\partial t} \mathfrak{B} \psi_r (t, z,
    p) = \label{bloch equation} H^{\varepsilon} \mathfrak{B} \psi_r (t, z, p)
    .
  \end{align}
With the basis $\{ \Psi^0_n \}_{n = 0}^{+ \infty}$, We expand $\mathfrak{B} \psi_r$ and $\mathfrak{B} \psi_r^{i n}$  and obtain
\begin{eqnarray}
  \mathfrak{B} \psi_r (t, z, p) & = & \sum_m \hat{\psi}_m^{\varepsilon} (t, p)
  \Psi^0_m (z, p),  \label{psi expansion}\\
  \mathfrak{B} \psi_r^{i n} (z, p) & = & \sum_m \hat{\psi}_m^{in} (p) \Psi^0_m
  (z, p).  \label{psi initial expansion}
\end{eqnarray}
Here, $\hat{\psi}_m^{\varepsilon}$ and $\hat{\psi}_m^{in} (p)$ are defined by
\begin{eqnarray*}
  \hat{\psi}_m^{\varepsilon} (t, p) & = & \langle \mathfrak{B} \psi_r (t, z,
  p), \Psi^0_m (z, p) \rangle,\\
  \hat{\psi}_m^{in} (p) & = & \langle \mathfrak{B} \psi_r^{i n} (z, p),
  \Psi^0_m (z, p) \rangle .
\end{eqnarray*}
and the inner product $\langle f, g \rangle$ is given by 
\[ \langle f, g \rangle = \int_{\Gamma  ^{}}  \overline{f} (z) g (z) d z. \]
We label the variables of integral as the subscript of the right bracket and the region
of integral is the domain of the function.

Since $\mathfrak{B}$ is an isometry, thanks to Parseval's Theorem, it follows that
\[ \sum_{i = 0}^{+ \infty} \| \hat{\psi}_m^{in} (p) \|_{L^2 (\Gamma^{\ast})} =
   1. \]
For any initial data $\psi^{i n}$, we can regard the wave function as a
superposition of its components on different energy bands. Since
\eqref{Schrodinger equation} is linear, we only need to study each component.
Therefore, without loss of generality, we can assume that
\begin{equation}
    \forall m \geqslant 1, \hat{\psi}_m^{in} (p) = 0, 
\end{equation}
and
\begin{equation}
    \| \hat{\psi}_0^{in} (p) \|_{L^2 (\Gamma^{\ast})} = 1.
\end{equation}
This means we focus on the component of the 0-th energy band.

Therefore, we will assume the initial condition $\psi^{i n}$ in \eqref{initial data} is given by
  \begin{equation}
    \psi^{i n} (x) = \varepsilon^{- 3 / 2} \mathfrak{B}^{- 1}
    (\hat{\psi}_0^{in} (p) \Psi^0_0 (x / \varepsilon, p)) . \label{id ansatz}
  \end{equation}
  After $\hat{\psi}_0^{in}$ is chosen, $\psi^{i n}$ is determined. Here, we
  assign $\hat{\psi}_0^{in}$ as a smooth periodic function on $\Gamma^{\ast}$.
  This is our ansatz for the initial data $\psi^{i n}$.

Plugging the expansion \eqref{psi expansion} into \eqref{bloch equation} yields
\begin{equation}
  \sum_m \mathi \varepsilon \frac{\partial}{\partial t} 
  \hat{\psi}_m^{\varepsilon} (t, p) \Psi^0_m (z, p) = \sum_m H^{\varepsilon} 
  (\hat{\psi}_m^{\varepsilon} (t, p) \Psi^0_m (z, p)) .
\end{equation}
Recall the expansion of $H^\varepsilon$ given by
\begin{equation}
  H^{\varepsilon}  (t) = H^0 (p) + \varepsilon^2 H^2  (t), \quad H^0 (p) = \tfrac{1}{2} (- \mathi \nabla_z + p)^2 + V (z), \quad H^2  (t) = - \mathi E (t) \cdot \nabla_p,
\end{equation}
and that $\Psi^0_m (t, z, p)$ is an eigenvector of $H^0 (p)$. Using the expansion of $H^\varepsilon$ leads to
\begin{align*}
   & \mathi \varepsilon \sum_m \frac{\partial}{\partial t} 
  \hat{\psi}_m^{\varepsilon} (t, p) \Psi^0_m (z, p)
  =  \sum_m \lambda_m (p)  \hat{\psi}_m^{\varepsilon} (t, p) \Psi^0_m (z,
  p) \\ &- \mathi \varepsilon^2 E (t) \cdot \sum_m (\nabla_p 
  \hat{\psi}_m^{\varepsilon} (t, p) \Psi^0_m (z, p) +
  \hat{\psi}_m^{\varepsilon} (t, p) \nabla_p \Psi^0_m (z, p)) .
\end{align*}
Take the inner product with $\Psi^{\varepsilon}_n$ and obtain
\begin{equation}
  \label{psi equation} \frac{\partial}{\partial t}  \hat{\psi}_n^{\varepsilon}
  (t, p) + \frac{\mathi}{\varepsilon} \lambda_n (p) 
  \hat{\psi}_n^{\varepsilon} (t, p) = - \varepsilon E (t) \cdot \nabla_p 
  \hat{\psi}_n^{\varepsilon} (t, p) - \mathi \varepsilon E (t) \cdot \sum_m
  A_{nm}  \hat{\psi}_m^{\varepsilon} (t, p) .
\end{equation}
Here, the vector $A_{nm}$, usually referred to as Berry connection, is defined
by
\begin{equation}
  \label{A def} A_{nm} = \langle \Psi^0_n, \mathi \nabla_p \Psi^0_m \rangle
  .
\end{equation}
Since
\begin{equation*}
    \mathi \nabla_p \langle \Psi^0_n,  \Psi^0_m \rangle =  \langle \Psi^0_n, \mathi \nabla_p \Psi^0_m \rangle -  \langle \mathi \nabla_p \Psi^0_n, \Psi^0_m \rangle = 0,
\end{equation*}
we know that
\[ A_{mn} = \bar{A}_{nm} . \]

\subsection{Density Matrix of Bloch Functions}
We define the density matrix of Bloch functions as 
\begin{equation}
  c_{mn} = \hat{\psi}_n^{\ast}  \hat{\psi}_m . \label{c def}
\end{equation}
Here, $\hat{\psi}^{\ast}$ means the conjugate of $\hat{\psi}$. Then we have
\begin{equation}
  \label{c equation} \frac{\partial}{\partial t} c_{mn} +
  \frac{\mathi}{\varepsilon} \lambda_{mn} (p) c_{mn} = - \varepsilon E (t)
  \cdot \nabla_p c_{mn} - \mathi \varepsilon E (t) \cdot \sum_p (A_{mp} c_{pn}
  - c_{mp} A_{pn}) .
\end{equation}
Here, $\lambda_{mn}$ is given by
\begin{equation}
  \lambda_{mn} = \lambda_m - \lambda_n . \label{lambda mn def}
\end{equation}
The initial data is given by 
\begin{equation}
  c_{mn}^{in} (p) = (\hat{\psi}^{in}_n (p) )^{\ast} \hat{\psi}_m^{in} (p) =
  \left\{\begin{array}{ll}
    | \hat{\psi}^{in}_0 (p) |^2,  & m = n = 0,\\
    0, & \tmop{otherwise} .
  \end{array}\right. \label{c id}
\end{equation}
When $m = n$, the second term of the right side of \eqref{c equation} contains
two terms $A_{nn} c_{nn}$ and $- c_{nn} A_{nn}$, which cancel with each other. Hereafter, we will concentrate on $c_{mn}$
instead of $\hat{\psi}_n$.

From \eqref{psi equation}, we can define the following characteristic line,
\begin{equation}
  \left\{\begin{array}{l}
    \dot{p} (t) = \varepsilon E (t),\\
    p (0) = p_0,
  \end{array}\right. \label{x_orbit}
\end{equation}
and let $P_t (p_0)$ denote its solution. It is obvious that
\begin{equation}
  P_t (p_0) = p_0 + \varepsilon \int_0^t E (s) d s. \label{P def}
\end{equation}
For a function $f$, let $\breve{f}$ denote
\begin{equation}
  \breve{f} (t, p) = f (t, P_t (p)), \label{alt u}
\end{equation}
Applying the method of characteristic line,
$\breve{c}_{m n}$ satisfies the following equation
\begin{equation}
  \frac{\partial}{\partial t}  \breve{c}_{m n} + \frac{\mathi}{\varepsilon} 
  \breve{\lambda}_{mn} (p)  \breve{c}_{m n} = - \mathi \varepsilon E (t)
  \cdummy \sum_p (\breve{A}_{mp}  \breve{c}_{p n} - \breve{c}_{mp}
  \breve{A}_{pn}) . \label{new c eq}
\end{equation}
with initial data still given by
\begin{equation}
  \breve{c}_{mn}^{in} (p) = \left\{\begin{array}{ll}
    | \hat{\psi}^{in}_0 (p) |^2,  & m = n = 0,\\
    0, & \tmop{otherwise} .
  \end{array}\right. \label{cu id}
\end{equation}

\subsection{Observables}

Our goal is to find an asymptotic expansion of electric current $J$ when $\varepsilon
\rightarrow 0$. In this subsection, we start with deriving an expression of $J$ with $c_{m n}$.

Here, we focus on some observables given by
\[ O (t) = \int_{R^d} \psi_r (t, z)^{\ast} \mathcal{O} \psi_r (t, z) dz, \]
with $\mathcal{O}$ denoting some self-adjoint operators. Here, recall that
$\psi_r$ is the rescaled wave function defined by \eqref{rescaling}. We will express $O$ using $c_{m n}$
defined by \eqref{c def}.

Since the Bloch transformation is an isometry, we have
\begin{align*}
  O (t) = \langle \psi_r , \mathcal{O} \psi_r  \rangle = \langle \mathfrak{B} \psi_r , \mathfrak{B} (\mathcal{O} \psi_r) 
  \rangle_{z, p} .
\end{align*}
Let
\[ \hat{\mathcal{O}} =\mathfrak{B} \circ \mathcal{O} \circ \mathfrak{B}^{- 1}
   . \]
Then it follows that
\begin{eqnarray*}
  O (t) & = & \langle \mathfrak{B} \psi_r , \hat{\mathcal{O}} \mathfrak{B}
  \psi_r  \rangle_{z, p}\\
  & = & \sum_{m, n} \langle \hat{\psi}^{\varepsilon}_m \Psi^0_m,
  \hat{\mathcal{O}}  (\hat{\psi}^{\varepsilon}_n \Psi^0_n)  \rangle_{z, p} .
\end{eqnarray*}
If we further assume that
\begin{equation}
  \hat{\mathcal{O}}  (\hat{\psi}^{\varepsilon}_n \Psi^0_n) =
  \hat{\psi}^{\varepsilon}_n  \hat{\mathcal{O}} \Psi^0_n, \label{O commutable}
\end{equation}
then we get
\[ O (t) = \int_{\Gamma^{\ast}} \sum_{m, n} c_{nm} O_{mn} dp, \]
with
\[ O_{mn} (p) = \langle \Psi^0_m, \hat{\mathcal{O}} \Psi^0_n \rangle (p) .
\]
\begin{remark}
  \label{rmk1}Two common examples are $\mathcal{O}f (z) = z f (z)$ and
  $\mathcal{O}f (z) = - \mathi \nabla f (z)$. The corresponding
  $\hat{\mathcal{O}}$ are given by $\hat{\mathcal{O}} f (z, p) = z f (z, p)$
  and $\hat{\mathcal{O}} f (z, p) = (- \mathi \nabla_z + p) f (z, p)$. (In the
  first case, $\hat{\mathcal{O}}$ is easy to know and we will deduce the
  expression of $\hat{\mathcal{O}}$ for the second case later). They both
  satisfy \eqref{O commutable}.
\end{remark}

Hereafter, we will concentrate on the case where $O = - \mathi \nabla_z$. In
this case, $O (t)$ is just the current $J$ defined by
\begin{equation}
  J (t) = - \mathi \varepsilon \int_{R^d} \psi (t, x)^{\ast} \nabla_x \psi (t,
  x) dx. \label{J def}
\end{equation}
In fact, we only need to notice that after rescaling, $J$ can be rewritten as
\begin{equation}
  J (t) = - \mathi \int_{R^d} \psi_r (t, z)^{\ast} \nabla_z \psi_r (t, z) dx.
  \label{J def rescaling}
\end{equation}
It is easy to see that
\begin{align*}
  \nabla_z  (\mathfrak{B} \psi_r ) & = \frac{1}{\lvert \Gamma^{\ast} \rvert^{1
  / 2}}  \sum_{Y \in \mathbb{L}} \nabla_z  (\psi_r  (z + Y) e^{- \mathi p
  \cdot (z + Y)})\\
  & = \frac{1}{\lvert \Gamma^{\ast} \rvert^{1 / 2}}  \sum_{Y \in \mathbb{L}}
  \nabla_z \psi_r   (z + Y) e^{- \mathi p \cdot (z + Y)} - \frac{1}{\lvert
  \Gamma^{\ast} \rvert^{1 / 2}}  \sum_{Y \in \mathbb{L}} \mathi p \psi_r   (z
  + Y) e^{- \mathi p \cdot (z + Y)}\\
  & =\mathfrak{B} (\nabla_z \psi_r ) - \mathi p\mathfrak{B} \psi_r .
\end{align*}
Therefore, we obtain 
\[ \hat{\mathcal{O}} = (- \mathi \nabla_z + p) . \]
Then $J (t) = (J_1(t), J_2(t), J_3(t))$ is given by
\begin{equation}
  \label{J exp} J_k (t) = \int_{\Gamma^{\ast}} \sum_{m, n} c_{nm} D^k_{mn} dp,\quad k=1,2,3.
\end{equation}
with
\begin{equation}
  \label{D def} D^k_{mn} (p) = \langle \Psi^0_m (p, z), (p_k \delta_{mn} -
  \mathi \partial_{z_k}) \Psi^0_n (p, z) \rangle .
\end{equation}
$D^k_{mn}$ are independent of
$\varepsilon$ and can also be written as
\begin{equation}
  D_{mn}(p) = \langle \Psi^0_m(p, \cdot), \nabla_p H^0 (p) \Psi^0_n(p, \cdot) \rangle  .
\end{equation}
Recalling \eqref{alt u}, if we concentrate on $\breve{c}_{m n}$ satisfying
\eqref{new c eq}, then the expression of $J$ is
\begin{equation}
  J_k (t) = \int_{\Gamma^{\ast}} \sum_{m, n} \breve{c}_{nm} \breve{D}^k_{mn}
  dp. \label{new J exp}
\end{equation}
Here, recall that $\breve{D}^k_{mn} $ is defined by \eqref{alt u}. For the convenience of computation, we collect some useful properties for $A$ and $D$ in Appendix \ref{A}.

%% file: Chapter_3.tex
\section{Weyl nodes and effective Hamiltonian}\label{2}

The major difficulty for the asymptotic expansion of $c_{m n}$ defined by
\eqref{c def} is the degeneracy of energy spectrum. Specifically, the degeneracy of energy gap leads to the singularity of Berry connections, which requires a multiscale analysis of the system. Therefore, in this part, we
will concentrate on crossing points of two bands. We will show that with some
assumptions, the degeneracy points are discrete points and we will also give
an approximation of eigenvalues and eigenvectors when $p$ is near the
degeneracy points. Despite the local structures at degenerated points, we can
still get a global expansion for the current $J$. Our strategy is to separate
the contribution from the local structures of degenerated points and the
contribution from the regular region. 

\subsection{Weyl nodes}

Suppose that $p_0$ is an isolated degenerate point for the $n$-th band, such that $\lambda_{mn} (p_0) = 0\, (m
> n)$. Such points
are called Weyl nodes. The Hamiltonian $H^0 (p_0)$ has a degenerate eigenvalue
$\lambda_n (p_0)$. Let $\Omega_0$ denote the corresponding eigenspace. Without loss of generality, we assume that $p_0 = 0$ and only focus on double-fold degeneracy. In
other words,
\begin{equation}
  \Omega_0 = \{ \phi \in L^2 (\Gamma ) : H^0 (p_0) \phi (x) = \lambda_n (p_0)
  \phi (x) \} = \ker (\lambda_n (p_0) I - H^0 (p_0)),\  \dim
(\Omega_0) = 2\label{Omega0}
\end{equation}
We further assume that $n=0,\,m=1$. Let $P_0$ denote the projection
operator from $L^2 \left( {\Gamma }^{\nobracket \nobracket} \right)$ onto
$\Omega_0$. We can choose a normalized orthogonal basis $\phi_1, \phi_2$ of
$\Omega_0$. The choice of the basis is arbitrary. Let
\begin{eqnarray}
  (H_i)_{jk} & = & \langle \phi_j, \partial_{p_i} H^0 (p) \phi_k \rangle,\, 1
  \leqslant j, k \leqslant 2,\, 1 \leqslant i \leqslant 3,  \label{w_i}\\
  (w_0)_i & = & (H_i)_{11} + (H_i)_{22},\, 1 \leqslant i
  \leqslant 3,  \label{w0}\\
  (w_3)_i & = & (H_i)_{11} - (H_i)_{22},\, 1 \leqslant i \leqslant 3, \\
  (w_1)_i & = & 2 \tmop{Re} ((H_i)_{21}),\, 1 \leqslant i \leqslant 3, \\
  (w_2)_i & = & 2 \tmop{Im} ((H_i)_{21}),\, 1 \leqslant i \leqslant 3, 
  \label{w2}\\
  W & = & (w_1, w_2, w_3)^T .  \label{W def}
\end{eqnarray}
Then each $H_i (1 \leqslant i \leqslant 3)$ is a $2 \times 2$ matrix and $W$
is a $3 \times 3$ matrix. For the neighborhood of degeneracy points, we have
the following lemma.

\begin{lemma}
  \label{eh lemma}Let $H = - \frac{1}{2} \Delta + V$. Here, the potential $V$ is a real-valued
  smooth function and periodic with respect to $\mathbb{L}$. Assume further
  that $\lambda_n (p_0)$ is a double-fold eigenvalue. Specificly, $\lambda_n
  (p_0) = \lambda_{n + 1} (p_0)$. Let $\lambda_+ (p)$ and $\lambda_- (p)$
  denote
  \begin{equation}
    \lambda_+ (p) \assign \lambda_{n + 1} (p) - \lambda_{n + 1} (p_0), \quad \lambda_- (p) \assign \lambda_n (p) - \lambda_n (p_0) .
    \label{eq:eigenvalue def}
  \end{equation}
  We further assume that $w_i, i = 1, 2, 3$ in \eqref{w_i} are
  linearly independent. 
  Then $H^0(p)$ acting on $L^2(\Gamma)$ has a dispersion surface which, in a neighborhood of
  $p_0$, is conical. That is, for $p$ near $p_0 = 0$, there are two distinct
  branches of eigenvalues of the Floquet-Bloch eigenvalue problem with
  quasi-momentum $p$:
  \begin{equation}
    \lambda_{\pm} (p) = w_0 \cdummy (p) \pm |W p| + o (p) . \label{lambda
    asym}
  \end{equation}
  $w_0$ is a vector given by \eqref{w_i} and $W$ is a matrix given by
  \eqref{W def}.
\end{lemma}
\begin{remark}
  Generally, we call this kind of degenerate points ``Weyl nodes''. If
  $\mathrm{sgn} (\det (W)) = 1$, we call this Weyl node ``right handed'' Weyl
  nodes. If $\mathrm{sgn} (\det (W)) = - 1$, we call this Weyl node ``left
  handed'' Weyl nodes. 
\end{remark}

To prove Lemma \ref{eh lemma}, we need to determine for $p$ near $p_0 = 0$ the eigenvalues
$\lambda_{\pm} (p)$ and the corresponding eigenvectors $\Psi_{\pm} (p, x)$
defined by
\begin{equation}
  H (p) \Psi_{\pm} (p,x) = \lambda_{\pm} (p) \Psi_{\pm} (p,x) . \label{eq:eigen
  vector def}
\end{equation}
We can use the method in {\cite{fefferman_honeycomb_2012}} and get two
equations \eqref{det M} and \eqref{M vec} which determine $\lambda_{\pm}$ and
$\Psi_{\pm}$. The details are in Appendix \ref{Weyl nodes}. Furthermore, we
can also obtain the asymptotic expansion of the eigenvectors $\Psi_{\pm}$ in a
small neighborhood around $p_0$; the proof is also deferred to  Appendix \ref{Weyl
nodes}.

\begin{corollary}
  \label{eh cor}Recall that $\phi_1, \phi_2$ are the basis of $\Omega_0$
  defined by \eqref{Omega0}. Under the same assumption of Lemma \ref{eh
  lemma}, we can decompose the eigenvectors into two parts:
  \[ \Psi_{\pm} (p, x) = \Psi^0_{\pm} (p, x) + \Psi^1_{\pm} (p, x), \]
  with
  \begin{align*}
    \Psi^0_{\pm} (p, x) & =  \alpha_{\pm} (p) \phi_1 (x) + \beta_{\pm} (p)
    \phi_2 (x) \in \Omega_0,\\
    \Psi^1_{\pm} (p, x) & \in \Omega_0^{\bot} .
  \end{align*}
  Then the asymptotic expansions of $\alpha_{\pm}, \beta_{\pm}$ are given by
  \begin{eqnarray*}
    \alpha_{\pm} (p) & = & \alpha_{\pm}^0 (p) + O (| p |),\\
    \beta_{\pm} (p) & = & \beta_{\pm}^0 (p) + O (| p |) .
  \end{eqnarray*}
  Here, $(\alpha_{\pm}^0, \beta_{\pm}^0)$ solves the equation
  \begin{equation}
    \left( \begin{array}{cc}
      \lambda_{\pm} + \langle \phi_1, 2 \mathi p \cdot \nabla \phi_1 \rangle &
      \langle \phi_1, 2 \mathi p \cdot \nabla \phi_2 \rangle\\
      \langle \phi_2, 2 \mathi p \cdot \nabla \phi_1 \rangle & \lambda_{\pm} +
      \langle \phi_2, 2 \mathi p \cdot \nabla \phi_2 \rangle
    \end{array} \right) \left(\begin{array}{c}
      \alpha_{\pm}^0\\
      \beta_{\pm}^0
    \end{array}\right) = 0. \label{leading order psi0}
  \end{equation}
  And the asymptotic expansions of $\Psi^1_{\pm}$ is
  \begin{equation}
      \label{eq: eig vec asym}
  \Psi^1_{\pm} (p, x) = \alpha_{\pm}^0 (p) \psi^{1, 0}_{\pm} (p, x) +
     \beta_{\pm}^0 (p) \psi^{2, 0}_{\pm} (p, x) + O (| p |^2),
  \end{equation}
  with
  \[ \psi^{i, 0}_{\pm} = (\lambda_n - H (p_0))^{- 1}  (I - P_0)  (2 \mathi p
     \cdot \nabla) \phi_i, \text{ } i = 1, 2. \]
  Here, $(\lambda_n - H (p_0))^{- 1}$ is a bounded operator defined on $\ker
  (H (p_0) - \lambda_n I)^{\perp}$.
\end{corollary}

We further assume that there are only finite Weyl nodes in each band. Then this result implies that for all $\varepsilon$ small
enough, $\{ \lambda_{n + 1} (p) - \lambda_n (p) < \varepsilon \}$ consists of
several disjoint neighborhoods of all Weyl nodes $p_i$. In other words, there
exists $\varepsilon_0 > 0$, such that for any $0 < \varepsilon < \varepsilon_0$, there exists $\delta 
> 0$, such that
\begin{equation}
  \{ \lambda_{n + 1} (p) - \lambda_n (p) < \varepsilon \} \subset \bigcup_i B
  (p_i, \delta) \infixand B (p_i, \delta) \bigcap B (p_j, \delta) =
  \varnothing,\, \forall i \neq j. \label{wn}
\end{equation}

\subsection{Effective Hamiltonian}\label{effective hamiltonian}

Lemma \ref{eh lemma} gives the asymptotic expansion of the eigenvalues
$\lambda_{\pm} (p)$ near $p_0$. In Corollary \ref{eh cor}, we have decomposed
$\Psi_{\pm}$ into two parts. When $p$ is small, $\Psi^1_{\pm}$ is smaller than
$\Psi^0_{\pm}$ and thus usually neglibile. To be more clear, $H^0 (p)$ can be
approximated by $H^e (p)$ which satisfies
\begin{equation}
  H^e (p) \Psi^0_{\pm} (p, x) = \lambda_{\pm} (p) \Psi^0_{\pm} (p, x), \text{ }
  \Psi^0_{\pm} \in \Omega_0,
\end{equation}
Notice that the leading order term of $\lambda_{\pm}$ and $\Psi^0_{\pm} (p, x)$
is given by \eqref{lambda asym} and \eqref{leading order psi0}, we know that
\begin{equation}
  H^e (p) = \left( \begin{array}{cc}
    \langle \phi_1, - 2 \mathi p \cdot \nabla \phi_1 \rangle & \langle \phi_1,
    - 2 \mathi p \cdot \nabla \phi_2 \rangle\\
    \langle \phi_2, - 2 \mathi p \cdot \nabla \phi_1 \rangle & \langle \phi_2,
    - 2 \mathi p \cdot \nabla \phi_2 \rangle
  \end{array} \right) + O (| p |^2) . \label{eh exp0}
\end{equation}
Thus we have constructed an an asymptotic expression for the Hamiltonian near
a degeneracy point. This is often referred to as the \emph{effective Hamiltonian.}

Although $\lambda_{\pm} (p)$ defined by \eqref{eq:eigenvalue def} are not
smooth and generally $\Psi_{\pm} (p, x)$ defined by \eqref{eq:eigen vector def}
are not continuous with respect to $p$ at $p_0$, we can relate them to the
eigenvalues and eigenvectors of a smooth operator function $H^e (p)$ on
$\Omega_0$, which makes $\lambda_{\pm} (p)$ and $\Psi_{\pm} (p, x)$ computable in
the asymptotic sense. Specificly, we construct a Hamiltonian $H^e (p)$ on \
$\Omega_0$ such that $\lambda_{\pm}$ are the eigenvalues of $H^e (p)$.

On the other hand, for the eigenfunctions $\Psi_{\pm} (p, x)$, we need to
construct another operator $U (p)$ to map the eigenvectors of $H^e (p)$ to the
eigenvectors of $H^0 (p)$. As we mentioned, $\Psi_{\pm} (p, x)$ are not
continuous with respect to $p$ at $p_0$. But since they are very close to
$\Psi_{\pm}^0 (p)$ the eigenvectors of $H^e (p)$, the operator $U$ is
continuous. In other words, $\Psi_{\pm}^0 (p)$ captures the discontinuity of
$\Psi_{\pm} (p, x)$. We put out the idea with the following equations
\begin{align}
    H^e (p) \Psi^0_{\pm} (p, x) &= \lambda_{\pm} (p) \Psi^0_{\pm} (p, x),\\
    U (p) \Psi^0_{\pm} (p, x) &= \Psi_{\pm} (p, x) .
\end{align}
For the eigenvectors $\Psi_{\pm} (p, x)$ and $\Psi^0_{\pm} (p, x)$, they are
unique up to a complex constant. We can determine this constant by requiring
that
\begin{eqnarray*}
  \langle \Psi^0_{\pm} (p, \cdot), \phi_1 \rangle & > & 0,\\
  \| \Psi^0_{\pm} (p, \cdot) \|_{L^2_x (\Gamma)} & = & 1,\\
  P_0 \Psi_{\pm} & = & \Psi_{\pm}^0 .
\end{eqnarray*}
The first two conditions determine $\Psi^0_{\pm} (p, x)$ uniquely and the third
condition determines $\Psi_{\pm} (p, x)$ uniquely. Then $U (p)$ is also uniquely
determined. Thanks to Corollary \ref{eh cor}, we obtain the asymptotic
expansion of $U (p)$:
\begin{equation}
  U (p) = (\lambda_n - H (p_0))^{- 1}  (I - P_0)  (2 \mathi p \cdot \nabla)
  P_0 + O (| p |^2) . \label{u exp0}
\end{equation}
For the effective Hamiltionian, more details can be found in
{\cite{jorgensen_effective_1975,bloch_sur_1958}}. We use the results directly
here. Since $p$ is small, we can expand $H^e$ and $U (p)$ to the first order
and get
\begin{align}
  H^e (p) & = P_0  (p \cdot \nabla_p H^0 (p_0)) P_0 + O (| p |^2), \\
  U (p) & = P_0 + (\lambda_n - H^0)^{- 1} (I
  - P_0) (p \cdummy \nabla_p H^0 (p_0)) P_0 + O (| p |^2) .
\end{align}

Notice that this result is consistent with \eqref{eh exp0} and \eqref{u exp0}.
{\cite{jorgensen_effective_1975}} also shows that the high order terms are
bounded. We discard the high order terms of $O(|p|^2)$ without changing the notation. We
should notice that $H^e$ is the restriction of $H^0$ on $\Omega_0$. Since
$\Omega_0$ is a two-dimensional subspace, $H^e (p)$ can be expressed as a $2
\times 2$ matrix, straightforward calculations show that
\begin{equation}
  H^e (p) = \sum_{i = 1}^3 p_i H_i, \label{eff ham0}
\end{equation}
with $H_i$ given by \eqref{w_i}. Using Pauli matrix, we can rewrite $H^e$ as
\begin{equation}
  H^e (p) = (p \cdot w_0) I + \sum_{i = 1}^3 (p \cdot w_i) \sigma_i .
  \label{eff ham}
\end{equation}
Here, $I$ is the identity and $\sigma_i$ is the Pauli matrix. $w_i$ is given
by \eqref{w0} to \eqref{w2}. Remember that we assume that the vectors $w_i \, (i
= 1, 2, 3)$ are linear independent. In this way, we can guarantee that the two
eigenvalues of the matrix are different.

The eigenvalues of $H^e (p)$ are given by \eqref{lambda asym}. $\Psi^0_{\pm}$
are the eigenvectors of $H^e$, and finally $\Psi_{\pm}$ can be obtained by the
expansion of $U (p)$. We should notice that we discard high order terms here.
Using the expression of $\Psi^0_{\pm}$, we can obtain some useful results
for the following calculation of $J$. They are stated in the following
Proposition.

\begin{proposition}
  \label{eff result}Under the same assumptions as Lemma \ref{eh lemma}, for $p$ near $p_0=0$, we have the following
  approximations,
  \begin{eqnarray}
    \lambda_+ - \lambda_- & = & | W p | + o (| p |),  \label{gap}\\
    \Delta_{+ -} & = & \nabla_p  (\lambda_+ - \lambda_-) = \frac{W^T W p}{| W
    p |} + o (| p |),  \label{dlambda asymptotic}\\
    \int_{\mathbb{S}^2} \omega^j D^k_{+ -} (r \omega) D^l_{- +} (r \omega) d
    \omega & = & \frac{\mathi \pi}{3} \delta^{jkl} \tmop{sgn} (| \det W |) + o
    (| p |),\quad \forall r > 0,  \label{delta jkl}\\
    \int_{\mathbb{S}^2} \tilde{D}^k_{- -} (r \omega) d \omega & = & o (| p |),\quad
    \forall r > 0,  \label{D--}\\
    \int_{\mathbb{S }^2_W} \tilde{D}^k_{+ -} (r \omega)  \tilde{D}^l_{- +} (r
    \omega) d \omega & = & \frac{2 \pi}{3} \sum_j  \frac{W_{k j} W_{l j}}{|
    \det W |} + o (| p |),\quad \forall r > 0,  \label{delta kl}
  \end{eqnarray}
  $w_0$ is a vector given by \eqref{w_i} and $W$ is a matrix given by
  \eqref{W def}.
\end{proposition}

We postpone the proof of this Proposition in Appendix \ref{2.3}.

\begin{remark}
  By Lemma \ref{eh lemma} and Corollary \ref{eh cor}, we have shown that the
  high order terms of $\Psi_{\pm} (p, x)$ and $\lambda_{\pm} (p)$ are bounded,
  which is sufficient to justify all calculations above. We also refer to {\cite{bloch_sur_1958}} for the estimate for the remainders. 
\end{remark}

%% file: Chapter_4.tex
\section{Main results}\label{3}
In this section, we present our main result, which is the asymptotic expansion for the current $J$. 
Before stating the main result, we will first introduce some regularity
assumptions.

\begin{question}
  \label{regular asump}
  
  (a) There are only finite Weyl nodes on each energy band. Only the $0$-th band and the $1$-st band have crossings. More
  precisely, we assume that
  \[ \forall m \neq n, m + n \geqslant 2, | \lambda_{m n} (p) | \geqslant C,
  \]
  for some constant $C > 0$. This also means that all degenerate points are
  double-fold.
  
  (b) We assume that
  \begin{eqnarray}
    & \lambda_n (p) \in C^{\infty} (\Gamma^{\ast} \backslash \{ p_j^0 \}),
    \Psi_n (x, p) \in C^{\infty} (\Gamma \times (\Gamma^{\ast} \backslash \{
    p_j^0 \})), \forall n \in \{ 0, 1 \} .  \label{n smooth}\\
    & \lambda_n (p) \in C^{\infty} (\Gamma^{\ast}), \Psi_n (x, p) \in
    C^{\infty} (\Gamma \times \Gamma^{\ast}), \forall n \geqslant 2. 
  \end{eqnarray}
  Here, $\{ p_j^0 \}$ are the set of the Weyl nodes of the 0th band and the
  1st band. Then it is obvious that
  \begin{eqnarray}
    & \lambda_{m n}, D_{m n}, A_{m n} \in C^{\infty} (\Gamma^{\ast}
    \backslash \{ p_j^0 \}), \forall m \in \{ 0, 1 \} \infixor n \in \{ 0, 1
    \},  \label{smooth}\\
    & \lambda_{m n}, D_{m n}, A_{m n} \in C^{\infty} (\Gamma^{\ast}
    \backslash \{ p_j^0 \}), \forall m, n \geqslant 2. \nonumber
  \end{eqnarray}
  
  (c) For each Weyl point $p_j^0$, let $W_j$ denote the matrix given by Lemma
  \ref{eh lemma} near $p_j^0$. Besides, let $A_{j, 10}$ denote $\tilde{A}_{+
  -}$ given by \eqref{rule1} in subsection \ref{2.3} to approximate $A_{10}$
  with $W = W_j$. $D_{j, m n}, A_{j, m n}, \Delta_{j, m n}, \Lambda_{j, m n}$
  with $m, n \in \{ 0, 1 \}$ is similarly defined. $\lambda_{j, m}$ with $m, n
  \in \{ 0, 1 \}$ should be given by \eqref{rule2} and $\lambda_{j, m n} = \lambda_{j, m} - \lambda_{j, n}$.
  
  Thanks to the effective Hamiltonian method, we use $D_{j, m n}, A_{j, m n}$
  and $\lambda_{j, m n}$ to approximate $D_{m n}, A_{m n}$ and $\lambda_{m n}$
  when $p$ is close to $p_j^0$. Since $D_{j, m n}, A_{j, m n}$ and
  $\lambda_{j, m n}$ can be rewritten as the following form,
  \begin{eqnarray*}
    D_{j, m n} (r \omega) & = & F (\omega),\\
    A_{j, m n} (r \omega) & = & G (\omega) / r,\\
    \lambda_{j, m n} (r \omega) & = & H (\omega) r.
  \end{eqnarray*}
  Here, $\omega \in \mathbb{S}^2$ and $r > 0$ and $F, G, H$ denote some smooth
  functions. Considering the expression of $D_{j, m n}, A_{j, m n}$ and
  $\lambda_{j, m n}$, we make the following regularity assumption for $D_{m
  n}, A_{m n}$ and $\lambda_{m n}$,
  \begin{eqnarray*}
    \| D_{m n} (p) \|_{W^{l, \infty}} & \leqslant & C / | p - p_j^0 |^l,\\
    \| \lambda_{m n} (p) \|_{W^{l, \infty}} & \leqslant & C / | p - p_j^0 |^{l
    - 1},\\
    \| A_{m n} (p) \|_{W^{l, \infty}} & \leqslant & C / | p - p_j^0 |^{l + 1},
    \forall m, n \in \{ 0, 1 \}, l \in \mathbb{N}, p \in B (p_j^0, r_2) .
  \end{eqnarray*}
  with some constant $C$ independent of $p$ and $\varepsilon$.
  
  (d) For every critical point of $\lambda_{10}(p)$, say $p_0$, we assume that
  \[ \det (\nabla^2 \lambda_{10} (p_0)) \neq 0. \]
  This assumption will allow us to use the stationary phase method.

  (e) The initial data is given by \eqref{id ansatz} with smooth $\hat{\psi}_0^{in} (p)$.
\end{question}

In fact, the first assumption is not necessary (see Remark \ref{weaker
assumption}), but it can make our analysis simpler. We will see later that only
the 0-th band and the 1-st band are important because for our initial data,
$c_{0 0}$ is the only nonzero term.

\medskip 

Let us now state our main results.

\begin{theorem}
  \label{J asym exp}For the semiclassical Schr{\"o}dinger equation
  \eqref{Schrodinger equation}, we can transform this system into equation
  \eqref{c equation} of the Bloch density matrix $\{c_{m n}\}$.
  Then the current $J (t)$ is given by \eqref{J exp} in terms of the Bloch density matrix. Assume that 
  Assumption \ref{regular asump} holds. In the semiclassical regime $\varepsilon \ll 1$, there exists a time $T$ independent of $\varepsilon$ such that for any $0<t<T$, the
  current $J$ has the following asymptotic expansion,
  \[ J_k (t) = J_{k, 1} (t) + J_{k, 2} (t) + o (\varepsilon^3) . \]
  Here, $J_{k, 1} (t)$ is given by
  \begin{equation}\label{J1 def}
      \begin{split}
          J_{k, 1} (t) = & \int_{\Gamma^{\ast}} D^k_{0 0} (p) c_{00}^{i n} (p) dp - \varepsilon
    \int_0^t E (s) d s \cdummy \int_{\Gamma^{\ast}} \nabla c_{00}^{i n} (p)
    D^k_{0 0} (p) dp \\
    + & \frac{\varepsilon^2}{2} \int_0^t E (s) \otimes E (s) d s :
    \int_{\Gamma^{\ast}} \nabla^2 c_{00}^{i n} (p) D^k_{0 0} (p) d p
    \\
    + & 2 \varepsilon^2 \sum_{n \geqslant 1} \int_{\Gamma^{\ast}}
    \frac{c_{00}^{i n} (p) E (t) \cdummy \tmop{Im} (D_{0 n}  (p) D^k_{n 0} 
    (p))}{\lambda_{n 0} (p)^2} dp \\ - & 2 \varepsilon^3 \sum_{n \geqslant 2}
    \int_{\Gamma^{\ast}} c_{00}^{i n} (p) \frac{E' (t) \cdummy \tmop{Re} (D_{0
    n}  (p) D^k_{n 0} (p))}{\lambda_{n 0} (p)^3} dp \\
    + & 2 \varepsilon^3 E' (t) \cdummy \int_{\Gamma^{\ast}} \tmop{Re}\left(
    \frac{ D_{01} D^k_{1 0}  (p)}{\lambda_{10} (p)^3} c_{00}^{i n} (p) -
    \sum_j I_{B_{W_j} (p_j^0, r_2)} (p) c_{00}^{i n} (p_j^0) \frac{D_{j, 01} 
    D^k_{j, 10}  (p - p_j^0)}{| W_j (p - p_j^0) |^3} \right) dp,
      \end{split}
  \end{equation}
  and $J_{k, 2} (t)$ is given by
  \begin{equation}
    J_{k, 2} (t) = \varepsilon^3 \int_{\mathbb{R}^3} \sum_{m, n \in \{ 0, 1
    \}} \left( D^k_{j, n m} s_{m n}^j - c_{00}^{i n} (p_j^0) I_{| W_j
    \tilde{p}_j | > \frac{r_2}{\varepsilon}} \frac{E' (t) \cdummy D_{j, m n}
    D^k_{j, n m}}{| W_j \tilde{p}_j |^3} \right) d \tilde{p}_j . \label{J2
    def}
  \end{equation}
  In the above, recall $W_j$ and $D_{j, m n}$ are defined in Assumption \ref{regular asump}.
  $D_{m n}$ and $\lambda_{m n}$ are defined by \eqref{D def} and \eqref{lambda
  mn def}. While $s_{j, m n}$ is given by \eqref{limit system}.
\end{theorem}

We elaborate in the following remark about the result, especially the physical meaning of the two terms $J_{k, 1} (t)$ and $J_{k, 2}(t)$.

\begin{remark} \label{rm5}
Here we divide the current $J$ into two parts, $J_{k, 1} (t)$ and $J_{k, 2}
  (t)$.
\begin{itemize}
    \item The expression of $J_{k, 1} (t)$ only contains $D_{0 n}  (p)$ and
  $\lambda_{n 0} (p)$ except the last term. They are functions on the whole
  Brillouin zone. These terms still exists if there are no Weyl nodes. Generally, $J_1(t)$ is $O(1)$ and thus the leading order term. For special cases, e.g. if the initial data $c_{00}$ is a constant, then only the last two terms of $J_{k, 1}$ will survive. In this situation, $J_{k, 1}$ is $O(\varepsilon^3)$ and smaller than $J_{k,2}$.

    \item The expression of $J_{k, 2} (t)$ completely comes from the local structure
  of Weyl nodes. We should notice that if we ignore the second term, it is
  just the current related to the limit system \eqref{limit system}, which we
  will introduce later. Generally, $J_{k,2}$ is $O(\varepsilon^3 |\log \varepsilon|)$ and hence smaller than $J_{k,1}$. But in special case that $c_{00}$ is a constant, $J_{k,2}$ is the leading order term. 

  \item The last term of $J_{k, 1} (t)$ depends on both functions on the whole
Brillouin zone ($D_{0 n}  (p)$ and $\lambda_{n 0} (p)$) and the properties of
each Weyl nodes ($D_{j, 10} (\tilde{p}^j) $). This term is the result of the singularity nature of Weyl nodes. We will explain more about this after all the technical details are presented.
\end{itemize}
\end{remark}

\begin{remark}
   This result is closely associated with the result in
{\cite{de_juan_quantized_2017}}. In {\cite{de_juan_quantized_2017}}, the
author directly consider the limit system \eqref{limit system}, and the
current is given by $J_{k,2}$ in \eqref{J exp final}. Therefore, our
result can be viewed as a rigorous justification of the result in
{\cite{de_juan_quantized_2017}}.

However, there are some significant differences between this result and the
result in {\cite{de_juan_quantized_2017}}. First of all, in
{\cite{de_juan_quantized_2017}}, the author used some conclusions in
{\cite{sipe_second-order_2000}}. The setting is to start from $t = - \infty$.
However, here we consider the Cauchy initial data problem and start from $t =
0$. Second, in {\cite{de_juan_quantized_2017}}, the author only impose a small
parameter before the electric field while in our setting, the small parameter
stands for the semiclassical regime. Therefore, to completely justify the
result in {\cite{de_juan_quantized_2017}}, a possible way is to add a small
parameter before $E$ in the limit system \eqref{limit system}, consider the
asymptotic expansion of $J_{2, k}$ and pass to the limit $t \rightarrow +
\infty$.

\end{remark}

\section{Multiscale expansion}
The proof of the main result involves several key steps. In this section, we proceed with calculating the asymptotic expansion of $c_{m n}$. The rest of the error estimate will be completed in Section \ref{4}. 

Instead of $c_{mn}$, we will focus on $\breve{c}_{m n}$ defined by \eqref{alt u} and
\eqref{c def}, which satisfies \eqref{new c eq}. One advantage of focusing on
$\breve{c}_{m n}$ is that \eqref{new c eq} is a infinite-dimension ODE with
parameter $p$. We can get pointwise estimate like
\[ | c_{m n} (t, p) | \leqslant F (p), \]
for some nonnegative function $F (p)$.

As mentioned before, the degeneracy of energy gap makes the problem difficult.
By Duhamel's formula, we can rewrite \eqref{new c eq} in the form of integral
as
\begin{equation}
  \breve{c}_{mn} = c_{mn}^{i n} - \mathi \varepsilon \int^t_0
  {e^{\frac{\mathi}{\varepsilon}  \int^t_s \breve{\lambda}_{mn} (r, p) d r}} 
  E (s) \cdummy \sum_p (\breve{A}_{mp}  \breve{c}_{pn} - \breve{c}_{mp} 
  \breve{A}_{pn}) (s, p) d s. \label{int eq} 
\end{equation}
When $\lambda_{m n} = O (1)$, the integral of the right side is oscillatory
and we can use the stationary phase method. But when $\lambda_{m n} = O
(\varepsilon)$, the term $\frac{\lambda_{m n}}{\varepsilon}$ is $O (1)$
instead of $O (1 / \varepsilon)$. In this situation, we seek a more comprehensive and global analysis.
Besides, when $\lambda_{m n} = O (\varepsilon)$, the Berry connection $A_{m n}$
might be $O \left( \frac{1}{\varepsilon} \right)$ because of degeneracy (see \eqref{AD
relation}), which makes the analysis more difficult.

To solve this problem, we will first divide the Brillouin zone into three
parts, which will be introduced in \S\ref{3.1}. Different strategies are used for different regions, which will be elaborated in  \S\ref{3.3} and \S\ref{3.6}.

\subsection{Division of the Brillouin zone and our strategies}\label{3.1}
Now we introduce our overall strategy of the proof. We first establish  the asymptotic
expansion of $J$ in this section. The second step will be carried out in Section
\ref{4}, where we will also give more explanation about $J_{k, 1} (t)$ and
$J_{k, 2} (t)$.

As we mentioned before, we proposed to resolve the problem of singularity by treating the cases $\lambda_{m n} = O (1)$ and $\lambda_{m
n} = O (\varepsilon)$ separately. More specifically, we divide the Brillouin zone $\Gamma^{\ast}$
into three parts, $\Gamma_s, \Gamma_m$ and $\Gamma_r$, as the following
picture shows.\begin{figure}[thb]
  {\includegraphics[width=12cm,height=5cm]{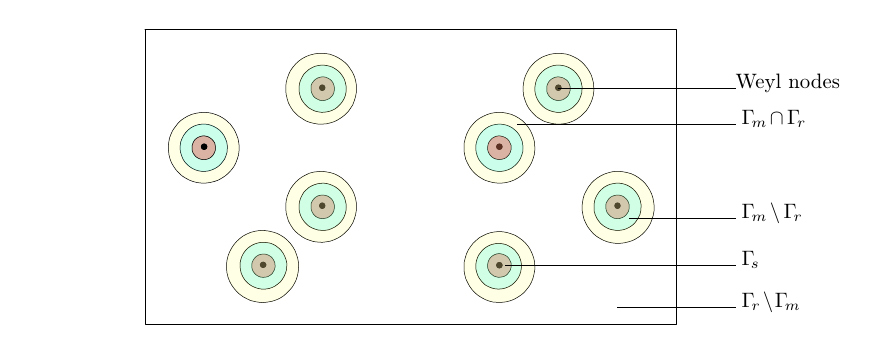}}
  
  \
  \caption{Division of Brillouin zone. The red region is $\Gamma_s$, the
  green and yellow regions are $\Gamma_m$ and the yellow and white regions are
  $\Gamma_r$.}
\end{figure}

{\tmem{1. Definition of $\Gamma_s$:}} First, we choose a $r_1$ independent of
$\varepsilon$ and let $\Gamma_s$ be

\begin{equation}
  \Gamma_s = \bigcup^N_{j = 1} B_{W_j} (p_j^0, \varepsilon r_1) . \label{Gs}
\end{equation}

Here, $B_{W_j} (p_j^0, \varepsilon r_1)$ is defined by \eqref{Bw def}. Notice
that $W_j$ is invertible so $B_{W_j} (p_j^0, r)$ is an ellipsoid. $r_1$ should
be large enough such that for any $0< t < T$,
\begin{equation}
  \forall p \nin \Gamma_s, | P_t (p) - p_j^0 | > \varepsilon d, \label{small
  dist}
\end{equation}
for some constant $d > 0$ independent of $\varepsilon$. Here, since
\[ \breve{\lambda}_{10} (t, p) = \lambda_{10}  (P_t (p)), \]
we can assume that
\[ \breve{\lambda}_{10} (t, p) > \varepsilon r_0,\quad \forall 0<t<T,\quad p \nin
   \Gamma_s, \]
for some constant $r_0 > 0$ independent of $\varepsilon$.

{\tmem{2. Definition of $\Gamma_m$:}} Then we choose a $r_2 < C$ independent
of $\varepsilon$ such that in $B_{W_j}  (p_i, r_2) \setminus \{p_i \}$ we have
\begin{equation}
  \label{delta condition} C_i / 2 < | \partial_r \lambda_{10} (p_j^0 + r
  \omega) | < 2 C_i,\quad \forall \omega \in \mathbb{S }^2_{W_j},\quad 0 < r < r_2 .
\end{equation}
Here, $\mathbb{S }^2_{W_j}$ is defined by $\mathbb{S }^2_{W_j} = \partial
B_{W_j}  (0, 1)$. Let the second part $\Gamma_m$ be
\begin{equation}
  \Gamma_m = \bigcup^N_{j = 1} (B_{W_j} (p_j^0, r_2) \backslash B_{W_j}
  (p_j^0, \varepsilon r_1)) . \label{Gm}
\end{equation}
Then we can set $q_i = \lambda_{10}  (p_i + r \omega)$ with the inverse
function $r = r_i (q_i, \omega)$ thanks to \eqref{delta condition}.

\begin{remark}
  We choose $\Gamma_m$ and $\Gamma_s$ in this way because for the following
  integral
  \[ \int_{B_{W_i} (p_i, r_2)} f (\lambda_{10} (p), p) d p , \]
  we can change the variable of integral and get
  \begin{eqnarray}
    \label{change
    vari} &\int_{B_{W_j}  (p_i, r_2)} f (\lambda_{10} (p), p) d p = 
    \int_{\mathbb{S }^2_{W_j}} \int_0^{r_2} f (\lambda_{10} (p_i + r \omega),
    p_i + r \omega) r^2 d r d \omega \\
     =& \int_{\mathbb{S }^2_{W_j}} \int_0^{\lambda_{10}  (p_i + r_2
    \omega)} f (q, p_i + r_i (q, \omega) \omega) r_i (q, \omega)^2
    \frac{\partial}{\partial q} r_i (q, \omega) d q d \omega .  \nonumber
  \end{eqnarray}
\end{remark}

{\tmem{3. Definition of $\Gamma_r$:}} Finally, let the last part $\Gamma_r$ be
complementary set of $\Gamma_m$ and $\Gamma_s$. In other words, let $\Gamma_r$
denote
\begin{equation}
  \Gamma_r = \Gamma^{\ast} \backslash \bigcup^N_{j = 1} B_{W_j} \left( p_j^0,
  \frac{r_2}{2} \right) . \label{Gr}
\end{equation}
Notice that $\Gamma_m \bigcap \Gamma_r \neq \varnothing$. The reason is that
we need to use the partition of unity of $\Gamma^{\ast}$. Let $\chi_1
\infixand \chi_2$ be two smooth nonnegative function defined on $\Gamma^\ast$, such that
\begin{eqnarray*}
  \chi_1 + \chi_2 & = & 1,\\
  \chi_1 & \equiv & 0 \quad \tmop{on} \quad \Gamma_r^c,\\
  \chi_2 & \equiv & 0 \quad \tmop{on} \quad \left( \Gamma_m \bigcup \Gamma_s
  \right)^c .
\end{eqnarray*}
Here $\Gamma_r^c$ means the complementary set of $\Gamma_r$.

\begin{remark}
  Here is a summary about the order of $\lambda_{m n}$ on different regions.
  For all $\lambda_{m n}$ except $\lambda_{01}$ and $\lambda_{10}$, recalling Assumption \ref{regular asump}, we know that
  \[ | \lambda_{m n} | > C'_0,\quad \forall m \neq n \infixand m + n > 1. \]
  For $\lambda_{10}$, we have
  \begin{align*}
    | \lambda_{1 0} | > & C'_0 \quad \tmop{on} \quad \Gamma_r,\\
    \varepsilon C'_2 < | \lambda_{1 0} | < & C'_1 \quad \tmop{on} \quad \Gamma_m,\\
    | \lambda_{1 0} | < & \varepsilon C'_2 \quad \tmop{on} \quad \Gamma_s,
  \end{align*}
  for some constant $C'_0, C'_1$ and $C'_2$.
\end{remark}

\begin{remark}
    The choice of $r_1$ and $r_2$ has no influence on the final result. The asymptotic expression of $J$ given by \eqref{J1 def} and \eqref{J2 def} is independent of the division of the
Brillouin zone. The final expression does not contains $r_1$ so it is
independent of $r_1$. On the other hand, if we take the derivative with
respect to $r_2$, all $r_2$ dependent terms will cancel. Therefore, the expression is also
independent of $r_2$.
\end{remark}

Now let us elaborate on our strategy. It mainly contains two ingredients, which we
will introduce in Section \ref{3.3} and Section \ref{3.6} separately. The first ingredient is
the following expansion, which we will apply to $\Gamma_r$ and $\Gamma_m$,
\begin{equation}
  \breve{c}_{m n} = c_{m n}^0 + \frac{\varepsilon}{\lambda_{10}} c_{m n}^1 +
  \left( \frac{\varepsilon}{\lambda_{10}} \right)^2 c_{m n}^2 + \left(
  \frac{\varepsilon}{\lambda_{10}} \right)^3 c_{m n}^3 + \left(
  \frac{\varepsilon}{\lambda_{10}} \right)^4 \varsigma_{m n}^4 . \label{c asym
  ori}
\end{equation}
Here we use $\frac{\varepsilon}{\lambda_{10}}$, which is $O (\varepsilon)$ on
$\Gamma_r$ and varies from $O (\varepsilon)$ to $O (1)$ on $\Gamma_m$, instead
of $\varepsilon$ as the parameter of asymptotic expansion. In this way, we can
get a uniform bound for the remainders. We will show that $c_{m n}^k (0
\leqslant k \leqslant 3)$ and $\varsigma_{m n}^4$ satisfy
\begin{eqnarray}
  c_{m n}^k & \leqslant & C, \nonumber\\
  \varsigma_{m n}^4 & \leqslant & C,  \label{Rmn4}
\end{eqnarray}
uniformly on $\Gamma_r \cup \Gamma_m$. Notice that
\begin{equation}
  \left| \frac{\lambda_{10} (p)}{\breve{\lambda}_{m n} (t_1, p)} \right|
  \leqslant C, \forall t_1 \in \mathbb{R}_+, m, n \in \mathbb{N}.
  \label{uniform bound}
\end{equation}
This means that if a term is $O \left( \frac{\varepsilon}{\breve{\lambda}_{m
n}} \right)$, it is also $O \left( \frac{\varepsilon}{\lambda_{10}} \right)$.
\ This part will be shown in \ref{3.3}.

The second ingredient is rescaling, which we will apply to $\Gamma_m \cup
\Gamma_s$. We need to use the rescaling
\begin{equation}
  \tilde{p}_j = \frac{(p - p_j^0)}{\varepsilon}, \label{rescaling p}
\end{equation}
and get a rescaled system. In the rescaled system, let
\begin{equation}
  \mu^{\varepsilon, j}_{10} (\tilde{p}_j) = \frac{\lambda_{10} (p_j^0 +
  \varepsilon \tilde{p}_j)}{\varepsilon} . \label{mu def}
\end{equation}
Pass to the limit $\varepsilon \rightarrow 0$ and we can get a limit system. In
the limit system, $\mu^{\varepsilon, j}_{10} (\tilde{p}_j)$ becomes $| W_j 
\tilde{p}_j |$ because of \eqref{lpm asym}. This part will be shown in
\ref{3.6}.

We need to combine the two methods together. One the one hand, because of the presence of degeneracy, it is not enough to only use the expansion \eqref{c asym ori}. On the other hand, the Effective Hamiltonian method is a linear approximation near Weyl nodes and thus only valid in a small region. As a result, the expansion works well in $\Gamma_r$ while the Effective Hamiltonian method is applied in $\Gamma_s$. In  $\Gamma_m$, both methods are involved.

\subsection{The asymptotic expansion of $c$}\label{3.3}

In this part we will introduce the result on the asymptotic expansion of $c_{m
n}$ \eqref{c asym ori}. Here is the result.

\begin{lemma}
  \label{c result}Suppose $c_{m n}$ is the solution of \eqref{c equation} and
  $\breve{c}_{m n}$ is defined by \eqref{alt u}. If Assumption
  \ref{regular asump} holds, then we have the followingexpansion on $\Gamma_r
  \cup \Gamma_m$,
  \[  \breve{c}_{m n} = c_{m n}^0 + \frac{\varepsilon}{\lambda_{10}} c_{m n}^1
     + \left( \frac{\varepsilon}{\lambda_{10}} \right)^2 c_{m n}^2 + \left(
     \frac{\varepsilon}{\lambda_{10}} \right)^3 c_{m n}^3 + \left(
     \frac{\varepsilon}{\lambda_{10}} \right)^4 \varsigma_{m n}^4, \]
  with $c_{m n}^k (0 \leqslant k \leqslant 3)$ given by \
  \begin{eqnarray}
    c_{m n}^0 & = & c_{m n}^{i n},  \label{c0}\\
    c_{m n}^1 & = & 0,  \label{c1}\\
    c_{m n}^2 & = & (\lambda_{10})^2 T_{m n} (t, p) E (0) \cdummy X_{m n}  (p)
    - (\lambda_{10})^2 E (t) \cdummy \breve{X}_{m n}  (t, p),  \label{c2}\\
    \label{c3} c_{m n}^3 & = & \mathi (\lambda_{10})^3 T_{m n} (t, p) \frac{f_{mn} E' (0)
    \cdummy A_{mn}  (p) \lambda_{n m} (p)}{\lambda_{n m} (p)^3} \\
    & - & \mathi
    (\lambda_{10})^3 \frac{f_{mn} E' (t) \cdummy \breve{A}_{mn}  (t, p)
    \breve{\lambda}_{n m} (t, p)}{\breve{\lambda}_{n m} (t, p)^3} 
    \nonumber\\
    & + & \mathi T_{m n} (t, p) \int^t_0  (\lambda_{10})^3 E (s) \cdummy
    (\breve{A}_{n n} - \breve{A}_{m m}) (s, p) E (0) \cdummy X_{m n}  (p) d s.
    \nonumber
  \end{eqnarray}
  Here, $X_{m n}$ is defined by \eqref{X def} and \eqref{f def} and $T_{m n}$
  is defined by \eqref{T def}. And the remainder $\varsigma_{m n}^4$ given by
  \eqref{r4mn tilde} is bounded by a constant independent of $\varepsilon$.
\end{lemma}

Here, we only introduce the outline of the proof and the details of the proof are
given in Appendix \ref{pf c result}. For the expansion
\begin{equation}
  \breve{c}_{m n} = c_{m n}^0 + \left( \frac{\varepsilon}{\lambda_{10}}
  \right) c_{m n}^1 + \left( \frac{\varepsilon}{\lambda_{10}} \right)^2 c_{m
  n}^2 + \left( \frac{\varepsilon}{\lambda_{10}} \right)^3 c_{m n}^3 + \left(
  \frac{\varepsilon}{\lambda_{10}} \right)^4 \varsigma_{m n}^4, \label{repeat
  c asym}
\end{equation}
we will try to calculate $c_{m n}^k (0 \leqslant k \leqslant 3)$. Recalling
the initial data of $\breve{c}_{m n}$ given by \eqref{cu id}, then the initial
data of $c^k_{mn}$ is given by

\begin{align}
  c^k_{mn}  (0, p) & = c_{m n}^{i n} \delta_{k 0} .  \label{cn initial data}
\end{align}

By Duhamel's formula, we can rewrite \eqref{new c eq} in the form of integral
as
\begin{eqnarray}
  \breve{c}_{mn} & = & c_{m n}^{i n} - \mathi \varepsilon \int^t_0
  {e^{\frac{\mathi}{\varepsilon}  \int^t_s \breve{\lambda}_{mn} (r, p) d r}} 
  E (s) \cdummy \sum_p (\breve{A}_{mp}  \breve{c}_{pn} - \breve{c}_{mp} 
  \breve{A}_{pn}) (s, p) d s, 
\end{eqnarray}
Using $T_{m n}$ defined by \eqref{T def}, we can also rewrite the equation
above as
\begin{equation}
  \breve{c}_{mn} = c_{m n}^{i n} - \mathi \left(
  \frac{\varepsilon}{\lambda_{10}} \right) \lambda_{10} T_{m n} (t, p)
  \int^t_0 T_{n m} (s, p) E (s) \cdummy \sum_p (\breve{A}_{mp}  \breve{c}_{pn}
  - \breve{c}_{mp}  \breve{A}_{pn}) (s, p) d s, \label{cmn int}
\end{equation}
Notice the factor $\varepsilon$ on the right side of \eqref{cmn int}, we can
immediately give the zeroth-order terms as
\begin{equation}
  c_{mn}^0 (t, p) = c^{in}_{mn} (p) . \label{zeroth}
\end{equation}
It is natural to think that higher-order terms are given by
\begin{equation}
  c^k_{mn} = - \mathi \lambda_{10} T_{m n} (t, p) \int^t_0 T_{n m} (s, p) E
  (s) \cdummy \sum_p (\breve{A}_{mp} c_{pn}^{k - 1} - c_{mp}^{k - 1} 
  \breve{A}_{pn}) (s, p),\quad k \geqslant 1. \label{origin ckmn}
\end{equation}
But generally this is not right. The integral of the right side might be
oscillatory and in this case, some terms of the right side is $O \left(
\frac{\varepsilon}{\lambda_{10}} \right)$. Thus we need to check whether each
term of the right side is $O (1)$, and collect $O (1)$ terms for
$c^k_{mn}$ while high order terms should be moved to equations of $c^j_{mn} (j > k)$.

The process consists of two steps. Inspired by \eqref{origin ckmn}, we can
rewrite $ \breve{c}_{m n}$ as
\begin{equation}
  \breve{c}_{m n} = d_{m n}^0 + \left( \frac{\varepsilon}{\lambda_{10}}
  \right) d_{m n}^1 + \left( \frac{\varepsilon}{\lambda_{10}} \right)^2 d_{m
  n}^2 + \left( \frac{\varepsilon}{\lambda_{10}} \right)^3 d_{m n}^3 + \left(
  \frac{\varepsilon}{\lambda_{10}} \right)^4 R_{m n} . \label{d exp}
\end{equation}
Here, $d_{m n}^k$ is given by
\begin{eqnarray}
  d_{m n}^0 & = & c_{mn}^0,  \label{d0}\\
  d^k_{mn} & = & - \mathi \lambda_{10} T_{m n} (t, p) \int^t_0 T_{n m} (s, p)
  E (s) \cdummy \sum_p (\breve{A}_{mp} d_{pn}^{k - 1} - d_{mp}^{k - 1} 
  \breve{A}_{pn}) (s, p) d s,\, k \geqslant 1.  \label{ckmn}
\end{eqnarray}
Notice that this is not an asymptotic expansion because $d_{m n}^1$ might be
$O \left( \frac{\varepsilon}{\lambda_{10}} \right)$ instead of $O (1)$.

In the first step, we will analyze and rearrange $d_{m n}^k$ with $k \leqslant
3$ and distribute them to $c_{m n}^k$ with $k \leqslant 3$. In other words,
for each $d_{m n}^k (k \leqslant 3)$, we have the following asymptotic
expansion,
\begin{equation}
  d_{m n}^k = d_{m n}^{k, 0} + \left( \frac{\varepsilon}{\lambda_{10}} \right)
  d_{m n}^{k, 1} + \cdots + \left( \frac{\varepsilon}{\lambda_{10}} \right)^{3
  - k} d_{m n}^{k, 3 - k} + \kappa_{m n}^k, \label{dkmn asym}
\end{equation}
with $\kappa_{m n}^k = O \left( \left( \frac{\varepsilon}{\lambda_{10}}
\right)^{4 - k} \right)$ Let
\begin{equation}
  \breve{c}_{m n}^k = d_{m n}^{0, k} + d_{m n}^{1, k - 1} + \cdots + d_{m
  n}^{k, 0} . \label{ckmn cal}
\end{equation}
Then we will get
\begin{eqnarray*}
  &  & d_{m n}^0 + \left( \frac{\varepsilon}{\lambda_{10}} \right) d_{m n}^1
  + \left( \frac{\varepsilon}{\lambda_{10}} \right)^2 d_{m n}^2 + \left(
  \frac{\varepsilon}{\lambda_{10}} \right)^3 d_{m n}^3\\
  & = & c_{m n}^0 + \left( \frac{\varepsilon}{\lambda_{10}} \right) c_{m n}^1
  + \left( \frac{\varepsilon}{\lambda_{10}} \right)^2 c_{m n}^2 + \left(
  \frac{\varepsilon}{\lambda_{10}} \right)^3 c_{m n}^3 + O \left( \left(
  \frac{\varepsilon}{\lambda_{10}} \right)^4 \right) .
\end{eqnarray*}
In this step, each $c_{m n}^k$ will be $O (1)$ terms and not contain high
order terms. Finally, it is obvious that the remainder $\varsigma_{m n}^4$ is
given by
\begin{equation}
  \varsigma_{m n}^4 = R_{m n} + \sum_{k = 0}^3 \left(
  \frac{\varepsilon}{\lambda_{10}} \right)^{k - 4} \kappa_{m n}^k .
  \text{\label{r4mn tilde}}
\end{equation}
In the last step, we will prove that the remainder $R_{m n}$ in \eqref{d exp}
is bounded. In this way, it cannot contribute to $c_{m n}^k$ with $k \leqslant
3$. Then we get the asymptotic expansion of $\breve{c}_{m n}$.

\begin{remark}
  $c_{m n}^k$ might not be independent of $\varepsilon$. There are two
  reasons. The first reason is that $c_{m n}^k$ might contains $T_{m n} (t, p)
  {= e^{\frac{\mathi}{\varepsilon}  \int^t_0 \breve{\lambda}_{mn} (r, p) d r}}
  $, which is $O (1)$ and does not have strong limit as $\varepsilon
  \rightarrow 0$. The other reason is that $\breve{f} (t, p)$ defined by
  \eqref{alt u} depends on $\varepsilon$ and the expression of $c_{m n}^k$
  contains terms like $\breve{\lambda}_{m n}$. But both reasons do not affect the validity of the asymptotic expansion.
\end{remark}

\subsection{The rescaled system and its limit system}\label{3.6}

In this part, we will introduce the rescaled system and its limit system.

\subsubsection{The rescaled system}

For each Weyl node $p_j^0$, we first apply the rescaling \eqref{rescaling p}.
Let $b^j_{m n} (t, \tilde{p}_j)$ denote $c_{m n} (t, p)$ after rescaling. In
other words,
\begin{equation}
  b^j_{m n} (t, \tilde{p}_j) = c_{m n} (t, p_j^0 + \varepsilon \tilde{p}_j) .
  \label{b def}
\end{equation}
Then recalling $\mu_{m n}^j$ defined by \eqref{mu def}, $b^j_{m n}$ satisfy
the following equation from \eqref{c equation},
\begin{equation}
  \label{b equation} \frac{\partial}{\partial t} b^j_{m n} + \mathi
  \mu^{\varepsilon, j}_{mn} (\tilde{p}_j) b^j_{m n} = - E (t) \cdot
  \nabla_{\tilde{p}_j} b^j_{m n} - \mathi E (t) \cdot \sum_p (L_{mp}^j
  (\tilde{p}_j) b^j_{p n} - b^j_{m p} L_{pn}^j (\tilde{p}_j)) .
\end{equation}
Here, we define
\begin{eqnarray}
  L_{mp}^j (\tilde{p}_j) & = & \varepsilon A_{mp} (p_j^0 + \varepsilon
  \tilde{p}_j), \nonumber\\
  K^j_{m p} (\tilde{p}_j) & = & D_{mp} (p_j^0 + \varepsilon \tilde{p}_j) . 
  \label{E def}
\end{eqnarray}
And we have used the fact that $\nabla_{\tilde{p}_j} = \varepsilon \nabla_p$. After
this rescaling, $B_{W_j} (p_j^0, \varepsilon r_1)$ becomes $B_{W_j} (0, r_1)$.
Similarly, $B_{W_j} (p_j^0, r_2)$ becomes $B_{W_j} \left( 0,
\frac{r_2}{\varepsilon} \right)$. Then the characteristic line becomes
\begin{equation}
  \left\{\begin{array}{l}
    \frac{d}{d t} p (t) = E (t),\\
    p (0) = p_0,
  \end{array}\right. \label{new x_orbit}
\end{equation}
Now let $\tilde{P}_t (p_0)$ denote its solution and for each function $f (t,
p)$, define $\check{f} (t, p)$ as
\begin{equation}
  \check{f} (t, p) = f (t, \tilde{P}_t (p)) . \label{check def}
\end{equation}
It is easy to show that
\[ \check{b}^j_{m n} (t, \tilde{p}_j) = \breve{c}_{m n}  (t, p_j^0 +
   \varepsilon \tilde{p}_j), \]
and that $\check{b}^j_{m n}$ satisfies
\begin{equation}
  \label{new s equation} \frac{\partial}{\partial t}  \check{b}^j_{m n} +
  \mathi \check{\mu}^j_{mn} (\tilde{p}_j)  \check{b}^j_{m n} = - \mathi E (t)
  \cdot \sum_p (\check{L}^j_{mp} (\tilde{p}_j) \check{b}^j_{p n} -
  \check{b}^j_{m p}  \check{L}^j_{pn} (\tilde{p}_j)) .
\end{equation}
Since
\[ \frac{\varepsilon}{\lambda_{10} (p_j^0 + \varepsilon \tilde{p}_j)} =
   \frac{1}{\mu^{\varepsilon, j}_{10} (\tilde{p}_j)}, \]
we know that \eqref{c asym ori} is equivalent to the following expansion,
\begin{equation}
  \check{b}^j_{m n} = b_{m n}^{j, 0} + (\mu^{\varepsilon, j}_{10})^{- 1} b_{m
  n}^{j, 1} + (\mu^{\varepsilon, j}_{10})^{- 2} {b_{m n}^{j, 2}}   +
  (\mu^{\varepsilon, j}_{10})^{- 3} b_{m n}^{j, 3} + (\mu^{\varepsilon,
  j}_{10})^{- 4} \beta_{m n}^{j, 4}, \label{c asym res} 
\end{equation}
with
\begin{eqnarray*}
  b_{m n}^{j, k} (\tilde{p}_j) & = & c_{m n}^k (p_j^0 + \varepsilon
  \tilde{p}_j), \quad 0 \leqslant k \leqslant 3,\\
  \beta_{m n}^{j, 4} (\tilde{p}_j) & = & \varsigma_{m n}^4 (p_j^0 +
  \varepsilon \tilde{p}_j) .
\end{eqnarray*}
For the remainder $\varsigma_{m n}^4$ in \eqref{repeat c asym}, we have
\[ \varsigma_{m n}^4 = R_{m n} + \kappa_{m n}^0 + \kappa_{m n}^1 + \kappa_{m
   n}^2 + \kappa_{m n}^3 . \]
Then for $\beta_{m n}^{j, 4}$, we can also have the corresponding expression
\[ \beta_{m n}^{j, 4} = r_{m n}^{j, 4} + \alpha_{m n}^{j, 0} + \alpha_{m
   n}^{j, 1} + \alpha_{m n}^{j, 2} + \alpha_{m n}^{j, 3} . \]
Here,
\begin{eqnarray}
  r_{m n}^{j, 4} (\tilde{p}_j) & : = & R_{m n} (p_j^0 + \varepsilon
  \tilde{p}_j),  \label{small r def}\\
  \alpha_{m n}^{j, k} (\tilde{p}_j) & : = & \kappa_{m n}^k (p_j^0 +
  \varepsilon \tilde{p}_j) .  \label{alpha def}
\end{eqnarray}
Remember that $R_{m n}$ satisfies
\[ \begin{array}{lll}
     R_{m n} & = & - \mathi \lambda_{10} \int^t_0
     {e^{\frac{\mathi}{\varepsilon}  \int^t_s \breve{\lambda}_{mn} (r, p) d
     r}}  E (s) \cdummy \sum_p (\breve{A}_{mp} d_{pn}^3 - d_{mp}^3 
     \breve{A}_{pn}) (s, p) d s\\
     & - & \mathi \varepsilon \int^t_0 {e^{\frac{\mathi}{\varepsilon} 
     \int^t_s \breve{\lambda}_{mn} (r, p) d r}}  E (s) \cdummy \sum_p
     (\breve{A}_{mp} R_{p n} - R_{m p}  \breve{A}_{pn}) (s, p) d s.
   \end{array} \]
Then $r_{m n}^{j, 4}$ satisfies
\begin{equation}
  \begin{array}{lll}
    r_{m n}^{j, 4} & = & - \mathi \mu_{10} \int^t_0 {e^{\mathi \int^t_s
    \check{\mu}^j_{mn} (r, \tilde{p}_j) d r}}  E (s) \cdummy \sum_p
    (\check{L}^j_{mp} {b_{p n}^{j, 3}}   {- b_{m p}^{j, 3}}  
    \check{L}^j_{pn}) (s, \tilde{p}_j) d s\\
    & - & \mathi \int^t_0 {e^{\mathi \int^t_s \check{\mu}^j_{mn} (r,
    \tilde{p}_j) d r}}  E (s) \cdummy \sum_p (\check{L}^j_{mp} r_{p n}^{j, 4}
    - r_{m p}^{j, 4}  \check{L}^j_{pn}) (s, \tilde{p}_j) d s.
  \end{array} \label{small r equation}
\end{equation}

\subsubsection{The limit system}

To capture the leading order behavior of the rescaled system \eqref{b equation}, we can pass to the limit
$\varepsilon \rightarrow 0$ uniformly on this region, then thanks to
Propostion \ref{eff result},
\begin{eqnarray}
  \mu^{\varepsilon, j}_{m n} (\tilde{p}_j) & = & \frac{\lambda_{m n} (p_j^0 +
  \varepsilon \tilde{p}_j)}{\varepsilon} \rightarrow \lambda_{j, m n},\quad m, n
  \in \{ 0, 1 \},  \label{lim 1}\\
  L^j_{m n} (\tilde{p}_j) & = & \varepsilon A_{m n} (p_j^0 + \varepsilon
  \tilde{p}_j) \rightarrow A_{j, m n},\quad m, n \in \{ 0, 1 \}, \\
  L^j_{m n} (\tilde{p}_j) & = & \varepsilon A_{mp} (p_j^0 + \varepsilon
  \tilde{p}_j) \rightarrow 0,\quad  \forall m > 1 \infixor n > 1, \label{lim 4}\\
  K^j_{m n} (\tilde{p}_j) & = & D_{m n} (p_j^0 + \varepsilon \tilde{p}_j)
  \rightarrow D_{j, m n} (\tilde{p}_j).  \label{lim 5}
\end{eqnarray}
\begin{remark}
  In \eqref{c equation}, we should notice that there are two plausibly
  singular terms $\frac{\lambda_{m n}}{\varepsilon} \infixand \mathi
  \varepsilon E (t) \cdot \sum_p (A_{mp} c_{pn} - c_{mp} A_{pn})$ when
  $\varepsilon \rightarrow 0$. According to the previous assumptions, the limit of $\frac{\lambda_{m
  n}}{\varepsilon}(m,n\in\{0,1\})$ exists. For the second term, recalling \eqref{AD
  relation}, the plausibly singularity comes from
  $\frac{\varepsilon}{\lambda_{m n}}$ (when $m = 0, n = 1$ the denominator
  might be zero), which is just the reciprocal of $\frac{\lambda_{m
  n}}{\varepsilon}$. This coincidence guarantees that we will not get
  singularity when passing to the limit. 
\end{remark}

Here, we use the approximation from the effective Hamiltonian as we mentioned
in Assumption \ref{regular asump}. Notice that \eqref{lim 4} implies that if
we take the limit $\varepsilon \rightarrow 0$, then all $\check{b}_{m n}^j
\rightarrow 0 \left( m > 1 \infixor n > 1 \right)$. Only $\check{b}_{m n}^j$
with $m, n \in \{ 0, 1 \}$ are left. Let $s_{m n}$ denote the limit of $b_{m
n}$ with $m, n \in \{ 0, 1 \}$. Then we can write the limit system as
\begin{equation}
  \frac{\partial}{\partial t} s^j_{m n} + \mathi \lambda_{j, m n} s^j_{m n} =
  - E (t) \cdot \nabla_{\tilde{p}_j} s^j_{m n} - \mathi E (t) \cdot \sum_{p =
  0, 1} (A_{j, m p} s^j_{p n} - s^j_{m p} A_{j, p n} ), m, n \in \{ 0, 1 \},
  \label{limit system}
\end{equation}
or specificly,
\begin{eqnarray*}
  \frac{\partial}{\partial t} s^j_{00} & = & - E (t) \cdot
  \nabla_{\tilde{p}_j} s^j_{00} - \mathi E (t) \cdot (\frac{D_{j, 01}
  (\tilde{p}_j)}{\mathi | W_j  \tilde{p}_j |} s^j_{1 0} - s^j_{01} 
  \frac{D_{j, 10} (\tilde{p}_j)}{\mathi | W_j  \tilde{p}_j |}) \\
  \frac{\partial}{\partial t} s^j_{11} & = & - E (t) \cdot
  \nabla_{\tilde{p}_j} s^j_{11} - \mathi E (t) \cdot (\frac{D_{j, 10}
  (\tilde{p}_j)}{\mathi | W_j  \tilde{p}_j |} s^j_{01} - s^j_{10}  \frac{D_{j,
  01} (\tilde{p}_j)}{\mathi | W_j  \tilde{p}_j |}) \\
  \frac{\partial}{\partial t} s^j_{10} + \mathi | W_j  \tilde{p}_j | s^j_{10}
  & = & - E (t) \cdot \nabla_{\tilde{p}_j} s^j_{10} - \mathi E (t) \cdot
  (\frac{D_{j, 10} (\tilde{p}_j)}{\mathi | W_j  \tilde{p}_j |} (s^j_{00} -
  s^j_{11}) + \Lambda_{j, 10} (\tilde{p}_j) s^j_{10}) . 
\end{eqnarray*}
$s^j_{01}$ is given by $(s^j_{10})^{\ast}$. The initial data is given by
\[ s^j_{m n} (0, \tilde{p}_j) = c_{m n}^{i n} (p_j^0) . \]
\begin{remark}
  Here, we discard high-order terms of $o(\varepsilon^3)$. In fact, in this note we only plan to
  give the asymptotic expansion to the order $O (\varepsilon^3)$. We only pass
  to this limit directly on $\Gamma_s$, the area of which is $O
  (\varepsilon^3)$. Thus the contribution of high order terms are at least $o
  (\varepsilon^3)$.
\end{remark}

Since $B_{W_j} (0, r_1)$ is a bounded area, we can uniformly pass to the limit
$\varepsilon \rightarrow 0$ on $B_{W_j} (0, r_1)$ and the error will be $o
(1)$. In fact, we can pass to this limit uniformly in any bounded area. But we
cannot pass to the limit on $B_{W_j} \left( 0, \frac{r_2}{\varepsilon} \right)
.$

It is easy to see that $\check{s}^j_{m n}$ satisfies
\begin{equation}
  \frac{\partial}{\partial t} \check{s}^j_{m n} + \mathi \check{\lambda}_{j, m
  n} \check{s}^j_{m n} = - \mathi E (t) \cdot \sum_{p = 0, 1} (\check{A}_{j, m
  p} \check{s}^j_{p n} - \check{s}^j_{m p} \check{A}_{j, p n} ), m, n \in \{
  0, 1 \} . \label{limit system check}
\end{equation}
Then for $\check{s}^j_{m n}$, because of \eqref{c asym res}, we can assume
that $\check{s}^j_{m n}$ has the following expansion,
\begin{equation}
  \check{s}^j_{m n} = s_{m n}^{j, 0} + | W_j  \tilde{p}_j |^{- 1} s_{m n}^{j,
  1} + | W_j  \tilde{p}_j |^{- 2} s_{m n}^{j, 2} + | W_j  \tilde{p}_j |^{- 3}
  s_{m n}^{j, 3} + | W_j  \tilde{p}_j |^{- 4} \sigma_{m n}^{j, 4} . \label{s
  asym}
\end{equation}
Since $b_{m n}^{j, k}$ in \eqref{c asym res} has explicit expressions, we can
pass to the limit $\varepsilon \rightarrow 0$ with the help of \eqref{lim 1}
to \eqref{lim 4}. Notice that the expression of $b_{m n}^{j, k}$ comes from
Lemma \ref{c result}. Thus when we pass to the limit, we just need to replace
the terms in Lemma \ref{c result} by their limits. Then we can get the
expression of $s_{m n}^{j, k}$. They are
\begin{eqnarray*}
  s_{m n}^{j, 0} & = & c_{m n}^0 (p_j^0),  \label{s0}\\
  s_{m n}^{j, 1} & = & s_{0 0}^{j, 2} = s_{11}^{j, 2} = s_{0 0}^{j, 3} = s_{0
  0}^{j, 3} = 0, \\
  s_{10}^{j, 2} (t, \tilde{p}_j) & = & \mathi\, e^ { \mathi \int_0^t |
  W_j \tilde{P}_t (\tilde{p}_j) | d s } E (0) \cdummy D_{j, 10} 
  (\tilde{p}_j) - \mathi E (t) \cdummy \check{D}_{j, 10} (\tilde{p}_j), \\
  s_{10}^{j, 3} (t, p) & = & e^ { \mathi \int_0^t | W_j \tilde{P}_t
  (\tilde{p}_j) | d s } D_{j, 10}  (\tilde{p}_j) \cdummy \left( | W_j 
  \tilde{p}_j | E (0) \int^t_0 E (s) \cdummy \check{\Lambda}_{j, 10} (s,
  \tilde{p}_j) d s - E' (0) \right) \nonumber\\
  &  & + \frac{E' (t) \cdummy \check{D}_{j, 10}  (t, \tilde{p}_j) | W_j 
  \tilde{p}_j |^3}{| W_j \tilde{P}_t (\tilde{p}_j) |^3} .  \label{s4}
\end{eqnarray*}
Similarly, for $\alpha_{m n}^k (1 \leqslant k \leqslant 3)$ defined by
\eqref{alpha def} we have explicit expressions and we can suppose that
$\alpha_{m n}^k \rightarrow \eta_{m n}^k (0 \leqslant k \leqslant 3)$.
Therefore, we can suppose that
\[ \sigma_{m n}^{j, 4} = \rho_{m n}^{j, 4} + \eta_{m n}^0 + \eta_{m n}^1 +
   \eta_{m n}^2 + \eta_{m n}^3 . \]
For \eqref{small r equation}, we can pass to the limit $\varepsilon
\rightarrow 0$ and let
\begin{equation}
  \lim_{\varepsilon \rightarrow 0} r_{m n}^{j, 4} = \rho_{m n}^{j, 4} .
  \label{remainder limit part}
\end{equation}
Here, $\rho_{m n}^{j, 4}$ is given by
\begin{eqnarray*}
  \rho_{0 0}^{j, 4} & = & - \mathi | W_j  \tilde{p}_j | \int^t_0 E (s) \cdummy
  \left( \frac{\eta_{10}^3 \check{D}_{j, 01} (s, \tilde{p}_j)}{\mathi | W_j
  \tilde{P}_s (\tilde{p}_j) |} - \frac{\eta_{01}^3 \check{D}_{j, 10} (s,
  \tilde{p}_j)}{\mathi | W_j \tilde{P}_s (\tilde{p}_j) |} \right) d s
  \nonumber\\
  & - & \mathi \int^t_0 E (s) \cdummy \left( \frac{\rho_{10}^{j, 4}
  \check{D}_{j, 01} (s, \tilde{p}_j)}{\mathi | W_j \tilde{P}_s (\tilde{p}_j)
  |} - \frac{\rho_{01}^{j, 4} \check{D}_{j, 10} (s, \tilde{p}_j)}{\mathi | W_j
  \tilde{P}_s (\tilde{p}_j) |} \right) d s. \nonumber\\
  \rho_{01}^{j, 4} & = & - \mathi | W_j  \tilde{p}_j | \int^t_0 {e^{\mathi
  \int^t_s | W_j \tilde{P}_s (\tilde{p}_j) | d r}}  E (s) \cdummy \left(
  \frac{(\eta_{00}^3 - \eta_{11}^3) \check{D}_{j, 10} (s, \tilde{p}_j)}{\mathi
  | W_j \tilde{P}_s (\tilde{p}_j) |} - \eta_{01}^3  \check{\Lambda}_{j, 10}
  (s, \tilde{p}_j) \right) d s \nonumber\\
  & - & \mathi \int^t_0 {e^{\mathi \int^t_s | W_j \tilde{P}_s (\tilde{p}_j) |
  d r}}  E (s) \cdummy \left( \frac{(\rho_{00}^{j, 4} - \rho_{11}^{j, 4})
  \check{D}_{j, 10} (s, \tilde{p}_j)}{\mathi | W_j \tilde{P}_s (\tilde{p}_j)
  |} - \rho_{01}^{j, 4}  \check{\Lambda}_{j, 10} (s, \tilde{p}_j) \right) d s,
  \nonumber\\
  \rho_{m n}^{j, 4} & = & 0, \quad m > 1 \infixor n > 1.  \label{vanish}
\end{eqnarray*}
We do not list $\rho_{11}^{j, 4}$ and $\rho_{10}^{j, 4}$ here because it is
obvious that
\begin{equation}
  \rho_{0 0}^{j, 4} + \rho_{11}^{j, 4} = 0. \label{oppo}
\end{equation}
and that
\begin{equation}
  \rho_{10}^{j, 4} = (\rho_{01}^{j, 4} )^{\ast} .
\end{equation}
It is clear that the validity of the limit system holds uniformly for $0 < C'_1 < | W_j 
\tilde{p}_j | < C'_2$. It follows that
\begin{equation}
  \lim_{\varepsilon \rightarrow 0} \beta_{m n}^{j, 4} = \sigma_{m n}^{j, 4},
  \label{remainder limit}
\end{equation}
holds uniformly for $0 < C'_1 < | W_j  \tilde{p}_j | < C'_2$. Finally, for the
limit system, the current $J$ for the $k$-th Weyl node is given by
\begin{equation}
  J_k = \sum_{m, n \in \{ 0, 1 \}} \int_{\mathbb{R}^3} s_{m n}^k D_{k, n m} d
  p. \label{limit J}
\end{equation}

\subsection{Summary}

In summary, we apply a rescaling for the system in a neighborhood of each Weyl
node and pass to the limit $\varepsilon \rightarrow 0$ and get a limit system.
We can pass to the limit $\varepsilon \rightarrow 0$ uniformly in any area \
$0 < C'_1 < | W_j  \tilde{p}_j | < C'_2$. Here are the asymptotic expansions of
the original, the rescaled and the limit system.

For the original system, the expansion is
\begin{eqnarray*}
  \breve{c}_{m n} & = & c_{m n}^0 + \left( \frac{\varepsilon}{\lambda_{10}}
  \right) c_{m n}^1 + \left( \frac{\varepsilon}{\lambda_{10}} \right)^2 c_{m
  n}^2 + \left( \frac{\varepsilon}{\lambda_{10}} \right)^3 c_{m n}^3 + \left(
  \frac{\varepsilon}{\lambda_{10}} \right)^4 \varsigma_{m n}^4,\\
  \varsigma_{m n}^4 & = & R_{m n} + \kappa_{m n}^0 + \kappa_{m n}^1 +
  \kappa_{m n}^2 + \kappa_{m n}^3 .
\end{eqnarray*}
For the rescaled system, the expansion is
\begin{eqnarray*}
  \check{b}^j_{m n} & = & b_{m n}^{j, 0} + (\mu^{\varepsilon, j}_{10})^{- 1}
  b_{m n}^{j, 1} + (\mu^{\varepsilon, j}_{10})^{- 2} {b_{m n}^{j, 2}}   +
  (\mu^{\varepsilon, j}_{10})^{- 3} b_{m n}^{j, 3} + (\mu^{\varepsilon,
  j}_{10})^{- 4} \beta_{m n}^{j, 4},\\
  \beta_{m n}^{j, 4} & = & r_{m n}^{j, 4} + \alpha_{m n}^{j, 0} + \alpha_{m
  n}^{j, 1} + \alpha_{m n}^{j, 2} + \alpha_{m n}^{j, 3} .
\end{eqnarray*}
For the limit system, the expansion is
\begin{eqnarray*}
  \check{s}^j_{m n} & = & s_{m n}^{j, 0} + | W_j  \tilde{p}_j |^{- 1} s_{m
  n}^{j, 1} + | W_j  \tilde{p}_j |^{- 2} s_{m n}^{j, 2} + | W_j  \tilde{p}_j
  |^{- 3} s_{m n}^{j, 3} + | W_j  \tilde{p}_j |^{- 4} \sigma_{m n}^{j, 4},\\
  \sigma_{m n}^{j, 4} & = & \rho_{m n}^{j, 4} + \eta_{m n}^0 + \eta_{m n}^1 +
  \eta_{m n}^2 + \eta_{m n}^3 .
\end{eqnarray*}
So far, we have enough knowledge about the asymptotic behavior of $c_{m n}$ in
different regions. In the next section, we can plug the asymptotic expansion
of $c_{m n}$ into $J$ and finish the proof of Theorem \ref{J asym exp}, which
is the second part of our proof.

%% file: Chapter_5.tex
\section{Asymptotic Expansion of the Current}\label{4}

In this part, we will use the asymptotic expansion of the Bloch density matrix $c_{m n}$ to calculate $J$. In
other words, for $J$ given by \eqref{J exp}, our goal is to find $J_k^i (i =
0, 1, 2, 3)$ independent of $\varepsilon$ such that
\begin{equation}
  J_k (t) = J_k^0 (t) + \varepsilon J_k^1 (t) + (\varepsilon)^2 J_k^2 (t) +
  (\varepsilon)^3 J_k^3 (t) + o (\varepsilon^3) .
\end{equation}
The strategy is to substitute \eqref{c asym ori} into \eqref{new J exp} and
calculate each term. There are four steps. The current can be divided into
three terms,
\[ J_k (t) = \sum_{m, n} \int_{\Gamma_r} \breve{c}_{nm} \breve{D}^k_{mn}
   \chi_1 dp + \sum_{m, n} \int_{\Gamma_m} \breve{c}_{nm} \breve{D}^k_{mn}
   \chi_2 dp + \sum_{m, n} \int_{\Gamma_s} \breve{c}_{nm} \breve{D}^k_{mn} dp.
\]
In the first three steps, we will calculate the asymptotic expansion of the
three parts with different strategies. In the last step, we will glue them
together.

\begin{itemize}
    \item {\tmem{Calculation of integral on $\Gamma_s$: }}For the integral on
$\Gamma_s$, we can directly pass to the limit because in $\Gamma_s$, $|
\mu^{\varepsilon, j}_{10} | < r_1$. However, we cannot pass to the limit for those on $\Gamma_m$ because $\mu^{\varepsilon, j}_{10}$
varies from $O (1)$ to $O (1 / \varepsilon)$.
    \item {\tmem{Calculation of integral on $\Gamma_r$: }}For $\Gamma_r$, we can
directly use the asymptotic expansion and discard high-order terms because
$\lambda_{10}$ is bounded from below. For $\Gamma_m$, again, we cannot discard
high-order terms because $\frac{\varepsilon}{\lambda_{10}}$ might be $O (1)$
instead of $O (\varepsilon)$.
    \item {\tmem{Calculation of integral on $\Gamma_m$: }}For $\Gamma_m$, we need a
more delicate treatment by combining the two strategies together. First,
we plug the asymptotic expansion of $\breve{c}_{m n}$ in the integral. Second, we analyze the integral by taking the rescaling
\eqref{rescaling p} for each Weyl node $p_j^0$ and passing to the limit $\varepsilon \to 0$. 

The original region of
integral is $B_{W_j} (p_i, r_2) \backslash B_{W_j} (p_i, \varepsilon r_1)$, and
$| p - p_j^0 |$ varies from $O (\varepsilon)$ to $O (1)$. After the change of
variable, the region of integral for $\check{p}_i$ is $B_{W_j} \left( 0,
\frac{r_2}{\varepsilon} \right) \backslash B_{W_j} (0, r_1)$ and $|
\tilde{p}_j |$ varies from $O (1)$ to $O (1 / \varepsilon)$. 
We can apply the dominant convergence
theorem and other techniques to analyze the reformulated current.
\end{itemize}

\begin{remark}
  In practice, we do not have to strictly follow the three steps above. Instead, we
  divide \eqref{new J exp} into two parts,
  \[ J_k = J_{k, r m} + J_{k, s}, \]
  with
  \begin{eqnarray*}
    J_{k, r m} & = & \sum_{m, n} \int_{\Gamma_r \cup \Gamma_m} \breve{D}^k_{n
    m}  \breve{c}_{m n} dp,\\
    J_{k, s} & = & \sum_{m, n} \int_{\Gamma s} \breve{D}^k_{n m}  \breve{c}_{m
    n} dp.
  \end{eqnarray*}
  Recall the limit system \eqref{limit system}. For $J_{k, s}$, we can
  directly pass to the limit $\varepsilon \rightarrow 0$ and get
  \[ J_{k, s} = \varepsilon^3 \sum_j \sum_{m, n \in \{ 0, 1 \}} \int_{B_{W_j}
     (0, r_1)} \check{s}_{m n}^j \check{D}^k_{j, m n} d p + o (\varepsilon^3)
     . \]
  For $J_{k, r m}$, plugging \eqref{c asym ori} into the expression of $J_{k,
  r m}$ yields
  \[ J_{k, r m} = \int_{\Gamma r \cup \Gamma_m} \breve{D}^k_{n m}  \left( c_{m
     n}^0 + \frac{\varepsilon}{\lambda_{10}} c_{m n}^1 + \left(
     \frac{\varepsilon}{\lambda_{10}} \right)^2 c_{m n}^2 + \left(
     \frac{\varepsilon}{\lambda_{10}} \right)^3 c_{m n}^3 + \left(
     \frac{\varepsilon}{\lambda_{10}} \right)^4 \varsigma_{m n}^4 \right) dp.
  \]
  Then we analyze each term seperately.
\end{remark}

\begin{itemize}
    \item {\tmem{Glue the asymptotic expansions on different regions together:
}}In particular, we will show that the result does not depend on the choice of $r_1$
and $r_2$. Here is an heuristic explanation of our strategy.

Plugging \eqref{s asym} into the expression of $J_{k, s}$, we can formally
get
\begin{eqnarray*}
  J_{k, s} & = & \varepsilon^3 \sum_j \sum_{m, n \in \{ 0, 1 \}} \int_{B_{W_j}
  (0, r_1)} \check{D}^k_{j, m n} (\nobracket s_{m n}^{j, 0} + | W_j 
  \tilde{p}_j |^{- 1} s_{m n}^{j, 1} + | W_j  \tilde{p}_j |^{- 2} s_{m n}^{j,
  2}\\
  & + & | W_j  \tilde{p}_j |^{- 3} s_{m n}^{j, 3} + | W_j  \tilde{p}_j |^{-
  4} \sigma_{m n}^{j, 4} \nobracket) d \tilde{p}_j .
\end{eqnarray*}
After rescaling and passing to the limit $\varepsilon \rightarrow 0$,
$\Gamma_r \cup \Gamma_m$ corresponds to the rescaled region $\mathbb{R}^3 \backslash
B_{W_j} (0, r_1)$. In this region, we can divide those terms above into three
cases.

{\tmem{Case A}} Some terms decaying slowly at infinity are not integrable on
$\mathbb{R}^3$ but are integrable on $B_{W_j} (0, r_1)$. Let us take
$s_{m n}^{j, 0} = 1$ as an example. Since the integral of this term on $B_{W_j}
(0, r_1)$ is $O (\varepsilon^3)$, we can replace $s_{m n}^{j, 0}$ by $c_{m
n}^j$ with an $o (\varepsilon^3)$ error and merge the integral of
$c_{m n}^j$ on $\Gamma_s$ with that on $\Gamma_r \cup \Gamma_m$ together.

{\tmem{Case B}} Some terms are integrable on $\mathbb{R}^3 \backslash
B_{W_j} (0, r_1)$. On $\Gamma_r$, we directly passed to the limit $\varepsilon
\rightarrow 0$ and get the integral of these terms on $\mathbb{R}^3 \backslash
B_{W_j} (0, r_1)$. Thus we can directly merge the integrals of these terms on
$\mathbb{R}^3 \backslash B_{W_j} (0, r_1)$ and those on $B_{W_j} (0, r_1)$.

{\tmem{Case C}} There is another term which is integrable on neither $B_{W_j} (0,
r_1)$ nor $\mathbb{R}^3 \backslash B_{W_j} (0, r_1)$. We need to deal with it
specially, which we will elaborate on later.

\end{itemize}

\begin{remark}
    \label{rmk integrable}
  $| W_j  \tilde{p}_j |^{- 3} s_{m n}^{j, 3}$ and $| W_j  \tilde{p}_j |^{- 4}
  \sigma_{m n}^{j, 4}$ are singular and not integrable. But if we add these
  two terms together, the sum of them is just $s_{m n}^j - (s_{m n}^{j, 0} + |
  W_j  \tilde{p}_j |^{- 1} s_{m n}^{j, 1} + | W_j  \tilde{p}_j |^{- 2} s_{m
  n}^{j, 2})$. It is integrable since $s_{m n}^j, s_{m n}^{j, 0}, | W_j 
  \tilde{p}_j |^{- 1} s_{m n}^{j, 1}$ and $| W_j  \tilde{p}_j |^{- 2} s_{m
  n}^{j, 2}$ are all integrable. Therefore, on $\Gamma_s$ we consider these
  two terms together. In our strategy, we do not separate the integral of $|
  W_j  \tilde{p}_j |^{- 3} s_{m n}^{j, 3} + | W_j  \tilde{p}_j |^{- 4}
  \sigma_{m n}^{j, 4}$, which makes $J_{k, r m} + J_{k, s}$ computable.
\end{remark}

The following part is devoted to proving Theorem \ref{J asym exp}.

\begin{proof}
  To start with, we can first rewrite the asymptotic expansion of $c_{m n}$ as
  \begin{eqnarray}
    c_{m n}^{2, a} & = & (\lambda_{10})^2 T_{m n} (t, p) E (0) \cdummy X_{m n}
    (p), \\
    c_{m n}^{2, b} & = & - (\lambda_{10})^2 E (t) \cdummy \breve{X}_{m n}  (t,
    p),  \label{c2b def}\\
    c_{m n}^{3, a} & = & \mathi (\lambda_{10})^3 T_{m n} (t, p) \frac{f_{mn}
    E' (0) \cdummy A_{mn}  (p) \lambda_{n m} (p)}{\lambda_{n m} (p)^3}, \\
    c_{m n}^{3, b} & = & - \mathi (\lambda_{10})^3 \frac{f_{mn} E' (t) \cdummy
    \breve{A}_{mn}  (t, p) \breve{\lambda}_{n m} (t, p)}{\breve{\lambda}_{n m}
    (t, p)^3}, \\
    c_{m n}^{3, c} & = & \mathi T_{m n} (t, p) \int^t_0  (\lambda_{10})^3 E
    (s) \cdummy (\breve{A}_{n n} - \breve{A}_{m m}) (s, p) E (0) \cdummy X_{m
    n}  (p) d s. 
  \end{eqnarray}
  Then according to Proposition \ref{c result} we have
  \begin{eqnarray}
    c_{m n}^2 & = & c_{m n}^{2, a} + c_{m n}^{2, b}, \\
    c_{m n}^3 & = & c_{m n}^{3, a} + c_{m n}^{3, b} + c_{m n}^{3, c} . 
  \end{eqnarray}
  Similarly, we can also define $s_{m n}^{j, 2, a}, s_{m n}^{j, 2, b}, s_{m
  n}^{j, 3, a}, s_{m n}^{j, 3, b} \infixand s_{m n}^{j, 3, c}$, which are the
  limit of $c_{m n}^{2, a}, c_{m n}^{2, b}, c_{m n}^{3, a}, c_{m n}^{3, b}
  \infixand c_{m n}^{3, c}$ as $\varepsilon \to 0$. In the process of
  calculating $J_k$, there are several cases, which we will introduce
  respectively.
  
  \begin{case}
    For $c_{m n}^{2, b}$ and $c_{m n}^0$, we use the same strategy. We first
    focus on $c_{m n}^0$. For $c_{m n}^0$, we have
    \begin{eqnarray*}
        & \sum_{m, n} \int_{\Gamma_r \cup \Gamma_m} \breve{D}^k_{n m} (t, p) c_{m
       n}^0 (t, p) dp \\
       = & \sum_{m, n} \left( \int_{\Gamma^{\ast}} \breve{D}^k_{n
       m} (t, p) c_{m n}^0 (t, p) dp - \int_{\Gamma_s} \breve{D}^k_{n m} (t,
       p) c_{m n}^0 (t, p) dp \right) .
    \end{eqnarray*}
    Notice that $s_{m n}^{j, 0}$ and $\check{D}^k_{j, m n}$ are the limit of
    $c_{m n}^0$ and $\breve{D}^k_{n m}$. It follows that
    \[ \sum_{m, n} \int_{\Gamma_s} \breve{D}^k_{n m} (t, p) c_{m
       n}^0 (t, p) dp = \varepsilon^3 \sum_j \sum_{m, n \in \{ 0, 1 \}}
       \int_{B_{W_j} (0, r_1)} \check{D}^k_{j, m n} s_{m n}^{j, 0} d
       \tilde{p}_j + o (\varepsilon^3) . \]
    This means that we can merge and get
    \begin{align}
      & \sum_{m, n} \int_{\Gamma_r \cup \Gamma_m} \breve{D}^k_{n m} (t, p)
      c_{m n}^0 (t, p) dp + \varepsilon^3 \sum_j \sum_{m, n \in \{ 0, 1 \}}
      \int_{B_{W_j} (0, r_1)} \check{D}^k_{j, m n} s_{m n}^{j, 0} d
      \tilde{p}_j \nonumber\\
      = & \int_{\Gamma^{\ast}} \breve{D}^k_{0 0} (t, p) c_{00}^{i n} (p) dp +
      o (\varepsilon^3) .  \label{sym1}
    \end{align}
    Here we used the fact that only $c_{00}^{i n}$ is nonzero. Recall that
    \[ \breve{D}^k_{0 0} (t, p) = D^k_{0 0} \left( p + \varepsilon \int_0^t E
       (s) d s \right) . \]
    By change of variable, we immediately get
    \[ \int_{\Gamma^{\ast}} \breve{D}^k_{0 0} (t, p) c_{00}^{i n} (p) dp =
       \int_{\Gamma^{\ast}} D^k_{0 0} (p) c_{00}^{i n} \left( p - \varepsilon
       \int_0^t E (s) d s \right) dp. \]
    By Taylor's expansion, we finally get
    \begin{align*}
      & \sum_{m, n} \int_{\Gamma_r \cup \Gamma_m} \breve{D}^k_{n m} (t, p)
      c_{m n}^0 (t, p) dp + \varepsilon^3 \sum_j \sum_{m, n \in \{ 0, 1 \}}
      \int_{B_{W_j} (0, r_1)} \check{D}^k_{j, m n} s_{m n}^{j, 0} d
      \tilde{p}_j\\
      = & \int_{\Gamma^{\ast}} D^k_{0 0} (p) c_{00}^{i n} (p) dp - \varepsilon
      \int_0^t E (s) d s \cdummy \int_{\Gamma^{\ast}} \nabla c_{00}^{i n} (p)
      D^k_{0 0} (p) dp\\
      + & \frac{\varepsilon^2}{2} \int_0^t E (s) \otimes E (s) d s :
      \int_{\Gamma^{\ast}} \nabla^2 c_{00}^{i n} (p) D^k_{0 0} (p) dp + o
      (\varepsilon^3) .
    \end{align*}
    For $c_{m n}^{2, b}$, the contribution is given by
    \begin{align*}
      & \sum_{m, n} \int_{\Gamma_r \cup \Gamma_m} \breve{D}^k_{n m} (t, p)
      \left( \frac{\varepsilon}{\lambda_{10} (p)} \right)^2 c_{m n}^{2, b} (t,
      p) dp  \\
      & + \varepsilon^3 \sum_j \sum_{m, n \in \{ 0, 1 \}} \int_{B_{W_j}
      (0, r_1)} |  \tilde{p}_j |^{- 2} \check{D}^k_{j, m n} s_{m n}^{j, 2, b}
      d \tilde{p}_j\\
      = & 2 \varepsilon^2 \sum_{n \geqslant 1} \int_{\Gamma^{\ast}}
      \frac{c_{00}^{i n} (p) E (t) \cdummy \tmop{Im} (D_{0 n}  (p) D^k_{n 0} 
      (p))}{\lambda_{n 0} (p)^2} dp + o (\varepsilon^3) .
    \end{align*}
    The details are given in Appendix \ref{C1}.
  \end{case}

    \begin{case}
      For $c_{m n}^{2, a}$, $c_{m n}^{3, a}$ and $c_{m n}^{3, c}$, we should
      notice that they are oscillatory. Here we only show the calculation of
      $c_{m n}^{2, a}$. The other two terms are similar. For $c_{m n}^{2, a}$,
      we have
      \[ \sum_{m, n} \int_{\Gamma r \cup \Gamma_m} \left(
         \frac{\varepsilon}{\lambda_{10}} \right)^2 c_{m n}^{2, a} 
         \breve{D}^k_{n m} dp = L_1 + L_2, \]
      with
      \begin{eqnarray}
        L_1 & = & \varepsilon^2 \sum_{m, n} \int_{\Gamma_r} T_{m n} (t, p) E
        (0) \cdummy X_{m n}  (p)  \breve{D}^k_{n m} \chi_1 dp,  \label{L1
        def}\\
        L_2 & = & \varepsilon^2 \sum_{m, n} \int_{\Gamma_m} T_{m n} (t, p) E
        (0) \cdummy X_{m n}  (p)  \breve{D}^k_{n m} \chi_2 dp.  \label{L2 def}
      \end{eqnarray}
      $L_1$ is an oscillatory integral if $m \neq n$ and zero if $m = n$. We
      need to apply the stationary phase method. Since $D_{n m}, A_{n m},
      \lambda_{n m} $ are all periodical functions, we can think that
      $\Gamma^{\ast}$ has no boundary. Then the only boundary of $\Gamma_r$ is
      $\bigcup^N_{j = 1} \partial B_{W_j} \left( p_j^0, \frac{r_2}{2}
      \right)$. But remember that $\chi_1$ vanishes on the boundary of
      $\Gamma_r$. So there are no boundary terms when we apply the stationary
      phase method. Thanks to Assumption \ref{regular asump}, applying the
      stationary method, we immediately get $L_1 = O (\varepsilon^{7 / 2})$ and
      should be discarded. For $L_2$, we can directly pass to the limit
      $\varepsilon \rightarrow 0$ and get
      \begin{equation}
        L_2 = \varepsilon^3 \sum_j \sum_{m, n \in \{ 0, 1 \}}
        \int_{\mathbb{R}^3 \backslash B_{W_j} (0, r_1)}  \check{D}^k_{j, m n}
        | W_j  \tilde{p}_j |^{- 2} s_{m n}^{j, 2, a} d \tilde{p}_j + o
        (\varepsilon^3) . \label{limit L2}
      \end{equation}
      The proof of the limit above is in Appendix \ref{C2}. Similarly, we have
      \begin{align}
        & \sum_{m, n} \int_{\Gamma r \cup \Gamma_m} \left(
        \frac{\varepsilon}{\lambda_{10}} \right)^3 (c_{m n}^{3, a} + c_{m
        n}^{3, c})  \breve{D}^k_{n m} dp \nonumber\\
        = & \varepsilon^3 \sum_j \sum_{m, n \in \{ 0, 1 \}} \int_{\mathbb{R}^3
        \backslash B_{W_j} (0, r_1)}  \check{D}^k_{j, m n} | W_j  \tilde{p}_j
        |^{- 3}  (s_{m n}^{j, 3, a} + s_{m n}^{j, 3, c}) d \tilde{p}_j + o
        (\varepsilon^3) .  \label{L4 limit}
      \end{align}
      The details are in Appendix \ref{C3}.
    \end{case}
  
  \begin{case}
    For the remainders $\varsigma_{m n}^4$, it is similar to Case 2. On
    $\Gamma_r$, they have no contributions because they are $O (\varepsilon^4)
    .$ On $\Gamma_m$, by Lemma \ref{tricky}, we can pass to the limit
    $\varepsilon \rightarrow 0 +$ and get
    \begin{eqnarray*}
        &\sum_{m, n} \int_{\Gamma r \cup \Gamma_m} \left(
       \frac{\varepsilon}{\lambda_{10}} \right)^4 \varsigma_{m n}^4 
       \breve{D}^k_{n m} dp\\
       &= \varepsilon^3 \sum_j \sum_{m, n \in \{ 0, 1 \}}
       \int_{\mathbb{R}^3 \backslash B_{W_j} (0, r_1)}  \check{D}^k_{j, m n} |
       W_j  \tilde{p}_j |^{- 4} \sigma_{m n}^{j, 4} d \tilde{p}_j + o
       (\varepsilon^3) .
    \end{eqnarray*}
  \end{case}

  \begin{case}
      There is only $c_{m n}^{3, b} (t, p)$ left. When $m > 1$ or $n > 1$,
      $\lambda_{n m}$ is bounded from below and we have
      \begin{align*}
          & \sum_{m > 1 \infixor n > 1} \int_{\Gamma r \cup \Gamma_m}
        \breve{D}^k_{n m} (t, p) \left( \frac{\varepsilon}{\lambda_{10}}
        \right)^3 c_{m n}^{3, b} (t, p) dp\\
         = & \sum_{m > 1 \infixor n > 1} \int_{\Gamma^{\ast}} \breve{D}^k_{n
        m} (t, p)  \left( \frac{\varepsilon}{\lambda_{10}} \right)^3 c_{m
        n}^{3, b} (t, p) dp + o (\varepsilon^3)\\
         = & - \varepsilon^3 \sum_{m > 1 \infixor n > 1}
        \int_{\Gamma^{\ast}} \frac{f_{mn} (p) E' (t) \cdummy D_{mn}  (p)
        D^k_{n m} (p)}{\lambda_{n m} (p)^3} dp + o (\varepsilon^3) .
      \end{align*}
      Thanks to \eqref{D symmetry} and the definition of $f_{mn} (p)$ given by
      \eqref{f def}, we finally get
      \begin{align*}
        & \sum_{m > 1 \infixor n > 1} \int_{\Gamma r \cup \Gamma_m}
        \breve{D}^k_{n m} (t, p) \left( \frac{\varepsilon}{\lambda_{10}}
        \right)^3 c_{m n}^{3, b} (t, p) dp\\
        = & - \varepsilon^3 \sum_{n \geqslant 2} \int_{\Gamma^{\ast}}
        c_{00}^{i n} (p) \frac{(E' (t) \cdummy D_{0 n}  (p) D^k_{n 0} (p) + E'
        (t) \cdummy D_{n 0}  (p) D^k_{0 n} (p))}{\lambda_{n 0} (p)^3} dp + o
        (\varepsilon^3)\\
        = & - 2 \varepsilon^3 \sum_{n \geqslant 2} \int_{\Gamma^{\ast}}
        c_{00}^{i n} (p) \frac{E' (t) \cdummy \tmop{Re} (D_{0 n}  (p) D^k_{n
        0} (p))}{\lambda_{n 0} (p)^3} dp + o (\varepsilon^3) .
      \end{align*}
      
  \end{case}

  \begin{case}
        Now we consider the cases of $0 \leqslant n, m \leqslant 1$. In fact,
        we only need to consider $m = 0, n = 1$ and $m = 1, n = 0$ because
        $c_{00}^{3, b} = c_{11}^{3, b} = 0$. In these cases, $\lambda_{n m}$
        is small near Weyl nodes. Let
        \begin{eqnarray*}
          L_0 & = & \int_{\Gamma r \cup \Gamma_m} \breve{D}^k_{10} (t, p)
          \left( \frac{1}{\lambda_{10}} \right)^3 c_{01}^{3, b} (t, p) dp +
          \int_{\Gamma r \cup \Gamma_m} \breve{D}^k_{01} (t, p) \left(
          \frac{1}{\lambda_{10}} \right)^3 c_{10}^{3, b} (t, p) dp\\
          & = & 2 \int_{P_t^{- 1} \left( \Gamma_r \bigcup \Gamma_m \right)}
          \frac{c_{00}^{i n} (p) E' (t) \cdummy \tmop{Re} (D_{01}  (p) D^k_{1
          0}  (p))}{\lambda_{10} (p)^3} dp.
        \end{eqnarray*}
        When $\varepsilon \rightarrow 0 +$, $L_0$ will blow up. The
        singularity comes from Weyl nodes $p_j^0$. We will separate the regular part and singular part. In other words, we can rewritten as
        \begin{equation}
          L_0 = L_r + L_s, \label{L0}
        \end{equation}
        with
        \begin{eqnarray*}
          L_r & = & 2 \int_{\Gamma^{\ast}} E' (t) \cdummy \tmop{Re} \left(
          \frac{D_{01}  (p) D^k_{1 0}  (p)}{\lambda_{10} (p)^3} c_{00}^{i n}
          (p) \right.\\
          & - & \left. \sum_j I_{B_{W_j} (p_j^0, r_2)} (p) c_{00}^{i n}
          (p_j^0) \frac{D_{j, 01}  (p - p_j^0) D^k_{j, 10}  (p - p_j^0)}{| W_j
          (p - p_j^0) |^3} \right) d p + o (1),\\
          L_s & = & \sum_{m, n \in \{ 0, 1 \}} \sum_j \int_{P_t \left( B_{W_j}
          \left( p_j^0, \frac{r_2}{\varepsilon} \right) \right) \backslash
          B_{W_j} (p_j^0, r_1)} c_{00}^{i n} (p_j^0) s_{m n}^{j, 3, b} (t,
          \tilde{p}_j) \check{D}^{j, k}_{m n}  (t, \tilde{p}_j) d \tilde{p}_j
          .
        \end{eqnarray*}
        The details are in Appendix \ref{C4}. Here, $L_r$ is the regular part and $L_s$ is the singular part, which can be merged with the integral on $\Gamma_s$ and become integrable. (See Remark \ref{rmk integrable}.)
      \end{case}

  Combine Case 1 to Case 5 and we get
      \begin{align*}
        J_k (t) = & \varepsilon^3 \sum_j \sum_{m, n \in \{ 0, 1 \}}
        \int_{\mathbb{R}^3 \backslash B_{W_j} (0, r_1)}  \check{D}^k_{j, n m}
        (| W_j  \tilde{p}_j |^{- 2} s_{m n}^{j, 2, a} \nobracket\\
        + & \left. | W_j  \tilde{p}_j |^{- 3}  \left( s_{m n}^{j, 3, a} +
        I_{P_t \left( B_{W_j} \left( p_j^0, \frac{r_2}{\varepsilon} \right)
        \right)} (\tilde{p}_j) s_{m n}^{j, 3, b} + s_{m n}^{j, 3, c} \right) +
        | W_j  \tilde{p}_j |^{- 4} \sigma_{m n}^{j, 4} \right) d \tilde{p}_j\\
        + & \varepsilon^3 \sum_j \sum_{m, n \in \{ 0, 1 \}} \int_{B_{W_j} (0,
        r_1)} \check{D}^k_{j, m n} (| W_j  \tilde{p}_j |^{- 2} s_{m n}^{j, 2,
        a} + | W_j  \tilde{p}_j |^{- 3} s_{m n}^{j, 3} + | W_j  \tilde{p}_j
        |^{- 4} \sigma_{m n}^{j, 4}) d \tilde{p}_j\\
        + & J_{k, 1} (t) + o (\varepsilon^3)\\
        = & \varepsilon^3 \sum_j \sum_{m, n \in \{ 0, 1 \}} \int_{B_{W_j} (0,
        r_1)}  \check{D}^k_{j, n m} (\check{s}_{m n}^j - (s_{m n}^{j, 0} + |
        W_j  \tilde{p}_j |^{- 1} s_{m n}^{j, 1} + | W_j  \tilde{p}_j |^{- 2}
        s_{m n}^{j, 2, b} \nobracket \nobracket\\
        + & \left. \left. I_{\mathbb{R}^3 \backslash P_t \left( B_{W_j} \left(
        p_j^0, \frac{r_2}{\varepsilon} \right) \right)} (\tilde{p}_j) s_{m
        n}^{j, 3, b} \right) \right) d \tilde{p}_j + J_{k, 1} (t) + o
        (\varepsilon^3) .
      \end{align*}

  Recall that $J_{k, 1}$ is defined by \eqref{J1 def}. It is obvious that
  \[ \int_{\mathbb{R}^3} \check{D}^k_{m n}  | W_j  \tilde{p}_j |^{- 1} s_{m
     n}^{j, 1} d \tilde{p}_j = 0. \]
  We can also show that (See Appendix \ref{E5})
  \begin{equation}
    \int_{\mathbb{R}^3} \check{D}^k_{m n} s_{m n}^{j, 0} d \tilde{p}_j =
    \int_{\mathbb{R}^3} \check{D}^k_{m n} | W_j  \tilde{p}_j |^{- 2} s_{m
    n}^{j, 2, b} d \tilde{p}_j = 0. \label{s2b vanish}
  \end{equation}
  Here, the integral should be understand as the principle value integral.
  Recall that $J_{k, 2}$ is defined by \eqref{J2 def}. It follows that
  \begin{align*}
    & \varepsilon^3 \sum_j \sum_{m, n \in \{ 0, 1 \}} \int_{B_{W_j} (0, r_1)}
    \check{D}^k_{j, n m} (\check{s}_{m n}^j - (s_{m n}^{j, 0} + | W_j 
    \tilde{p}_j |^{- 1} s_{m n}^{j, 1} + | W_j  \tilde{p}_j |^{- 2} s_{m
    n}^{j, 2, b} \nobracket \nobracket\\
    + & \left. \left. I_{\mathbb{R}^3 \backslash P_t \left( B_{W_j} \left(
    p_j^0, \frac{r_2}{\varepsilon} \right) \right)} (\tilde{p}_j) s_{m n}^{j,
    3, b} \right) \right) d \tilde{p}_j\\
    = & \varepsilon^3 \sum_{m, n \in \{ 0, 1 \}} \int_{\mathbb{R}^3}
    \check{D}^k_{j, n m}  \left( \check{s}_{m n}^j - I_{\mathbb{R}^3
    \backslash P_t \left( B_{W_j} \left( p_j^0, \frac{r_2}{\varepsilon}
    \right) \right)} (\tilde{p}_j) s_{m n}^{j, 3, b} \right) d \tilde{p}_j\\
    = & \varepsilon^3 \int_{\mathbb{R}^3} \sum_{m, n \in \{ 0, 1 \}} \left(
    D^k_{j, n m} s_{m n}^j - c_{00}^{i n} (p_j^0) I_{| W_j \tilde{p}_j | >
    \frac{r_2}{\varepsilon}} \frac{E' (t) \cdummy D_{j, m n} D^k_{j, n m}}{|
    W_j \tilde{p}_j |^3} \right) d \tilde{p}_j\\
    = & J_{k, 2} (t) .
  \end{align*}
  Therefore we can conclude that
  \begin{equation}
    J_k (t) = J_{k, 1} (t) + J_{k, 2} (t) + o (\varepsilon^3) . \label{J exp
    final}
  \end{equation}
  This finishes the proof.
\end{proof}

We finish this section with some remarks.

\begin{remark}
    Here we give more interpretations on $J_{k, 1} (t)$ and $J_{k, 2} (t)$.
\begin{itemize}
    \item As we mentioned in Remark \ref{rm5}, in the expression of $J_{k, 1} (t)$, we have only used
$D_{0 n}  (p)$ and $\lambda_{n 0} (p)$ except the last term. Note that they are
functions defined on the whole Brillouin zone, and these terms still exist when there are no Weyl nodes.
    \item $J_{k, 2} (t)$ completely depends on the local property of Weyl nodes. It is
nothing but the current related to the limit system \eqref{limit system}.
$s_{m n}^{j, 3, b}$ is not integrable because it does not decay fast enough at
infinity. To prevent the integral of $s_{m n}^{j, 3, b}$ on the region $|
\tilde{p}_j | > 1$ from blowing up as $\varepsilon \rightarrow 0$, we restrict it
on $| \tilde{p}_j | < \frac{r_2}{\varepsilon}$. It leads to the appearance of
$D_{j, m n}$ in the last term of $J_{k, 1} (t)$.
  \item The last term of $J_{k, 1} (t)$ depends on both the functions defined on the whole Brillouin zone, i.e. $D_{0 n}  (p)$ and $\lambda_{n 0} (p)$, and the properties of
each Weyl nodes through $D_{j, 10} (\tilde{p}^j) $. The existence of Weyl nodes
makes the integral of $\frac{D_{01}  (p) D^k_{1 0}  (p)}{\lambda_{10} (p)^3}$
singular and then we need to subtract the term $$\sum_j I_{\Gamma_m \bigcup
\Gamma_s} (p) \frac{D_{j, 01}  (p - p_j^0) D^k_{j, 10}  (p - p_j^0)}{| W_j (p
- p_j^0) |^3}$$ from it to make it regular. In fact, the latter term can be seen as the singular part of the former term.

\end{itemize}

\end{remark}

%% file: Chapter_6.tex
\section{Conclusion}

In this paper, we have developed an asymptotic analysis theory for the semiclassical model of Bloch electrons in the presence of Weyl nodes. The main result is a rigorous derivation of the asymptotic expansion of the current, which explicitly demonstrates the contribution from the singular points. The technical innovation lies in dealing with the multi-scale nature and singularities through a novel strategy of dividing the Brillouin zone into different regions and treating each appropriately. The quantitative estimates obtained justify the approximations used in the physics literature \cite{de_juan_quantized_2017} to explain experimental observations.

This work opens up several directions for future research. One is to consider more complex models incorporating electron interactions or disorder. Another is to extend the analysis to other physical observables beyond the current. It would also be interesting to apply the expansion techniques developed here to other multiscale problems with singular features. Overall, this paper provides a mathematical foundation for studying topological effects in quantum dynamics, which will lead to many fruitful avenues for future exploration.

%% file: Appendix_A.tex
\section{Some properties of $A$ and $D$}\label{A}

In this section, we will show some properties of $A$ and $D$. They are defined
by \eqref{A def} and \eqref{D def} and will appear repeatedly in the following
analysis.

We start with deriving useful identities involving them. Recall that
\begin{equation}
  H^0 (p) \Psi_n^0 (z, p) = \lambda_n (p) \Psi_n^0 (z, p),
\end{equation}
taking the gradient with respect to $p$ and taking the inner product with $\Psi_m^0 
(m \neq n)$ leads to
\begin{equation}
  \label{grad H} D_{mn} (p) - \mathi \lambda_m (p) A_{mn} (p) = - \mathi
  \lambda_n (p) A_{mn} (p) .
\end{equation}
Thus we have
\begin{equation}
  \label{AD relation} A_{mn} (p) = \frac{D_{mn} (p)}{\mathi \lambda_{mn} (p)}
  .
\end{equation}
This equation gives the relation between $D_{mn}$ and $A_{mn}$ when $m \neq
n$.

For the case where $m = n$, we do not have similar relations. If we repeat the
steps above, instead of \eqref{grad H}, we will get
\begin{equation}
  D_{nn} (p) - \mathi \lambda_n (p) A_{nn} (p) = \nabla \lambda_n (p) - \mathi
  \lambda_n (p) A_{nn} (p) .
\end{equation}
It follows that
\begin{equation}
  D_{nn} (p) = \nabla \lambda_n (p) .
\end{equation}
Although we do not get the relation between $A_{nn}$ and $D_{nn}$, this conclusion is also useful. A direct corollary is
\begin{equation}
  \label{delta mn} D_{mm} (p) - D_{nn} (p) = \nabla \lambda_{mn} (p) .
\end{equation}
For more information of $A_{nn}$, we can turn to the effective
Hamiltonian. The results have been shown in Appendix \ref{2.3}, but they are
valid only when $p$ is close to a Weyl node.

$A$ and $D$ also have good symmetry with respect to $p$. From the eigenvalue problem,
we know that
\begin{equation}
  \Psi^0_m (z, - p) = \bar{\Psi}^0_m (z, p),
\end{equation}
and
\begin{equation}
  \label{lambda symmetry} \lambda_m  (- p) = \lambda_m (p) .
\end{equation}
It follows that
\begin{equation}
  \label{A symmetry} A^k_{mn}  (- p) = A^k_{nm} (p) = \bar{A}^k_{mn} (p),
\end{equation}
and
\begin{equation}
  \label{D symmetry} D^k_{mn}  (- p) = - D^k_{nm} (p) = - \bar{D}^k_{mn} (p) .
\end{equation}
The symmetry will let a lot of terms cancel and simplify the calculations.



Regarding the periodicity in $p$, it is easy to check that we have the following relation,
\begin{eqnarray*}
  \lambda_n (p + Y) & = & \lambda_n (p),\\
  \Psi_n^0 (p + Y, z) & = & e^{- \mathi Y \cdummy z} f (p, z) .
\end{eqnarray*}
It follows that
\begin{eqnarray*}
  D_{m n} (p + Y) & = & D_{m n} (p),\\
  A_{m n} (p + Y) & = & A_{m n} (p) .
\end{eqnarray*}
In other words, $D_{n m}, A_{n m}, \lambda_{n m} $ are all periodical
functions on $\Gamma^{\ast}$.

%% file: Appendix_B.tex
\section{Some useful lemmas}\label{Osc int}

In this part, we will introduce some lemmas for
convergence of the integrals of interest.

\subsection{Fourier Integral}

To calculate the asymptotic expansion of $c_{m n}$, some oscillatory integrals
are involved, such as \eqref{cmn int}. More generally, we seek for a tool to
\begin{equation}
  \label{fourier integral} \int_a^b e^{\frac{\mathi}{\varepsilon} p (t)} q (t)
  dt.
\end{equation}
Here, the functions $p (t)$ and $q (t)$ satisfy the following properties:

\begin{question}
  \label{pq regularity}
  
  (i) $p (t)$ is real and $p' (t) > 0$ has a positive lower bound $\lambda$.
  q(t) can be real or complex.
  
  (ii) In $[a, b], p^{(m)} (t)$ and $q^{(m)} (t)$ are integrable, and $m$ is a
  nonnegative integer.
  
  (iii) There exists a constant $C$ independent of $\xi$, such that
  \begin{equation}
      | q^{(k)} (t) | + \left| \frac{p^{(k)} (s)}{p' (s)} \right| \leqslant
    C, \quad \forall k \leqslant m. \label{q regularity} 
  \end{equation}
\end{question}
Now let $P^s [p, q] (t)$ denote

\begin{equation}
  P^s [p, q] (t) = \left\{ \frac{1}{p' (t)}  \frac{d}{dt} \right\}^s \frac{q
  (t)}{p' (t)}, \quad s = 0, 1, \cdots, m.
\end{equation}

After change of variable $s=p(t)$, \eqref{fourier integral} becomes
\begin{equation*}
    \int_a^b e^{\frac{\mathi}{\varepsilon} p (t)} q (t)
  dt = \int_{p(a)}^{p(b)} e^{\frac{\mathi}{\varepsilon} s} \frac{q(p^{-1}(s))}{p'(p^{-1}(s))} ds
  dt
\end{equation*}

Notice that this is a fourier integral with no stationary points. Then we introduce the following lemma, which follows by integration by parts.
\begin{lemma}
  \label{asymptotic for fourier}Assume the conditions and notation above. Then
  \begin{equation}
    \begin{split}
      \int_a^b e^{\frac{\mathi}{\varepsilon} p (t)} q (t) dt & = \sum_{s =
      0}^{m - 2} (e^{\frac{\mathi}{\varepsilon} p (a)} P^s [p, q] (a) -
      e^{\frac{\mathi}{\varepsilon} p (b)} P^s [p, q] (b)) (\mathi
      \varepsilon)^{s + 1} + \delta_m (\varepsilon) \label{int est} .
    \end{split}
  \end{equation}
  The error term $\delta_m (\varepsilon)$ is bounded by
  \[ | \delta_m (\varepsilon) | \leqslant C \left( \frac{\varepsilon}{\lambda}
     \right)^m, \]
  with a constant $C$ independent of $\varepsilon$.
\end{lemma}

\subsection{Calculation of an oscillatory integral}

In this part we will treat a special oscillatory integral that we will meet
later. Specificly, we will consider the integral
\begin{equation}
  I_{jk} [f] (a, b, p) = \int_a^b T_{k j} (s, p) f (s, p) ds, \text{ } j \neq
  k \label{int osc} .
\end{equation}
Here, $T_{j k}$ is defined by
\begin{eqnarray}
  T_{jk} (t, p) & = & e^{\frac{\mathi}{\varepsilon}  \int_0^t
  \breve{\lambda}_{jk} (r, p) dr} . \label{T def} 
\end{eqnarray}
and $f (t, p)$ satisfy the following property:\\
\begin{equation}\label{assump:poly}
    \begin{split}
        &\text{ $f$ is the polynomial of the following variables:}\frac{\lambda_{10}}{\lambda_{m n}}, \{ (\lambda_{10})^{l - 1} \nabla_p^{(l)}
        \lambda_{m n} \}_{l \geqslant 1},\\
        &\{ (\lambda_{10})^l \nabla^{(l)}_p D_{m n}
        \}_{l \geqslant 0}, \{ (\lambda_{10})^{l + 1} \nabla^{(l)}_p A_{m n} \}_{l
        \geqslant 0}, \left\{ \frac{\varepsilon E^{(l)} (t)}{\lambda_{10}} \right\}_{l
        \geqslant 0},\frac{\lambda_{10}}{\breve{\lambda}_{m n}},\\
        &\{ (\lambda_{10})^{l- 1} \nabla_p^{(l)} \breve{\lambda}_{m n} \}_{l \geqslant 1},\{(\lambda_{10})^l \nabla^{(l)}_p \breve{D}_{m n} \}_{l \geqslant 0}, \{
        (\lambda_{10})^{l + 1} \nabla^{(l)}_p \breve{A}_{m n} \}_{l \geqslant 0}.
    \end{split}
\end{equation}

Thanks to Assumption \ref{regular asump}, it is obvious that each component of
this polynomial is bounded. Then it follows that $f (t, p)$ is bounded
independently of $\varepsilon$ and $p$. For any function $g(t, p)$, the time derivative of $\breve{g}$ is
\[ \frac{\partial}{\partial t}  \breve{g} (t, p) = \breve{g_t} (t, p) + \varepsilon E (t) \cdummy
   \nabla_p  \breve{g} (t, p) = \frac{\varepsilon E (t)}{\lambda_{10}} \cdummy
   (\lambda_{10} \nabla_p  \breve{g} (t, p)) . \]
Then we obtain that if $g$ satisfies \eqref{assump:poly}, then $\frac{\partial}{\partial t}  \breve{g} (t, p)$ also satisfies \eqref{assump:poly}. Thus it is also bounded and in fact, by
induction we obtain that $\left( \frac{\partial}{\partial t}  \right)^k \breve{g} (t,
p)$ is bounded independently of $\varepsilon$ and $p$ and so is
\[ \frac{1}{\breve{\lambda}_{jk} (t, p)} \left( \frac{\partial}{\partial t}
   \right)^k \int_0^t \breve{\lambda}_{jk} (r, p) dr \leqslant C''_k, \]
by the same argument. Then this integral satisfies the (ii) and (iii) of
Assumption \ref{pq regularity}. In $\Gamma_r$ we know that
\begin{equation}
  | \lambda_{jk} (p) | > C > 0, \text{ } \forall p \in \Gamma_r, \label{lower
  bound 1}
\end{equation}
for some constant $C$ independent of $\varepsilon$. For a Weyl node $p_j^0$,
in $B_{W_j} (p_j^0, r_2) \backslash B_{W_j} (p_j^0, \varepsilon r_1)$ we know
that
\begin{equation}
  | \lambda_{jk} (p) | > C | p - p_j^0 | > 0, \text{ } \forall p \in B_{W_j} (p_j^0,
  r_2) \backslash B_{W_j} (p_j^0, \varepsilon r_1), \label{lower bound 2}
\end{equation}
for some constant $C$ independent of $\varepsilon$, because we have proved
that $\eqref{lambda asym}$ near Weyl nodes. \eqref{lower bound 1} and
\eqref{lower bound 2} guarantee that this integral satisfies (i) of Assumption
\ref{pq regularity}. Therefore, we can apply Lemma \ref{asymptotic for
fourier}.

Let $P_{jk}^s [f]$ denote $P^s [\int_0^t \breve{\lambda}_{k j} (r, p) dr, f]$.
For $P_{jk}^s [f]$, we have

\begin{align*}
  P_{jk}^0 [f] (t, p) & = \frac{f (t, p)}{\breve{\lambda}_{k j} (t, p)},\\
  P_{jk}^1 [f] (t, p) & = \frac{\partial_t f (t, p)  \breve{\lambda}_{k j} (t,
  p) - \varepsilon E (t) \cdot \nabla_p  \breve{\lambda}_{k j} (t, p) f (t,
  p)}{\breve{\lambda}_{k j} (t, p)^3} .
\end{align*}

Now we can apply Lemma \ref{asymptotic for fourier} and get

\begin{lemma}
  \label{A2}Suppose Assumption \ref{regular asump} is satisfied. For the
  integral given by \eqref{int osc} with $j \neq k$, its asymptotic expansion
  is given by
  \begin{equation}
    I_{jk} [f] (a, b, p) = \sum_{s = 0}^{m - 2} \tilde{P}_{jk}^s [f] (t, p)
    |_{t = b}^a  (\mathi \varepsilon)^{s + 1} + \delta_{j k}^m [f]
    (\varepsilon) . \label{int exp}
  \end{equation}
  Here, we abuse the notation and let $\delta_{j k}^m [f] (\varepsilon)$
  denote the error term in Lemma \ref{asymptotic for fourier}. $\tilde{P}$ is
  defined by
  \begin{eqnarray*}
    \tilde{P}_{jk}^s [f] (t) & = & T_{k j} (t, p) P_{jk}^s [f] (t, p),
  \end{eqnarray*}
\end{lemma}

Let $m = 3$, we get the asymptotic expansion of $I_{jk} [f]$.

\subsection{A technical lemma on convergence}

This following lemma is used in Section \ref{4}, which is on the convergence of integrals.

\begin{lemma}
  \label{tricky} Suppose $F^{\varepsilon}$ is a smooth function on $B (p_j^0, r_2)
  \backslash B (p_j^0, \varepsilon r_1)$, and there exists $l \geqslant 2$ such
  that $F^{\varepsilon}$ satisfies the following two conditions:
\begin{enumerate}
    \item There exists a function $F$, such that $\forall \alpha \in \mathbb{N}^3$
  with $| \alpha | \leqslant \max (3 - l, 0)$,
  \begin{equation}
    \lim_{\varepsilon \rightarrow 0 +} \varepsilon^l \partial_p^{\alpha}
    F^{\varepsilon} (t, p_j^0 + \varepsilon p) = \partial_p^{\alpha} F (t, p),
    a.e. \quad \tmop{on} \quad B (p_j^0, r_2) \backslash B (p_j^0, \varepsilon
    r_1) . \label{cond1lem}
  \end{equation}
  \item There exists a constant $C$ independent of $\varepsilon$, such that
  $\forall \alpha \in \mathbb{N}^3$ with $| \alpha | \leqslant \max (3 - l,
  0)$,
  \begin{equation}
    | \varepsilon^l \partial_p^{\alpha} F^{\varepsilon} (t, p_j^0 + \tilde{p})
    | < C | \tilde{p} |^{- l - | \alpha |} . \label{cond2lem}
  \end{equation}
\end{enumerate}

  If $l = 2$, we need to additionally assume that
  \begin{equation}
    F^{\varepsilon} (t, p) = 0 \quad \tmop{on} \quad \partial B (p_j^0, r_2) .
    \label{cond3lem}
  \end{equation}
  Then
  \[ \lim_{\varepsilon \rightarrow 0 +} \varepsilon^{l - 3} \int_{B (p_j^0,
     r_2) \backslash B (p_j^0, \varepsilon r_1)} e^{\frac{\mathi}{\varepsilon}
     \lambda_{10} (p) t} F^{\varepsilon} (t, p) dp = \int_{\mathbb{R}^3
     \backslash B (0, r_1)} e^{\mathi | p | t} F (t, p) dp. \]
  If $l = 2$, the right-hand side should be understood as
  \[ \int_{\mathbb{R}^3 \backslash B (0, r_1)} e^{\mathi | p | t} F (t, p) dp
     = \lim_{\varepsilon \rightarrow 0 +} \int_{\mathbb{R}^3 \backslash B (0,
     r_1)} e^{\mathi | p | t - \varepsilon | p |^2} F (t, p) dp. \]
\end{lemma}

\begin{proof}
  From \eqref{change vari}, we know that
  \begin{eqnarray*}
    &  & \varepsilon^{l - 3} \int_{B (p_j^0, r_2) \backslash B (p_j^0,
    \varepsilon r_1)} e^{\frac{\mathi}{\varepsilon} \lambda_{10} (p) t}
    F^{\varepsilon} (t, p) dp\\
    & = & \varepsilon^{l - 3} \int_{\mathbb{S }^2_{W_j}} \int_{\lambda_{10}
    (p_j^0 + \varepsilon r_1 \omega)}^{\lambda_{10} (p_j^0 + r_2 \omega)}
    e^{\frac{\mathi}{\varepsilon} q t} F^{\varepsilon} (t, p_j^0 + r_j (q,
    \omega) \omega) r_j (q, \omega)^2 \frac{\partial}{\partial q} r_j (q,
    \omega) dq d \omega\\
    & = & \int_{\mathbb{S }^2_{W_j}} \int_{\lambda_{10} (p_j^0 + \varepsilon
    r_1 \omega) / \varepsilon}^{\lambda_{10} (p_j^0 + r_2 \omega) /
    \varepsilon} e^{\mathi \tilde{q} t} (\varepsilon^l F^{\varepsilon} (t,
    p_j^0 + r_j (\varepsilon \tilde{q}, \omega) \omega)) \left( \frac{r_j
    (\varepsilon \tilde{q}, \omega)}{\varepsilon} \right)^2
    \frac{\partial}{\partial q} r_j (\varepsilon \tilde{q}, \omega) d
    \tilde{q} d \omega
  \end{eqnarray*}
  If $l > 3$, we know that
  \[ (\varepsilon^l F^{\varepsilon} (t, p_j^0 + r_j (\varepsilon \tilde{q},
     \omega) \omega)) \left( \frac{r_j (\varepsilon \tilde{q},
     \omega)}{\varepsilon} \right)^2 \leqslant C | \tilde{q} |^{2 - l} . \]
  by the dominated convergence theorem, we get
  \begin{eqnarray*}
    \lim_{\varepsilon \rightarrow 0 +} \varepsilon^{l - 3} \int_{B (p_j^0,
    r_2) \backslash B (p_j^0, \varepsilon r_1)} e^{\frac{\mathi}{\varepsilon}
    \lambda_{10} (p) t} F^{\varepsilon} (t, p) dp & = & \int_{\mathbb{S
    }^2_{W_j}} \int_{r_1}^{+ \infty} e^{\mathi \tilde{q} t} (F (t, \tilde{q}
    \omega)) \tilde{q}^2 d \tilde{q} d \omega\\
    & = & \int_{\mathbb{R}^3 \backslash B (0, r_1)} e^{\mathi | p | t} F (t,
    p) dp.
  \end{eqnarray*}
  If $2 < l \leqslant 3$, by integrating by parts, we obtain
  \begin{eqnarray*}
    &  & \varepsilon^{l - 3} \int_{B (p_j^0, r_2) \backslash B (p_j^0,
    \varepsilon r_1)} e^{\frac{\mathi}{\varepsilon} \lambda_{10} (p) t}
    F^{\varepsilon} (t, p) dp\\
    & = & \int_{\mathbb{S }^2_{W_j}} \int_{\lambda_{10} (p_j^0 + \varepsilon
    r_1 \omega) / \varepsilon}^{\lambda_{10} (p_j^0 + r_2 \omega) /
    \varepsilon} e^{\mathi \tilde{q} t} (\varepsilon^l F^{\varepsilon} (t,
    p_j^0 + r_j (\varepsilon \tilde{q}, \omega) \omega)) \left( \frac{r_j
    (\varepsilon \tilde{q}, \omega)}{\varepsilon} \right)^2
    \frac{\partial}{\partial q} r_j (\varepsilon \tilde{q}, \omega) d
    \tilde{q} d \omega\\
    & = & \frac{1}{\mathi t} \int_{\mathbb{S }^2_{W_j}} e^{\mathi \tilde{q}
    t} (\varepsilon^l F^{\varepsilon} (t, p_j^0 + r_j (\varepsilon \tilde{q},
    \omega) \omega)) \left. \left( \frac{r_j (\varepsilon \tilde{q},
    \omega)}{\varepsilon} \right)^2 \frac{\partial}{\partial q} r_j
    (\varepsilon \tilde{q}, \omega) \right|_{\tilde{q} = \lambda_{10} (p_j^0 +
    \varepsilon r_1 \omega) / \varepsilon}^{\lambda_{10} (p_j^0 + r_2 \omega)
    / \varepsilon} d \tilde{q} d \omega\\
    & - & \frac{1}{\mathi t} \int_{\mathbb{S }^2_{W_j}} \int_{\lambda_{10}
    (p_j^0 + \varepsilon r_1 \omega) / \varepsilon}^{\lambda_{10} (p_j^0 + r_2
    \omega) / \varepsilon} e^{\mathi \tilde{q} t} \frac{\partial}{\partial
    \tilde{q}} \left( (\varepsilon^l F^{\varepsilon} (t, p_j^0 + r_j
    (\varepsilon \tilde{q}, \omega) \omega)) \left( \frac{r_j (\varepsilon
    \tilde{q}, \omega)}{\varepsilon} \right)^2 \frac{\partial}{\partial q} r_j
    (\varepsilon \tilde{q}, \omega) \right) d \tilde{q} d \omega .
  \end{eqnarray*}
  For the first term, we can pass to the limit and get
  \begin{eqnarray*}
    &  & \lim_{\varepsilon \rightarrow 0 +} \frac{1}{\mathi t}
    \int_{\mathbb{S }^2_{W_j}} e^{\mathi \tilde{q} t} (\varepsilon^l
    F^{\varepsilon} (t, p_j^0 + r_j (\varepsilon \tilde{q}, \omega) \omega))
    \left. \left( \frac{r_j (\varepsilon \tilde{q}, \omega)}{\varepsilon}
    \right)^2 \frac{\partial}{\partial q} r_j (\varepsilon \tilde{q}, \omega)
    \right|_{\tilde{q} = \lambda_{10} (p_j^0 + \varepsilon r_1 \omega) /
    \varepsilon}^{\lambda_{10} (p_j^0 + r_2 \omega) / \varepsilon} d \tilde{q}
    d \omega\\
    & = & \frac{1}{\mathi t} \int_{\mathbb{S }^2_{W_j}} e^{\mathi \tilde{q}
    t} (\varepsilon^l F^{\varepsilon} (t, p_j^0 + r_j (\varepsilon \tilde{q},
    \omega) \omega)) | \tilde{q} |^2 \nosymbol {\bigverbar_{\tilde{q} =
    r_1}^{+ \infty}}  d \tilde{q} d \omega
  \end{eqnarray*}
  For the second term, we have
  \begin{eqnarray*}
    &  & \int_{\mathbb{S }^2_{W_j}} \int_{\lambda_{10} (p_j^0 + \varepsilon
    r_1 \omega) / \varepsilon}^{\lambda_{10} (p_j^0 + r_2 \omega) /
    \varepsilon} e^{\mathi \tilde{q} t} \frac{\partial}{\partial \tilde{q}}
    \left( (\varepsilon^l F^{\varepsilon} (t, p_j^0 + r_j (\varepsilon
    \tilde{q}, \omega) \omega)) \left( \frac{r_j (\varepsilon \tilde{q},
    \omega)}{\varepsilon} \right)^2 \frac{\partial}{\partial q} r_j
    (\varepsilon \tilde{q}, \omega) \right) d \tilde{q} d \omega\\
    & = & \int_{\mathbb{S }^2_{W_j}} \int_{\lambda_{10} (p_j^0 + \varepsilon
    r_1 \omega) / \varepsilon}^{\lambda_{10} (p_j^0 + r_2 \omega) /
    \varepsilon} e^{\mathi \tilde{q} t} \varepsilon^l \omega \cdummy \nabla_p
    F^{\varepsilon} (t, p_j^0 + r_j (\varepsilon \tilde{q}, \omega) \omega)
    \left( \frac{r_j (\varepsilon \tilde{q}, \omega)}{\varepsilon} \right)^2
    \left( \frac{\partial}{\partial q} r_j (\varepsilon \tilde{q}, \omega)
    \right)^2 d \tilde{q} d \omega\\
    & + & 2 \int_{\mathbb{S }^2_{W_j}} \int_{\lambda_{10} (p_j^0 +
    \varepsilon r_1 \omega) / \varepsilon}^{\lambda_{10} (p_j^0 + r_2 \omega)
    / \varepsilon} e^{\mathi \tilde{q} t} \varepsilon^l F^{\varepsilon} (t,
    p_j^0 + r_j (\varepsilon \tilde{q}, \omega) \omega) \left( \frac{r_j
    (\varepsilon \tilde{q}, \omega)}{\varepsilon} \right) \left(
    \frac{\partial}{\partial q} r_j (\varepsilon \tilde{q}, \omega) \right)^2
    d \tilde{q} d \omega\\
    & + & \int_{\mathbb{S }^2_{W_j}} \int_{\lambda_{10} (p_j^0 + \varepsilon
    r_1 \omega) / \varepsilon}^{\lambda_{10} (p_j^0 + r_2 \omega) /
    \varepsilon} e^{\mathi \tilde{q} t} \varepsilon^l F^{\varepsilon} (t,
    p_j^0 + r_j (\varepsilon \tilde{q}, \omega) \omega) \left( \frac{r_j
    (\varepsilon \tilde{q}, \omega)}{\varepsilon} \right) r_j (\varepsilon
    \tilde{q}, \omega) \frac{\partial^2}{\partial q^2} r_j (\varepsilon
    \tilde{q}, \omega) d \tilde{q} d \omega .
  \end{eqnarray*}
  By \eqref{cond2lem} and the estimate $| r_j (\varepsilon \tilde{q}, \omega)
  | \leqslant \varepsilon \tilde{q}$, we can apply the dominated convergence
  theorem to all three integrals and get
  \begin{eqnarray*}
    &  & \lim_{\varepsilon \rightarrow 0 +} \int_{\mathbb{S }^2_{W_j}}
    \int_{\lambda_{10} (p_j^0 + \varepsilon r_1 \omega) /
    \varepsilon}^{\lambda_{10} (p_j^0 + r_2 \omega) / \varepsilon} e^{\mathi
    \tilde{q} t} \frac{\partial}{\partial \tilde{q}} \left( (\varepsilon^l
    F^{\varepsilon} (t, p_j^0 + r_j (\varepsilon \tilde{q}, \omega) \omega))
    \left( \frac{r_j (\varepsilon \tilde{q}, \omega)}{\varepsilon} \right)^2
    \frac{\partial}{\partial q} r_j (\varepsilon \tilde{q}, \omega) \right) d
    \tilde{q} d \omega\\
    & = & \int_{\mathbb{S }^2_{W_j}} \int_{r_1}^{+ \infty} e^{\mathi
    \tilde{q} t} \omega \cdummy \nabla_p F (t, \tilde{q} \omega) | \tilde{q}
    |^2 d \tilde{q} d \omega + 2 \int_{\mathbb{S }^2_{W_j}} \int_{r_1}^{+
    \infty} e^{\mathi \tilde{q} t} F (t, \tilde{q} \omega) \tilde{q} d
    \tilde{q} d \omega
  \end{eqnarray*}
  Then we get
  \begin{eqnarray*}
    &  & \lim_{\varepsilon \rightarrow 0 +} \varepsilon^{l - 3} \int_{B
    (p_j^0, r_2) \backslash B (p_j^0, \varepsilon r_1)}
    e^{\frac{\mathi}{\varepsilon} \lambda_{10} (p) t} F^{\varepsilon} (t, p)
    dp\\
    & = & - \frac{1}{\mathi t} \left( \int_{\mathbb{S }^2_{W_j}}
    \int_{r_1}^{+ \infty} e^{\mathi \tilde{q} t} \omega \cdummy \nabla_p F (t,
    \tilde{q} \omega) | \tilde{q} |^2 d \tilde{q} d \omega + 2 \int_{\mathbb{S
    }^2_{W_j}} \int_{r_1}^{+ \infty} e^{\mathi \tilde{q} t} F (t, \tilde{q}
    \omega) \tilde{q} d \tilde{q} d \omega \right)\\
    & + & \frac{1}{\mathi t} \int_{\mathbb{S }^2_{W_j}} e^{\mathi \tilde{q}
    t} (\varepsilon^l F^{\varepsilon} (t, p_j^0 + r_j (\varepsilon \tilde{q},
    \omega) \omega)) | \tilde{q} |^2 \nosymbol {\bigverbar_{\tilde{q} =
    r_1}^{+ \infty}}  d \tilde{q} d \omega\\
    & = & \int_{\mathbb{R}^3 \backslash B (0, r_1)} e^{\mathi | p | t} F (t,
    p) dp
  \end{eqnarray*}
  Similarly, if $l = 2$, we need to integrate by parts on more time and the
  conclusion is straightforward.
\end{proof}

%% file: Appendix_C.tex
\section{Effective Hamiltonian}

\subsection{Proof of Lemma \ref{eh lemma} and Corollary \ref{eh
cor}}\label{Weyl nodes}

Proving Lemma \ref{eh lemma} boils down to identifying equations that determine the eigenvalues and
eigenvectors. Recall $\Omega_0$ defined by \eqref{Omega0}, the projection
operator onto $\Omega_0$ is given by
\[ P_0 f = \langle f, \phi_1 \rangle \phi_1 + \langle f, \phi_2 \rangle \phi_2
   . \]
For a small perturbation $p$, the eigenvalue problem at the point $p_0 + p$ is
given by
\begin{equation}
  H^0  (p_0 + p) \Psi_n (p_0 + p) = \lambda_n (p_0 + p) \Psi_n (p_0 + p) .
\end{equation}
Let
\begin{align*}
      \lambda_n (p_0 + p)  = \lambda_n (p_0) + \lambda_- (p),& \, \lambda_{n + 1} (p_0 + p) =  \lambda_n (p_0) + \lambda_+ (p),\\
  \Psi_- (p) = \Psi_n (p_0 + p),&\, \Psi_+ (p) = \Psi_{n + 1} (p_0 + p).
\end{align*}
Then the eigenvalue problem can be rewritten as
\begin{equation}
  \label{eigen p} H^0  (p_0 + p) \Psi_{\pm} (p) = (\lambda_n + \lambda_{\pm}
  (p)) \Psi_{\pm} (p) .
\end{equation}
Here, it is obvious that $\lambda_+ \geqslant \lambda_-$. Let $\Omega (p)$ denote
the eigenspace spanned by $\Psi_{\pm} (p)$.

Now we will  calculate $\Psi_{\pm} (p)$ and $\lambda_{\pm} (p)$. Suppose that
\begin{equation}
  P_0 \Psi_{\pm} = \alpha_{\pm} \phi_1 + \beta_{\pm} \phi_2 .
\end{equation}
Using the method developed in {\cite{fefferman_honeycomb_2012}}, \eqref{eigen p} can be reformulated as
\begin{equation}
  \label{M equation} M (\lambda_{\pm}, p) \cdot (\alpha_{\pm}, \beta_{\pm})^T
  = 0,
\end{equation}
Here, $M (\lambda_{\pm}, p)$ is given by


\begin{equation}
  M (\lambda_{\pm}, p) = \left( \begin{array}{cc}
    \lambda_{\pm} + \langle \phi_1, 2 i p \cdot \nabla \phi_1 \rangle &
    \langle \phi_1, 2 ip \cdot \nabla \phi_2 \rangle\\
    \langle \phi_2, 2 ip \cdot \nabla \phi_1 \rangle & \lambda_{\pm} + \langle
    \phi_2, 2 ip \cdot \nabla \phi_2 \rangle
  \end{array} \right) + O (p^2) \label{M expression}
\end{equation}
By the Fredholm alternative theorem, we need the linear equation \eqref{M
equation} to have nontrivial solutions. Therefore, $\lambda_{\pm} (p)$ should
be the two roots of the following equation
\begin{equation}
  \label{det M} \det M (\lambda_{\pm}, p) = 0.
\end{equation}
$\alpha_{\pm} (p)$ and $\beta_{\pm} (p)$ are the solution of the following
equation,
\begin{equation}
  \label{M vec} M (\lambda_{\pm}, p) \cdot (\alpha_{\pm}, \beta_{\pm})^T = 0.
\end{equation}
Here, we said that \eqref{det M} has two solutions because it can be seen as a
quadratic equation with a small perturbation. Here we assume that
\begin{equation}
  \left( \begin{array}{cc}
    \langle \phi_1, 2 ip \cdot \nabla \phi_1 \rangle & \langle \phi_1, 2 ip
    \cdot \nabla \phi_2 \rangle\\
    \langle \phi_2, 2 ip \cdot \nabla \phi_1 \rangle & \langle \phi_2, 2 ip
    \cdot \nabla \phi_2 \rangle
  \end{array} \right) \neq c I. \label{split}
\end{equation}
then the rank of $M (\lambda_{\pm}, p)$ is 1 and $\lambda_+ \neq \lambda_-$.
In fact, the leading term of $M (\lambda_{\pm}, p)$ is given by
\begin{equation}
  M (\lambda_{\pm}, p) = \lambda_{\pm} I - H^e (p) + o (p) .
\end{equation}
Here, $H^e (p)$ the effective Hamiltonian defined by \eqref{eff ham0}(we will
introduce it after Lemma \ref{eh lemma}). Using another expression of $H^e
(p)$ given by \eqref{eff ham}, we can easily calculate the asymptotic
expansion of $\lambda_{\pm} (p)$, which is exactly given by \eqref{lambda
asym}. This finishes the proof of Lemma \ref{eh lemma}.

For the corollary \ref{eh cor}, thanks to \eqref{M expression}, it is obvious
that $\alpha_{\pm}, \beta_{\pm}$ can be approximated by $\alpha_{\pm}^0, \beta_{\pm}^0$,
which solves
\[ \left( \begin{array}{cc}
     \lambda_{\pm} + \langle \phi_1, 2 \mathi p \cdot \nabla \phi_1 \rangle &
     \langle \phi_1, 2 \mathi p \cdot \nabla \phi_2 \rangle\\
     \langle \phi_2, 2 \mathi p \cdot \nabla \phi_1 \rangle & \lambda_{\pm} +
     \langle \phi_2, 2 \mathi p \cdot \nabla \phi_2 \rangle
   \end{array} \right) \left(\begin{array}{c}
     \alpha_{\pm}^0\\
     \beta_{\pm}^0
   \end{array}\right) = 0. \]
\eqref{eq: eig vec asym} follows from straightforward calculations. This finishes the proof of Corollary \ref{eh cor}.

\subsection{Proof of Proposition \ref{eff result}}\label{2.3}

In this part, we will derive some useful identities and prove Proposition \ref{eff result}. Let
\begin{equation}
  \label{linear transformation} p':= (p'_1, p'_2, p'_3)^T  = Wp,
\end{equation}
Under the change of variable \eqref{linear transformation}, $H^e (p)$ given by
\eqref{eff ham} becomes
\begin{equation}
  \label{Effective H} H^e (p') = \frac{1}{2} \left[ \begin{array}{cc}
    p'_3 & p'_1 + ip'_2\\
    p'_1 - ip'_2 & - p'_3
  \end{array} \right] .
\end{equation}
Here we assume that $w_0 = 0$ without loss of generality. For the general
case, we only need an additional coordinate transformation.

First, we will list the terms we need and then we will calculate them respectively.
We will find an $A_{+ -}$
to approximate $A_{10}$ and $D_{+ -}$ to approximate $D_{10}$ (similarly for
other $A_{m n}$ and $D_{m n}$ with $m, n \in \{ 0, 1 \}$). 
Since only the difference of $A_{11}$ and
$A_{00}$ and that of $D_{11}$ and $D_{00}$ are needed, we use
\begin{equation}
  \Lambda_{+ -} = A_{+ +} - A_{- -}, \quad \Delta_{+ -} = D_{+ +} - D_{- -} . \label{Lambda def}
\end{equation}
to approximate $A_{11} - A_{00}$ and $D_{11} - D_{00}$ respectively. 
Besides these terms, the energy gap $\lambda_+ -
\lambda_-$ and its gradient are also important. We will also introduce the
berry curvature $\Omega_{+ -}$, which will be defined later. 

With effective Hamiltonian, $A_{nn}$ can be written as


\begin{equation}
  \begin{array}{ll}
    A_{+ +} (p) & = \langle \Psi^0_+, i \nabla_p \Psi^0_+ \rangle + O (p) .
  \end{array} \label{App}
\end{equation}
The expressions of $A_{- -}$ and $D_{+ -}$ are similar,
\begin{eqnarray}
  A_{- -} (p) & = & \langle \Psi^0_-, i \nabla_p \Psi^0_- \rangle + O (p), 
  \label{Amm}\\
  D_{+ -} (p) & = & \langle \Psi^0_+, \nabla_p H^e (p) \Psi^0_- \rangle + O
  (p) .  \label{Dpm}
\end{eqnarray}
Recalling that for each Weyl node, we have a matrix $W$ given by \eqref{W
def}. First we consider the case $W = I$. In this case, $H^e$ is given by
\eqref{Effective H}. Its eigenvalue is

\begin{align}
  \lambda_{\pm} (p) & = \pm \frac{r}{2},  \label{lpm asym}
\end{align}
with the sphere coordinate $(r, \phi, \theta)$. 
Applying the expression of the eigenvectors to \eqref{App}, \eqref{Amm} and
\eqref{Dpm}, straightforward calculation shows that
\begin{eqnarray}
  \label{Dpm asym} D_{+ -} & = & \frac{1}{2} \left( - \cos^2  \frac{\theta}{2} + \exp (- 2 i
  \phi) \sin^2  \frac{\theta}{2}, \right.\\
  & &\left. i \cos^2  \frac{\theta}{2} + i \exp (- 2 i
  \phi) \sin^2  \frac{\theta}{2}, \exp (- i \phi) \sin \theta \right)^T .\nonumber 
  \\
  \Lambda_{+ -} & = & \frac{2}{r}  (\tan \frac{\theta}{2} \sin \phi, - \tan
  \frac{\theta}{2} \cos \phi, 0) .  \label{pm asym}
\end{eqnarray}
Thanks to \eqref{grad H}, $A_{+ -}$ is given by
\begin{equation}
  \begin{split}
    A_{+ -} & = \frac{D_{+ -}}{i (\lambda_+ - \lambda_-)} = \frac{D_{+ -}}{ir}
    \label{A asymptotic} .
  \end{split}
\end{equation}
For $\Delta_{+ -}$ defined by \eqref{Lambda def}, we have
\begin{equation}
  \label{dlambda asymptotic st} \Delta_{+ -} = \nabla_p  (\lambda_+ -
  \lambda_-) = \frac{p}{r} .
\end{equation}
Currently, we have calculated $A_{+ -}, \Lambda_{+ -}, D_{+ -}, \Delta_{+
-}$,the energy gap $\lambda_+ - \lambda_-$ and its gradient. Next we calculate
Berry curvature. The Berry curvature, which is related to the Berry
connection, is defined by
\begin{equation}
  \label{omega def} \Omega_n = \nabla_p \times A_{nn},
\end{equation}
or
\begin{equation}
      \Omega^i_n = i \sum_{m \neq n} A_{nm} \times A_{mn}.
\end{equation}
The two definitions above are equivalent.





Now we focus on $\Omega^i_n$ around Weyl nodes with $m = n + 1$. For the
effective Hamiltonian, we can approximate the Berry curvature by defining 
\begin{equation}
  \Omega_+ = - \Omega_- = iA_{+ -} \times A_{- +} . \label{Omega def}
\end{equation}
Then Straightforward calculation gives the Berry curvature as
\begin{equation}
  \label{omega asymptotic} \Omega^i_+ (p) = \frac{p^i}{2 r^3} .
\end{equation}
The expression of $\Omega^i_+ (p)$ is the same as the result in
{\cite{berry_quantal_1984}}, which can be seen as a justification of our
calculation. Then we add some additional important conclusions. Let
$\delta^{ijk}$ denote $\delta^{ijk}_{123}$, where $\delta$ is the Kronecker
delta symbol. Noticing that $D_{+ -}$ is independent of $r$, straightforward
calculation shows that
\begin{equation}
  \label{delta jkl st} \int_{\mathbb{S}^2} \omega^j D^k_{+ -} (r \omega)
  D^l_{- +} (r \omega) d \omega = \frac{i \pi}{3} \delta^{jkl}, \forall r > 0,
\end{equation}
Besides, we can also show that
\begin{eqnarray}
  \int_{\mathbb{S}^2} D^k_{+ -} (r \omega) D^l_{- +} (r \omega) d \omega & = &
  \frac{2 \pi}{3} \delta^{kl}, \forall r > 0,  \label{delta kl st}\\
  \int_{\mathbb{S}^2} D^k_{- -} (r \omega) d \omega & = & 0, \forall r > 0. 
  \label{D-st}
\end{eqnarray}
These are the most important relations for our following calculations.
Currently, we have calculated $A_{+ -}, \Lambda_{+ -}$, $D_{+ -}, \Delta_{+ -}$,
the energy gap $\lambda_+ - \lambda_-$ and its gradient, the Berry curvature,
$\partial_{p^j} D_{+ -}^k$ and some important conclusions \ \eqref{delta jkl
st} to \eqref{D-st}.

Above we discussed the case where $W = I$. If $W \neq I$, we need to define
$B_W (p_0, r)$ as
\begin{equation}
  B_W (p_0, r) = \{ p : | W (p - p_0) | < r \}, \label{Bw def}
\end{equation}
And $\mathbb{S }^2_W$ as
\begin{equation}
  \mathbb{S }^2_W = \partial B_W  (0, 1) .
\end{equation}
After change of variable $p' = W (p - p_0)$, it returns to the case where $W =
I$. Therefore, it follows that the new ``$A_{+ -}$'' is given by
\begin{equation}
  \tilde{A}_{+ -} (p) = W^{- T} A_{+ -} \left( {W }  p \right) . \label{rule1}
\end{equation}
Here, we used the chain rule. New ``$D_{m n}$''s and ``$A_{m n}$''s with $m, n
\in \{ +, - \}$ are given in the similar way. New ``$\lambda_{\pm}$'' is given
by
\begin{equation}
  \tilde{\lambda}_{\pm} (p) = \lambda_{\pm} (W  p) . \label{rule2}
\end{equation}
Then straight forward calculation gives \eqref{gap} to \eqref{delta kl} and
finishes the proof of Proposition \ref{eff result}.

\

%% file: Appendix_D.tex
\section{Proof of Proposition \ref{c result}}\label{pf c result}

In this part, we will prove Proposition \ref{c result}. The result consists of
two parts. In the first part, we will calculate the asymptotic expansion of
$c$. In the second part, we will show the regularity of the remainder.

\subsection{Calculation of asymptotic analysis for c }\label{A.1}

We will first calculate the expansion of $d_{m n}^k$ given by
\eqref{dkmn asym}, and derive the expression of $c_{m n}^k$.

\subsubsection{Zeroth- and first-order terms}

For the zeroth-order terms, by \eqref{d0} and \eqref{dkmn asym}, it is
obvious that $d_{m n}^{0, 0} = c_{m n}^{i n}, \,d_{m n}^{0, k} = 0 \,(k \geqslant
1)$ and $\kappa_{m n}^0 = 0$.

For the first-order terms, we first consider $d^1_{n n}$. By \eqref{ckmn},
$d_{n n}^1$ is given by
\[ d_{n n}^1 = - \mathi \lambda_{10} \int^t_0 E (s) \cdummy \sum_p
   (\breve{A}_{n p} d^0_{pn} - d^0_{np}  \breve{A}_{pn}) (s, p) d s \text{},
\]
Since only $d^0_{00}$ does not vanish, this integral is zero. Then we can
obtain that
\begin{equation}
  d^1_{nn} = 0.
\end{equation}
For $d^1_{m n} (m \neq n)$, by $\eqref{ckmn}$, it is given by
\begin{equation}
  d^1_{m n} = - \mathi \lambda_{10} T_{m n} (t, p) \int^t_0 T_{n m} (s, p) E
  (s) \cdummy \sum_p (\breve{A}_{mp} d^0_{pn} - d^0_{mp}  \breve{A}_{pn}) (s,
  p) d s. \label{natural eq}
\end{equation}
Then we calculate its asymptotic expansion. Let
\begin{eqnarray*}
  F^1_{m n} (s, p) & = & \lambda_{10} E (s) \cdummy \sum_p (\breve{A}_{mp}
  d_{pn}^0 - d_{mp}^0  \breve{A}_{pn}) (s, p) .
\end{eqnarray*}
Noticing that all $d^0_{00}$ is not zero, $F_{m n}$ is given by
\[ F^1_{m n} (s, p) = - \lambda_{10} f_{mn} E (s) \cdummy \breve{A}_{mn}  (t,
   p), \]
with
\begin{eqnarray}
  f_{mn} & = & (\delta_{m 0} - \delta_{n 0}) c_{00}^{i n} . \label{f def} 
\end{eqnarray}
Then by \eqref{smooth}, $\| F_{m n}^1 \|_{W^{4, \infty}}$ is bounded uniformly
with respect to $\varepsilon$. Notice that $f_{m n}$ do not vanish only when
$m = 0, n > 0$ or $n = 0, m > 0$. Thus $d_{m n}^1 \neq 0$ only if $m = 0, n >
0$ or $n = 0, m > 0$\tmtextbf{}. Applying Lemma \ref{A2}, we get
\begin{eqnarray}
  d^1_{m n} & = & - \mathi T_{m n} (t, p) \sum_{s = 0}^1 \tilde{P}_{m n}^s
  [F^1_{m n}] (r, p) |_{r = t}^0  (\mathi \varepsilon)^{s + 1} - \mathi T_{m
  n} (t, p) \kappa_{m n}^3 [F^1_{m n}] (\varepsilon) . \label{intmn} 
\end{eqnarray}
$\tilde{P}_{m n}^s \infixand \kappa_{m n}^3$ are defined in Lemma \ref{A2}.
The error term $\kappa_{m n}^3 [f] (\varepsilon)$ is $O (\varepsilon^3 /
\lambda_{10} (p)^3)$ and we can discard it. $\tilde{P}_{m n}^0 [F^1_{m n}] (r,
p) |_{r = t}^0$ is $O (1)$ and given by
\[ \tilde{P}_{m n}^0 [F^1_{m n}] (r, p) |_{r = t}^0 = \lambda_{10} T_{n m} (t,
   p) E (t) \cdummy \breve{X}_{m n}  (t, p) - \lambda_{10} E (0) \cdummy X_{m
   n}  (p), \]
with
\begin{equation}
  X_{m n} (p) = \frac{f_{mn} A_{mn}  (p)}{\lambda_{n m} (p)}, \label{X def}
\end{equation}
and $\breve{X}_{m n}$ defined by \eqref{X def} and \eqref{alt u}. Notice that
$F_{m n}^1, X_{m n}$ do not vanish only when $m = 0, n > 0$ or $n = 0, m > 0$
because they contain $f_{m n}$. $\tilde{P}_{m n}^1 [F^1_{m n}] (r, p) |_{r =
t}^0$ contains high order terms. In fact, we have
\begin{eqnarray*}
  &  &  \tilde{P}_{m n}^1 [F^1_{m n}] (r, p) |_{r = t}^0  (\mathi
  \varepsilon)^2\\
  & = & \varepsilon^2  \frac{\partial_t F^1_{m n} (0, p) \lambda_{n m} (p) -
  \varepsilon E (t) \cdot \nabla_p \lambda_{n m} (p) F^1_{m n} (0,
  p)}{\lambda_{n m} (p)^3}\\
  & - & \varepsilon^2  \frac{\partial_t F^1_{m n} (t, p)  \breve{\lambda}_{n
  m} (t, p) - \varepsilon E (t) \cdot \nabla_p  \breve{\lambda}_{n m} (t, p)
  F^1_{m n} (t, p)}{\breve{\lambda}_{n m} (t, p)^3}\\
  & = & \varepsilon^2 T_{m n} (t, p) \frac{f_{mn} E' (t) \cdummy
  \breve{A}_{mn}  (t, p) \breve{\lambda}_{n m} (t, p)}{\breve{\lambda}_{n m}
  (t, p)^3} \\
  &- & \varepsilon^2 \lambda_{10} \frac{f_{mn} E' (0) \cdummy A_{mn} 
  (p) \lambda_{n m} (p)}{\lambda_{n m} (p)^3} + d^{1, r}_{m n},
\end{eqnarray*}
with
\begin{eqnarray*}
  d^{1, r}_{m n} & = & \varepsilon^2 T_{m n} (t, p) \frac{\varepsilon E (t)
  \cdot \nabla_p  \breve{\lambda}_{n m} (t, p) F^1_{m n} (t,
  p)}{\breve{\lambda}_{n m} (t, p)^3}\\
  & + & \varepsilon^2 \lambda_{10} T_{m n} (t, p) \frac{f_{mn} (\varepsilon E
  (t) \cdot \nabla_p) (E (t) \cdummy \breve{A}_{mn}  (t, p)) 
  \breve{\lambda}_{n m} (t, p)}{\breve{\lambda}_{n m} (t, p)^3}\\
  & - & \frac{\varepsilon^2 \lambda_{10} f_{mn} (\varepsilon E (0) \cdot
  \nabla_p) (E (0) \cdummy A_{mn}  (p)) \lambda_{n m} (p)}{\lambda_{n m}
  (p)^3} - \frac{\varepsilon^3 E (0) \cdot \nabla_p \lambda_{n m} (p) F^1_{m
  n} (0, p)}{\lambda_{n m} (p)^3}\\
  & = & O \left( \left( \frac{\varepsilon}{\lambda_{10}} \right)^3 \right) .
\end{eqnarray*}
$d^{1, r}_{m n}$ is $O ((\varepsilon / \lambda_{10})^3)$ and thus can be
discarded. Then recalling \eqref{dkmn asym}, the asymptotic expansion of $d_{m
n}^1 (m \neq n)$ is given by
\begin{eqnarray}
  d_{m n}^{1, 0} & = & 0, \nonumber\\
  d_{m n}^{1, 1} & = &  (\lambda_{10})^2 T_{m n} (t, p) E (0) \cdummy X_{m n} 
  (p) - (\lambda_{10})^2 E (t) \cdummy \breve{X}_{m n}  (t, p), \label{d11mn}
  \\
  d_{m n}^{1, 2} & = & \mathi (\lambda_{10})^3 T_{m n} (t, p) \frac{f_{mn} E'
  (0) \cdummy A_{mn}  (p) \lambda_{n m} (p)}{\lambda_{n m} (p)^3}\nonumber \\
  &-& \mathi
  (\lambda_{10})^3 \frac{f_{mn} E' (t) \cdummy \breve{A}_{mn}  (t, p)
  \breve{\lambda}_{n m} (t, p)}{\breve{\lambda}_{n m} (t, p)^3} . \nonumber
\end{eqnarray}
Finally, we can summary that
\begin{equation}
  d_{m n}^1 = d_{m n}^{1, 0} + \left( \frac{\varepsilon}{\lambda_{10}} \right)
  d_{m n}^{1, 1} + \left( \frac{\varepsilon}{\lambda_{10}} \right)^2 d_{m
  n}^{1, 2} + \kappa_{m n}^1, \label{d1mn asym}
\end{equation}
with
\[ \kappa_{m n}^1 = - (\varepsilon)^2 \mathi \lambda_{10} T_{m n} (t, p) d^{1,
   r}_{m n} - \mathi \lambda_{10} T_{m n} (t, p) \kappa_{m n}^3 [F_{m n}]
   (\varepsilon) = O ((\varepsilon / \lambda_{10})^3) \]
denoting the remainder.

\subsubsection{Second-order terms}

Now we can turn to the second-order terms. By \eqref{ckmn}, $d_{m n}^2$ is
given by
\[ d^2_{m n} = - \mathi \lambda_{10} T_{m n} (t, p) \int^t_0 T_{n m} (s, p) E
   (s) \cdummy \sum_p (\breve{A}_{mp} d^1_{pn} - d^1_{mp}  \breve{A}_{pn}) (s,
   p) d s \]
Substituting \eqref{d1mn asym} into the expression of $d^2_{m n}$ and noticing
that $\begin{array}{l}
  d_{m n}^{1, 0} = 0
\end{array}$, we will get
\[ d^2_{m n} = - \mathi \varepsilon T_{m n} (t, p) \int^t_0 T_{n m} (s, p) E
   (s) \cdummy \sum_p (\breve{A}_{mp} d^{1, 1}_{pn} - d^{1, 1}_{mp} 
   \breve{A}_{pn}) (s, p) d s + d_{m n}^{2, l}, \]
with
\begin{eqnarray*}
  d_{m n}^{2, l} & = & - \mathi \lambda_{10} T_{m n} (t, p) \int^t_0 T_{n m}
  (s, p) E (s) \cdummy \sum_p \left( \breve{A}_{mp} \left(  \left(
  \frac{\varepsilon}{\lambda_{10}} \right)^2 d_{p n}^{1, 2} + \kappa_{p n}^1
  \right) \right.\\
  & - & \left. \left(  \left( \frac{\varepsilon}{\lambda_{10}} \right)^2 d_{m
  p}^{1, 2} + \kappa_{m p}^1 \right)  \breve{A}_{pn} \right) (s, p) d s.
\end{eqnarray*}
$d_{m n}^{2, l} = O ((\varepsilon / \lambda_{10})^2)$ should be discarded.
Substituting \eqref{d11mn}, the expression of $d^{1, 1}_{pn}$, into the first
term of $d_{m n}^2$, we get
\begin{align*}
  & - \mathi \varepsilon T_{m n} (t, p) \int^t_0 T_{n m} (s, p) E (s) \cdummy
  \sum_p (\breve{A}_{mp} d^{1, 1}_{pn} - d^{1, 1}_{mp}  \breve{A}_{pn}) (s, p)
  d s\\
  = & \mathi \varepsilon (\lambda_{10})^2 T_{m n} (t, p) \int^t_0 T_{n m} (s,
  p)  \sum_p  (T_{m p} (s, p) E (0) \cdummy X_{m p}  (p) - E (s) \cdummy
  \breve{X}_{m p}  (s, p)) E (s) \cdummy \breve{A}_{pn} (s, p) d s\\
  - & \mathi \varepsilon (\lambda_{10})^2 T_{m n} (t, p) \int^t_0 T_{n m} (s,
  p)  \sum_p E (s) \cdummy \breve{A}_{mp}  (s, p) (T_{p n} (t, p) E (0)
  \cdummy X_{p n}  (p) - E (s) \cdummy \breve{X}_{p n}  (s, p)) d s\\
  = & \mathi \varepsilon (\lambda_{10})^2 T_{m n} (t, p) \int^t_0  \sum_p
  (\nobracket E (s) \cdummy (T_{n p} \breve{A}_{pn}) (s, p) E (0) \cdummy X_{m
  p} (p)\\
  - & E (s) \cdummy (T_{p m} \breve{A}_{mp}) (s, p) E (0) \cdummy X_{p n} (p)
  \nobracket) d s\\
  - & \mathi \varepsilon (\lambda_{10})^2 T_{m n} (t, p) \int^t_0 T_{n m} (s,
  p)  \sum_p  (E (s) \cdummy \breve{A}_{pn} (s, p) E (s) \cdummy \breve{X}_{m
  p}  (s, p) \nobracket\\
  - & E (s) \cdummy \breve{A}_{mp}  (s, p) E (s) \cdummy \breve{X}_{p n}  (s,
  p)  \nobracket) d s.
\end{align*}
Let
\[ F^2_{m n} (t, p) = (\lambda_{10})^3 \sum_p  (E (s) \cdummy \breve{A}_{pn}
   (s, p) E (t) \cdummy \breve{X}_{m p}  (s, p) - E (s) \cdummy \breve{A}_{mp}
   (s, p) E (s) \cdummy \breve{X}_{p n}  (s, p)) . \]
Then by \eqref{smooth}, $\| F^2_{m n} \|_{W^{2, \infty}}$ is bounded. When $m
\neq n$, again we use Lemma \ref{A2} and get
\begin{align*}
  & \varepsilon (\lambda_{10})^2 T_{m n} (t, p) \int^t_0 T_{n m} (s, p) 
  \sum_p (\nobracket E (s) \cdummy \breve{A}_{pn} (s, p) E (s) \cdummy
  \breve{X}_{m p}  (s, p)\\
  - & E (s) \cdummy \breve{A}_{mp}  (s, p) E (s) \cdummy \breve{X}_{p n}  (s,
  p) \nobracket) d s\\
  = &  (\varepsilon / \lambda_{10}) T_{m n} (t, p) \int^t_0 T_{n m} (s, p)
  F^2_{m n} (t, p) d s\\
  = & (\varepsilon / \lambda_{10}) T_{m n} (t, p) \kappa_{m n}^1 [F^2_{m n}]
  (\varepsilon)\\
  = & O \left( \left( \varepsilon / \lambda_{10} \right)^2  \right),
\end{align*}
is discarded. When $m = n$,
\[ \int^t_0 E (s) \cdummy \sum_p  (E (t) \cdummy \breve{X}_{n p}  (s, p)
   \breve{A}_{pn} (s, p) - \breve{A}_{n p}  (s, p) E (s) \cdummy \breve{X}_{p
   n}  (s, p)) = 0. \]
Similarly, let
\begin{eqnarray}
  F_{m p n}^3 (t, p) & = & (\lambda_{10})^3 E (t) \cdummy \breve{A}_{p n}  (t,
  p) E (0) \cdummy X_{m p}  (p), \label{F3 def} \\
  F_{m p n}^4 (t, p) & = & (\lambda_{10})^3 E (t) \cdummy \breve{A}_{mp}  (t,
  p) E (0) \cdummy X_{p n}  (p) . \label{F4 def} 
\end{eqnarray}
Then we can obtain,
\begin{eqnarray*}
  &  & \varepsilon (\lambda_{10})^2 T_{m n} (t, p) \int^t_0 E (s) \cdummy
  (T_{n p} \breve{A}_{pn}) (s, p) E (0) \cdummy X_{m p}  (p) d s\\
  & = &  (\varepsilon / \lambda_{10}) T_{m n} (t, p) \int^t_0  (T_{p m} F_{m
  p n}^3)  (s, p) d s\\
  & = & \left\{\begin{array}{l}
    O (\varepsilon / \lambda_{10}), n = p\\
    O \left( \left( \varepsilon / \lambda_{10} \right)^2  \right), n \neq p
  \end{array}\right.
\end{eqnarray*}
\[ \begin{array}{ll}
     & \varepsilon (\lambda_{10})^2 T_{m n} (t, p) \int^t_0 E (s) \cdummy
     (T_{p m} \breve{A}_{mp})  (s, p) E (0) \cdummy X_{p n}  (p) d s\\
     = &  (\varepsilon / \lambda_{10}) T_{m n} (t, p) \int^t_0  (T_{n p} F_{m
     p n}^4)  (s, p) d s\\
     = & \left\{\begin{array}{l}
       O (\varepsilon / \lambda_{10}), m = p\\
       O \left( \left( \varepsilon / \lambda_{10} \right)^2  \right), m \neq
       p
     \end{array}\right.
   \end{array} \]
Thus, we can conclude that
\begin{equation}
  d_{m n}^2 = d_{m n}^{2, 0} + \left( \frac{\varepsilon}{\lambda_{10}} \right)
  d_{m n}^{2, 1} + \kappa_{m n}^2, \label{d2mn asym}
\end{equation}
with
\begin{eqnarray*}
  d_{m n}^{2, 0} & = & 0,\\
  d_{m n}^{2, 1} & = & \mathi T_{m n} (t, p) \int^t_0  (F_{m n n}^3 (s, p) -
  F_{m m n}^4 (s, p)) d s,\\
  & = & \mathi T_{m n} (t, p) \int^t_0  (\lambda_{10})^3 E (s) \cdummy
  (\breve{A}_{n n} - \breve{A}_{m m}) (s, p) E (0) \cdummy X_{m n}  (p) d s,\\
  \kappa_{m n}^2 & = & d_{m n}^{2, l} - \mathi (\varepsilon / \lambda_{10})
  \sum_p \left( I_{m \neq n} T_{m n} (t, p) \int^t_0 T_{n m} (s, p) F^2_{m n}
  (t, p) d s \right)\\
  & + & \mathi (\varepsilon / \lambda_{10}) \sum_p \left( I_{m \neq p} T_{m
  n} (t, p) \int^t_0  (T_{p m} F_{m p n}^3)  (s, p) d s - I_{p \neq n} T_{m n}
  (t, p) \int^t_0  (T_{n p} F_{m p n}^4)  (s, p) d s \right)\\
  & = & O \left( \left( \varepsilon / \lambda_{10} \right)^2  \right) .
\end{eqnarray*}

\subsubsection{Third-order terms and summary}

For the third-order terms, it is easy to see that $d_{m n}^{3, 0} = 0$ because
$d_{m n}^{2, 0} = 0$. Then similar calculation shows that
\[ d_{m n}^3 = \kappa_{m n}^3 = O ((\varepsilon / \lambda_{10})^4) . \]
We have finished the asymptotic expansion of $d^k_{m n}$ and find each term in
\eqref{dkmn asym}. Then we can verify that $c^k_{m n}$ in the expansion
\eqref{repeat c asym} is given by \eqref{c0} to \eqref{c3} with the help of
\eqref{ckmn cal}.

Recall \eqref{repeat c asym},
\[  \breve{c}_{m n} = c_{m n}^0 + \left( \frac{\varepsilon}{\lambda_{10}}
   \right) c_{m n}^1 + \left( \frac{\varepsilon}{\lambda_{10}} \right)^2 c_{m
   n}^2 + \left( \frac{\varepsilon}{\lambda_{10}} \right)^3 c_{m n}^3 + \left(
   \frac{\varepsilon}{\lambda_{10}} \right)^4 \varsigma_{m n}^4, \]
with $\varsigma_{m n}^4$ given by
\begin{equation}
  \varsigma_{m n}^4 = R_{m n} + \kappa_{m n}^0 + \kappa_{m n}^1 + \kappa_{m
  n}^2 + \kappa_{m n}^3 . \label{remainder}
\end{equation}
Notice that in the process of calculation $d_{m n}^k$ by using \eqref{int
exp}, all the remainders $\kappa_{m n}^k (1 \leqslant k \leqslant 3)$ are $O
(\varepsilon^4 / \lambda_{10} (p)^4)$. Then we only need to estimate $R_{m n}$
to finish the proof of Lemma \ref{c result}.

\subsection{Regularity of the remainders}\label{3.5}

Now we will analyse the remainder $R_{m n}$. Recalling our expansion \eqref{d
exp},
\begin{equation}
  \breve{c}_{m n} = d_{m n}^0 + \left( \frac{\varepsilon}{\lambda_{10}}
  \right) d_{m n}^1 + \left( \frac{\varepsilon}{\lambda_{10}} \right)^2 d_{m
  n}^2 + \left( \frac{\varepsilon}{\lambda_{10}} \right)^3 d_{m n}^3 + \left(
  \frac{\varepsilon}{\lambda_{10}} \right)^4 R_{m n},
\end{equation}
with $d_{m n}^k$ determined by \eqref{d0} and \eqref{ckmn}. Substitude
\eqref{d exp} into \eqref{int eq}, we can get
\begin{align*}
   &d_{m n}^0 + \left( \frac{\varepsilon}{\lambda_{10}} \right) d_{m n}^1 +
  \left( \frac{\varepsilon}{\lambda_{10}} \right)^2 d_{m n}^2 + \left(
  \frac{\varepsilon}{\lambda_{10}} \right)^3 d_{m n}^3 + \left(
  \frac{\varepsilon}{\lambda_{10}} \right)^4 R_{m n}
  \\ = &\mathi \varepsilon \int^t_0  \left[ T_{m n} (s, p) E (s) \cdummy \sum_p
  (\breve{A}_{mp}  (s, p) \left( d_{p n}^0 + \left(
  \frac{\varepsilon}{\lambda_{10}} \right) d_{m n}^1 + \left(
  \frac{\varepsilon}{\lambda_{10}} \right)^2 d_{p n}^2 + \left(
  \frac{\varepsilon}{\lambda_{10}} \right)^3 d_{p n}^3 + \left(
  \frac{\varepsilon}{\lambda_{10}} \right)^4 R_{p n} \right) \right. \\  -&\left. \left( d_{m
  p}^0 + \left( \frac{\varepsilon}{\lambda_{10}} \right) d_{m p}^1 + \left(
  \frac{\varepsilon}{\lambda_{10}} \right)^2 d_{m p}^2 + \left(
  \frac{\varepsilon}{\lambda_{10}} \right)^3 d_{m p}^3 + \left(
  \frac{\varepsilon}{\lambda_{10}} \right)^4 R_{m p} \right)  \breve{A}_{pn} 
  (s, p) \right] d s + c_{m n}^0 (p_j^0) .
\end{align*}
By \eqref{d0} and \eqref{ckmn}, we can get the equation for $R_{m n}$,
\begin{eqnarray*}
  R_{m n}  =  & - \mathi \lambda_{10} \int^t_0 T_{m n} (s, p) E (s) \cdummy
  \sum_p (\breve{A}_{mp} d_{pn}^3 - d_{mp}^3  \breve{A}_{pn}) (s, p) d s \\ & -
  \mathi \varepsilon \int^t_0 T_{m n} (s, p) E (s) \cdummy \sum_p
  (\breve{A}_{mp} R_{p n} - R_{m p}  \breve{A}_{pn}) (s, p) d s.
\end{eqnarray*}
It follows that
\begin{eqnarray*}
  | R_{m n} (t, p) | & \leqslant & \sum_p \int^t_0 (| \lambda_{10} | (| E
  \cdummy \breve{A}_{mp} d_{pn}^3 | + | d_{mp}^3 E \cdummy \breve{A}_{pn} |)\\
  &+&
  \varepsilon (| E \cdummy \breve{A}_{mp} |  | R_{pn} | + | R_{mp} | | E
  \cdummy \breve{A}_{pn} |)) (s, p) d s.
\end{eqnarray*}
We can take the summation with respect to $m, n$ and get
\[ \sum_{m, n} | R_{m n} (t, p) | \leqslant C \sum_{m, n, p} \int^t_0 (|
   d_{pn}^3 | + | d_{mp}^3 |) (s, p) d s + C \int^t_0 \sum_{m, n} | R_{m n}
   (s, p) | d s, \]
By the Gronwall's inequality, we can obtain that
\begin{eqnarray*}
  \sum_{m, n} | R_{m n} (t, p) | & \leqslant & e^{C t} \sum_{m, n, p} \int^t_0
  (| d_{pn}^3 | + | d_{mp}^3 |) (s, p) d s\\
  & \leqslant & C.
\end{eqnarray*}
Also remember that in the last section, all the remainders $\kappa_{m n}^k (0
\leqslant k \leqslant 3)$ are $O ((\varepsilon / \lambda_{10})^4)$. Thus by we
get the regularity of the remainder $\varsigma_{m n}^4$,
\[ \varsigma_{m n}^4 \leqslant C. \]
This gives the estimate for the remainder and thus finishes the proof.

\begin{remark}
  \label{weaker assumption}In this proof, a key point is that all $A_{m n}$ is
  bounded when 
  $\max \{ m, n \} \geqslant 2$. This is a result from the first
  statement in Assumption \ref{regular asump}. In fact, we can use weaker
  assumptions. 
\end{remark}

%% file: Appendix_E.tex
\section{More details about the proof of Theorem \ref{J asym
exp}}\label{B4}

In this section, we will supplement some proof details of Theorem \ref{J asym exp}, which are omitted in the main text.

\subsection{Contribution of $c_{m n}^{2, b}$}\label{C1}

Similar to $c_{m n}^0$, for $c_{m n}^{2, b}$ we have
\begin{align*}
  & \sum_{m, n} \int_{\Gamma_r \cup \Gamma_m} \breve{D}^k_{n m} (t, p) \left(
  \frac{\varepsilon}{\lambda_{10} (p)} \right)^2 c_{m n}^{2, b} (t, p) dp\\
  +&
  \varepsilon^3 \sum_j \sum_{m, n \in \{ 0, 1 \}} \int_{B_{W_j} (0, r_1)} | 
  \tilde{p}_j |^{- 2} \check{D}^k_{j, m n} s_{m n}^{j, 2, b} d \tilde{p}_j\\
  = & \sum_{m, n} \int_{\Gamma^{\ast}} \breve{D}^k_{n m} (t, p)  \left(
  \frac{\varepsilon}{\lambda_{10} (p)} \right)^2 c_{m n}^{2, b} (t, p) dp + o
  (\varepsilon^3) .
\end{align*}
Recall that $c_{m n}^{2, b}$ is defined by \eqref{c2b def}. It follows that
\begin{eqnarray*}
  & &\sum_{m, n} \int_{\Gamma^{\ast}} \breve{D}^k_{n m} (t, p)  \left(
  \frac{\varepsilon}{\lambda_{10} (p)} \right)^2 c_{m n}^{2, b} (t, p) dp\\
  & =
  & - \varepsilon^2 \sum_{m, n} \int_{\Gamma^{\ast}} \breve{D}^k_{n m}  (t, p)
  E (t) \cdummy \breve{X}_{m n}  (t, p) dp\\
  & = & - \varepsilon^2 \sum_{m, n} \int_{\Gamma^{\ast}} D^k_{n m}  (p) E (t)
  \cdummy X_{m n}  (p) dp\\
  & = & \varepsilon^2 \sum_{m, n} \int_{\Gamma^{\ast}} \frac{f_{mn} (p) E (t)
  \cdummy D_{mn}  (p) D^k_{n m}  (p)}{\mathi \lambda_{m n} (p)^2} dp.
\end{eqnarray*}
According to \eqref{D symmetry} and the definition of $f_{mn} (p)$ given by
\eqref{f def}, we finally get
\begin{align*}
  & \sum_{m, n} \int_{\Gamma_r \cup \Gamma_m} \breve{D}^k_{n m} (t, p) \left(
  \frac{\varepsilon}{\lambda_{10} (p)} \right)^2 c_{m n}^{2, b} (t, p) dp\\
  + &
  \varepsilon^3 \sum_j \sum_{m, n \in \{ 0, 1 \}} \int_{B_{W_j} (0, r_1)} | 
  \tilde{p}_j |^{- 2} \check{D}^k_{j, m n} s_{m n}^{j, 2, b} d \tilde{p}_j\\
  = & \varepsilon^2 \sum_{n \geqslant 1} \int_{\Gamma^{\ast}} \frac{c_{00}^{i
  n} (p) (E (t) \cdummy D_{0 n}  (p) D^k_{n 0}  (p) - E (t) \cdummy D_{n 0} 
  (p) D^k_{0 n}  (p))}{\mathi \lambda_{n 0} (p)^2} dp + o (\varepsilon^3)\\
  = & 2 \varepsilon^2 \sum_{n \geqslant 1} \int_{\Gamma^{\ast}}
  \frac{c_{00}^{i n} (p) E (t) \cdummy \tmop{Im} (D_{0 n}  (p) D^k_{n 0} 
  (p))}{\lambda_{n 0} (p)^2} dp + o (\varepsilon^3) .
\end{align*}
\subsection{Limit of $L_2$}\label{C2}

Recall that the definition of $L_2$ is given by \eqref{L2 def}. We first
concentrate on the case with $n = 0, m = 1$. Since
\[ \breve{\lambda}_{10} (t, p) = \lambda_{10} (p) + \varepsilon \int_0^t E (s) d
   s \cdummy \int_0^1 \nabla_p \lambda_{10} \left( p + \alpha \varepsilon
   \int_0^t E (s) d s \right) d \alpha, \]
recalling \eqref{T def}, we can obtain that
\begin{eqnarray*}
  T_{10} (t, p) & = & e^{\frac{\mathi}{\varepsilon}  \int_0^t
  \breve{\lambda}_{10} (r, p) dr}\\
  & = & e^{\frac{\mathi}{\varepsilon} \lambda_{10} (p) t} g (t, p),
\end{eqnarray*}
with
\[ g (t, p) = \exp \left( \mathi \int_0^t E (s) d s \cdummy \int_0^1 \nabla_p
   \breve{\lambda}_{10} \left( p + \alpha \varepsilon \int_0^t E (s) d s \right)
   d \alpha \right) . \]
It follows that the term with $n = 0, m = 1$ becomes
\begin{align*}
  & \int_{\Gamma_m} T_{10} (t, p) E (0) \cdummy \breve{X}_{10}  (0, p)
  \breve{D}^k_{01} \chi_2 dp\\
  = & \int_{\Gamma_m} e^{\frac{\mathi}{\varepsilon} \lambda_{10} (p) t} g (t, p)
  E (0) \cdummy \breve{X}_{10}  (0, p) \breve{D}^k_{01} \chi_2 dp\\
  = & \sum_j \int_{B_{W_j} (p_j^0, r_2) \backslash B_{W_j} (p_j^0, \varepsilon
  r_1)} e^{\frac{\mathi}{\varepsilon} \lambda_{10} (p) t} g (t, p) E (0) \cdummy
  \breve{X}_{10}  (0, p) \breve{D}^k_{01} \chi_2 dp.
\end{align*}
Notice that $\breve{X}_{m n} (0, p) = X_{m n} (p)$. Thanks to \eqref{change
vari}, it follows that
\begin{eqnarray*}
  &  & \int_{\Gamma_m} T_{10} (t, p) E (0) \cdummy X_{10}  (p)
  \breve{D}^k_{01} (t, p) \chi_2 (p) dp\\
  & = & \sum_j \int_{B_{W_j} (p_j^0, r_2) \backslash B_{W_j} (p_j^0,
  \varepsilon r_1)} e^{\frac{\mathi}{\varepsilon} \lambda_{10} (p) t} g (t, p) E
  (0) \cdummy X_{10}  (p) \breve{D}^k_{01} (t, p) \chi_2 (p) dp\\
  & = & \mathi \sum_j \int_{B_{W_j} (p_j^0, r_2) \backslash B_{W_j} (p_j^0,
  \varepsilon r_1)} e^{\frac{\mathi}{\varepsilon} \lambda_{10} (p) t} (g
  \breve{D}^k_{01}) (t, p) \left( \frac{E (0) \cdummy D_{10} c_{00}^{i n}
  \chi_2}{\lambda_{10}^2} \right)  (p) d p.
\end{eqnarray*}
In the last step, we used the expression of $X_{1 0}$ given by \eqref{X def}
and the relation \eqref{AD relation}. Let
\[ F_j (q, \omega) = (g \breve{D}^k_{01}) (t, p_j^0 + r_j (q, \omega) \omega)
   (E (0) \cdummy D_{10} c_{00}^{i n} \chi_2)  (p_j^0 + r_j (q, \omega)
   \omega) . \]
It is easy to show that $F_j$ satisfies \eqref{cond1lem}, \eqref{cond2lem} and
\eqref{cond3lem} with $l = 2$. By the Lemma \ref{tricky}, we can conclude that
\begin{align*}
  & \int_{\Gamma_m} T_{10} (t, p) E (0) \cdummy X_{10}  (p) \breve{D}^k_{01}
  (t, p) c_{00}^{i n} (p) \chi_2 (p) dp\\
  = & \varepsilon \sum_j c_{00}^{i n} (p_j^0) \int_{\mathbb{R}^3 \backslash
  B_{W_j} (0, r_1)} \mathi \exp \left( \mathi \int_0^t | W_j \tilde{P}_s (p) |
  d s \right) \frac{E (0) \cdummy D_{j, 10}  (p) \check{D}^k_{j, 01} (t, p)}{|
  W_j p |^2} d p\\
  +& o (\varepsilon) .
\end{align*}
When $m = 0, n = 1$, the result is the conjugate of the expression above. When
$m = 0, n > 1$, we can show that such terms vanish as $\varepsilon \rightarrow
0 +$. Therefore we finish the proof of \eqref{limit L2}.

\subsection{Contribution of $c_{m n}^{3, a} \infixand c_{m n}^{3,
c}$}\label{C3}

Similarly to $c_{m n}^{2, b}$, for $c_{m n}^{3, a} + c_{m n}^{3, c}$, we have
\[ \sum_{m, n} \int_{\Gamma r \cup \Gamma_m} \left(
   \frac{\varepsilon}{\lambda_{10}} \right)^3 (c_{m n}^{3, a} + c_{m n}^{3,
   c})  \breve{D}^k_{n m} dp = L_3 + L_4, \]
with
\begin{eqnarray*}
  L_3 & = & \sum_{m, n} \int_{\Gamma_r} \left(
  \frac{\varepsilon}{\lambda_{10}} \right)^3 (c_{m n}^{3, a} + c_{m n}^{3, c})
  \breve{D}^k_{n m} \chi_1 dp\\
  & = & \varepsilon^3 \sum_{m, n} \int_{\Gamma_r} \mathi T_{m n} (t, p)
  \left( \frac{f_{mn} E' (0) \cdummy A_{mn}  (p) }{\lambda_{n m} (p)^2}
  \right.\\
  & + & \left. \int^t_0 E (s) \cdummy (\breve{A}_{n n} - \breve{A}_{m m}) (s,
  p) E (0) \cdummy X_{m n}  (p) d s \right) \breve{D}^k_{n m} (t, p) \chi_1
  (p) d p,\\
  L_4 & = & \sum_{m, n} \int_{\Gamma_m} \left(
  \frac{\varepsilon}{\lambda_{10}} \right)^3 (c_{m n}^{3, a} + c_{m n}^{3, c})
  \breve{D}^k_{n m} \chi_2 dp.
\end{eqnarray*}
Similar to $L_1$ defined by \eqref{L1 def}, it is easy to see that $L_3$ is $o
(\varepsilon^3)$. For $L_4$, by Lemma \ref{tricky}, we can conclude that
\[ L_4 = \varepsilon^3 \sum_j \sum_{m, n \in \{ 0, 1 \}} \int_{\mathbb{R}^3
   \backslash B_{W_j} (0, r_1)}  \check{D}^k_{j, m n} | W_j  \tilde{p}_j |^{-
   3}  (s_{m n}^{j, 3, a} + s_{m n}^{j, 3, c}) d \tilde{p}_j + o
   (\varepsilon^3) . \]
This finishes the proof of \eqref{L4 limit}.

\subsection{The analysis of $L_0$}\label{C4}

Therefore, we can let
\[ L_0 = L_r + L_s, \]
with
\begin{eqnarray*}
  L_r & = & 2 \int_{P_t^{- 1} \left( \Gamma_r \bigcup \Gamma_m \right)}
  c_{00}^{i n} (p) \tmop{Re} \left( \frac{E' (t) \cdummy D_{01}  (p) D^k_{1 0}
  (p)}{\lambda_{10} (p)^3} \right) d p\\
  & - & 2 \sum_j \int_{B_{W_j} (p_j^0, r_2) \backslash P_t^{- 1} (\Gamma_s)}
  c_{00}^{i n} (p_j^0) \tmop{Re} \left( \frac{E' (t) \cdummy D_{j, 01}  (p -
  p_j^0) D^k_{j, 10}  (p - p_j^0)}{| W_j (p - p_j^0) |^3} \right) dp,
\end{eqnarray*}
and
\[ L_s = 2 \sum_j \int_{B_{W_j} (p_j^0, r_2) \backslash P_t^{- 1} (\Gamma_s)}
   c_{00}^{i n} (p_j^0) \tmop{Re} \left( \frac{E' (t) \cdummy D_{j, 01}  (p -
   p_j^0) D^k_{j, 10}  (p - p_j^0)}{| W_j (p - p_j^0) |^3} \right) dp. \]
Recall that around the Weyl node $p_j^0$, we have
\begin{eqnarray*}
  D_{01}  (p_j^0 + \delta p) & = & D_{j, 01}  (\delta p) + O (| \delta p |),\\
  \lambda_{10} (p_j^0 + \delta p) & = & | W_j \delta p | + O (| \delta p |),\\
  c_{00}^{i n} (p_j^0 + \delta p) & = & c_{00}^{i n} (p_j^0) + O (| \delta p
  |) .
\end{eqnarray*}
Here $\delta p$ is a small perturbation. If follows that
\[ \frac{D_{01}  (p_j^0 + \delta p) D^k_{1 0}  (p_j^0 + \delta
   p)}{\lambda_{10} (p_j^0 + \delta p)^3} = \frac{D_{j, 01}  (\delta p)
   D^k_{j, 10}  (\delta p)}{| W_j \delta p |^3} + O (| \delta p |^{- 2}), \]
with the remainder integrable. $L_r$ is not singular and can be rewritten as
\begin{align*}
L_r  = & 2 \tmop{Re} E' (t) \cdummy \int_{P_t^{- 1} \left( \Gamma_r \bigcup
  \Gamma_m \right)} \left( c_{00}^{i n} (p) \frac{D_{01}  (p) D^k_{1 0} 
  (p)}{\lambda_{10} (p)^3} \right.\\
  - & \sum_j I_{B_{W_j} (p_j^0, r_2) \backslash P_t^{- 1} (\Gamma_s)} (p)
  c_{00}^{i n} (p_j^0) \left. \frac{D_{j, 01}  (p - p_j^0) D^k_{j, 10}  (p -
  p_j^0)}{| W_j (p - p_j^0) |^3} \right) dp\\
  = & 2 \tmop{Re} E' (t) \cdummy \int_{\Gamma^{\ast}} \left( c_{00}^{i n} (p)
  \frac{D_{01}  (p) D^k_{1 0}  (p)}{\lambda_{10} (p)^3}\right.\\
  - &\left.\sum_j c_{00}^{i n}
  (p_j^0) I_{B_{W_j} (p_j^0, r_2)} (p) \frac{D_{j, 01}  (p - p_j^0) D^k_{j,
  10}  (p - p_j^0)}{| W_j (p - p_j^0) |^3} \right) dp + o (1) .
\end{align*}
For $L_s$, we can rewrite it as
\begin{eqnarray*}
  L_s & = & 2 \tmop{Re} \sum_j c_{00}^{i n} (p_j^0) \int_{B_{W_j} (p_j^0,
  r_2)} I_{| W_j (P_t^{- 1} (p) - p_j^0) | > \varepsilon r_1} \frac{2 E' (t)
  \cdummy D_{j, 01}  (p - p_j^0) D^k_{j, 10}  (p - p_j^0)}{| W_j (p - p_j^0)
  |^3} dp\\
  & = & 2 \tmop{Re} \sum_j c_{00}^{i n} (p_j^0) \int_{P_t (B_{W_j} (p_j^0,
  r_2)) \backslash B_{W_j} (p_j^0, \varepsilon r_1)} \frac{2 E' (t) \cdummy
  \breve{D}_{j, 01}  (t, p - p_j^0) \breve{D}_{j, 10}^k  (t, p - p_j^0)}{| P_t
  (p) - p_j^0 |^3} dp\\
  & = & 2 \tmop{Re} \sum_j c_{00}^{i n} (p_j^0) \int_{P_t \left( B_{W_j}
  \left( p_j^0, \frac{r_2}{\varepsilon} \right) \right) \backslash B_{W_j}
  (p_j^0, r_1)} \frac{2 E' (t) \cdummy \check{D}_{j, 01}  (t, \tilde{p}_j)
  \check{D}^k_{j, 10}  (t, \tilde{p}_j)}{| P_t (\tilde{p}_j) |^3} d
  \tilde{p}_j .\\
  & = & 2 \tmop{Re} \sum_j c_{00}^{i n} (p_j^0) \int_{P_t \left( B_{W_j}
  \left( p_j^0, \frac{r_2}{\varepsilon} \right) \right) \backslash B_{W_j}
  (p_j^0, r_1)} s_{01}^{j, 3, b} (t, \tilde{p}_j) \check{D}^k_{j, 10}  (t,
  \tilde{p}_j) d \tilde{p}_j\\
  & = & \sum_{m, n \in \{ 0, 1 \}} \sum_j \int_{P_t \left( B_{W_j} \left(
  p_j^0, \frac{r_2}{\varepsilon} \right) \right) \backslash B_{W_j} (p_j^0,
  r_1)} c_{00}^{i n} (p_j^0) s_{m n}^{j, 3, b} (t, \tilde{p}_j) \check{D}^{j,
  k}_{m n}  (t, \tilde{p}_j) d \tilde{p}_j .
\end{eqnarray*}
This finishes the proof of \eqref{L0}.

\subsection{Vanishing terms}\label{E5}

In this part, we will prove \eqref{s2b vanish}. Recall that $s_{j, m n}^{2,
b}$ is the limit of $c_{m n}^{2, b}$ defined by \eqref{c2b def}. It follows
that
\begin{eqnarray*}
  & &\sum_{m, n \in \{ 0, 1 \}} \int_{\mathbb{R}^3} \check{D}^k_{m n} | W_j 
  \tilde{p}_j |^{- 2} s_{m n}^{j, 2, b} d \tilde{p}_j\\
  & = & \sum_{m, n \in \{
  0, 1 \}} \int_{\mathbb{R}^3}  \frac{f_{m n} (p_j^0) E (t) \cdummy
  \check{D}_{j, m n}  (t, \tilde{p}_j) \check{D}^k_{j, m n}  (t,
  \tilde{p}_j)}{| W_j P_t (\tilde{p}_j) |^2} d \tilde{p}_j\\
  & = & \sum_{m, n \in \{ 0, 1 \}} \int_{\mathbb{R}^3}  \frac{f_{m n} (p_j^0)
  E (t) \cdummy \check{D}_{j, m n}  (\tilde{p}_j) \check{D}^k_{j, m n} 
  (\tilde{p}_j)}{| W_j  \tilde{p}_j |^2} d \tilde{p}_j\\
  & = & \sum_{m, n \in \{ 0, 1 \}} \int_{\mathbb{R}^3}  \frac{c_{00}^{i n}
  (p_j^0) E (t) \cdummy \tmop{Im} (D_{j, 01}  (\tilde{p}_j) D^k_{j, 10} 
  (\tilde{p}_j))}{| W_j  \tilde{p}_j |^2} d \tilde{p}_j .
\end{eqnarray*}
By change of variable, we get
\begin{align*}
  & \sum_{m, n \in \{ 0, 1 \}} \int_{\mathbb{R}^3} \check{D}^k_{m n} | W_j 
  \tilde{p}_j |^{- 2} s_{m n}^{j, 2, b} d \tilde{p}_j\\
  = & \sum_{m, n \in \{ 0, 1 \}} \int_{\mathbb{R}^3}  \frac{c_{00}^{i n}
  (p_j^0) E (t) \cdummy \tmop{Im} (D_{j, 01}  (\tilde{p}_j) D^k_{j, 10} 
  (\tilde{p}_j))}{| W_j  \tilde{p}_j |^2} d \tilde{p}_j\\
  = & \sum_{m, n \in \{ 0, 1 \}} c_{00}^{i n} (p_j^0) \int_0^{+ \infty} r^{-
  2} \int_{\mathbb{S }^2_W} E (t) \cdummy \tmop{Im} (D_{j, 01}  (\omega r)
  D^k_{j, 10}  (\omega r)) d \omega dr\\
  = & 0.
\end{align*}
In the last step, we used \eqref{delta kl} in Proposition \ref{eff result}.
(Notice that the integral of $D_{j, 01}  (\omega r) D^k_{j, 10}  (\omega r)$
is real) This finishes the proof of the second equality of \eqref{s2b vanish}.
For the first equality, the proof is almost the same. The only difference is
that we need to use \eqref{D--} in Proposition \ref{eff result} instead of
\eqref{delta kl}.

%% file: main.bbl
\begin{thebibliography}{10}

\bibitem{armitage2018weyl}
{\sc N.~Armitage, E.~Mele, and A.~Vishwanath}, {\em Weyl and dirac semimetals
  in three-dimensional solids}, Reviews of Modern Physics, 90 (2018),
  p.~015001.

\bibitem{armitage_weyl_2018}
{\sc N.~Armitage, E.~Mele, and A.~Vishwanath}, {\em Weyl and {Dirac} semimetals
  in three-dimensional solids}, Reviews of Modern Physics, 90 (2018),
  p.~015001, \url{https://doi.org/10.1103/RevModPhys.90.015001},
  \url{https://link.aps.org/doi/10.1103/RevModPhys.90.015001} (accessed
  2023-02-22).

\bibitem{berry_quantal_1984}
{\sc M.~V. Berry}, {\em Quantal {Phase} {Factors} {Accompanying} {Adiabatic}
  {Changes}}, Proceedings of the Royal Society of London. Series A,
  Mathematical and Physical Sciences, 392 (1984), pp.~45--57,
  \url{http://www.jstor.org/stable/2397741}.

\bibitem{bloch_sur_1958}
{\sc C.~Bloch}, {\em Sur la théorie des perturbations des états liés},
  Nuclear Physics, 6 (1958), pp.~329--347.
\newblock Publisher: Elsevier.

\bibitem{bloch_uber_1929}
{\sc F.~Bloch}, {\em Über die quantenmechanik der elektronen in
  kristallgittern}, Zeitschrift für physik, 52 (1929), pp.~555--600.
\newblock Publisher: Springer.

\bibitem{bloch_ultracold_2005}
{\sc I.~Bloch}, {\em Ultracold quantum gases in optical lattices}, Nature
  Physics, 1 (2005), pp.~23--30, \url{https://doi.org/10.1038/nphys138},
  \url{http://www.nature.com/articles/nphys138} (accessed 2023-02-22).

\bibitem{carles_semiclassical_2004}
{\sc R.~Carles, P.~A. Markowich, and C.~Sparber}, {\em Semiclassical
  {Asymptotics} for {Weakly} {Nonlinear} {Bloch} {Waves}}, Journal of
  Statistical Physics, 117 (2004), pp.~343--375,
  \url{https://doi.org/10.1023/B:JOSS.0000044070.34410.17},
  \url{http://link.springer.com/10.1023/B:JOSS.0000044070.34410.17} (accessed
  2023-02-21).

\bibitem{culcer_transport_2020}
{\sc D.~Culcer, A.~Cem~Keser, Y.~Li, and G.~Tkachov}, {\em Transport in
  two-dimensional topological materials: recent developments in experiment and
  theory}, 2D Materials, 7 (2020), p.~022007,
  \url{https://doi.org/10.1088/2053-1583/ab6ff7},
  \url{https://iopscience.iop.org/article/10.1088/2053-1583/ab6ff7} (accessed
  2023-02-22).

\bibitem{de_juan_quantized_2017}
{\sc F.~de~Juan, A.~G. Grushin, T.~Morimoto, and J.~E. Moore}, {\em Quantized
  circular photogalvanic effect in {Weyl} semimetals}, Nature Communications, 8
  (2017), p.~15995, \url{https://doi.org/10.1038/ncomms15995},
  \url{https://www.nature.com/articles/ncomms15995} (accessed 2023-02-22).

\bibitem{drese_floquet_1999}
{\sc K.~Drese and M.~Holthaus}, {\em Floquet theory for short laser pulses},
  The European Physical Journal D-Atomic, Molecular, Optical and Plasma
  Physics, 5 (1999), pp.~119--134.
\newblock Publisher: Springer.

\bibitem{e_asymptotic_2013}
{\sc W.~E, J.-f. Lu, and X.~Yang}, {\em Asymptotic analysis of quantum dynamics
  in crystals: the {Bloch}-{Wigner} transform, {Bloch} dynamics and {Berry}
  phase}, Acta Mathematicae Applicatae Sinica, English Series, 29 (2013),
  pp.~465--476, \url{https://doi.org/10.1007/s10255-011-0095-5},
  \url{http://link.springer.com/10.1007/s10255-011-0095-5} (accessed
  2023-02-21).

\bibitem{fefferman_honeycomb_2012}
{\sc C.~Fefferman and M.~Weinstein}, {\em Honeycomb lattice potentials and
  {Dirac} points}, Journal of the American Mathematical Society, 25 (2012),
  pp.~1169--1220, \url{https://doi.org/10.1090/S0894-0347-2012-00745-0},
  \url{https://www.ams.org/jams/2012-25-04/S0894-0347-2012-00745-0/} (accessed
  2023-02-21).

\bibitem{fefferman_bifurcations_2016}
{\sc C.~L. Fefferman, J.~P. Lee-Thorp, and M.~I. Weinstein}, {\em Bifurcations
  of edge states—topologically protected and non-protected—in continuous
  {2D} honeycomb structures}, 2D Materials, 3 (2016), p.~014008,
  \url{https://doi.org/10.1088/2053-1583/3/1/014008},
  \url{https://iopscience.iop.org/article/10.1088/2053-1583/3/1/014008}
  (accessed 2023-02-21).

\bibitem{fefferman_edge_2016}
{\sc C.~L. Fefferman, J.~P. Lee-Thorp, and M.~I. Weinstein}, {\em Edge {States}
  in {Honeycomb} {Structures}}, Annals of PDE, 2 (2016), p.~12,
  \url{https://doi.org/10.1007/s40818-016-0015-3},
  \url{http://link.springer.com/10.1007/s40818-016-0015-3} (accessed
  2023-02-21).

\bibitem{gaim2021higher}
{\sc W.~Gaim}, {\em Higher Order Semiclassical Approximations for Hamiltonians
  with Operator-Valued Symbols}, PhD thesis, Universit{\"a}t T{\"u}bingen,
  2021.

\bibitem{gao_field_2014}
{\sc Y.~Gao, S.~A. Yang, and Q.~Niu}, {\em Field induced positional shift of
  {Bloch} electrons and its dynamical implications}, Physical review letters,
  112 (2014), p.~166601.
\newblock Publisher: APS.

\bibitem{giuliani2021anomaly}
{\sc A.~Giuliani, V.~Mastropietro, and M.~Porta}, {\em Anomaly
  non-renormalization in interacting weyl semimetals}, Communications in
  Mathematical Physics, 384 (2021), pp.~997--1060.

\bibitem{gerard_homogenization_1997}
{\sc P.~Gérard, P.~A. Markowich, N.~J. Mauser, and F.~Poupaud}, {\em
  Homogenization limits and {Wigner} transforms}, Communications on Pure and
  Applied Mathematics: A Journal Issued by the Courant Institute of
  Mathematical Sciences, 50 (1997), pp.~323--379.
\newblock Publisher: Wiley Online Library.

\bibitem{herring1937accidental}
{\sc C.~Herring}, {\em Accidental degeneracy in the energy bands of crystals},
  Physical Review, 52 (1937), p.~365.

\bibitem{hoddeson_development_1987}
{\sc L.~Hoddeson, G.~Baym, and M.~Eckert}, {\em The development of the
  quantum-mechanical electron theory of metals: 1928—1933}, Reviews of Modern
  Physics, 59 (1987), pp.~287--327,
  \url{https://doi.org/10.1103/RevModPhys.59.287},
  \url{https://link.aps.org/doi/10.1103/RevModPhys.59.287} (accessed
  2023-02-21).

\bibitem{hovermann_semiclassical_2001}
{\sc F.~Hövermann, H.~Spohn, and S.~Teufel}, {\em Semiclassical {Limit} for
  the {Schrödinger} {Equation}¶with a {Short} {Scale} {Periodic}
  {Potential}}, Communications in Mathematical Physics, 215 (2001),
  pp.~609--629, \url{https://doi.org/10.1007/s002200000314},
  \url{http://link.springer.com/10.1007/s002200000314} (accessed 2023-03-01).

\bibitem{jorgensen_effective_1975}
{\sc F.~Jørgensen}, {\em Effective hamiltonians}, Molecular Physics, 29
  (1975), pp.~1137--1164, \url{https://doi.org/10.1080/00268977500100971},
  \url{http://www.tandfonline.com/doi/abs/10.1080/00268977500100971} (accessed
  2023-02-22).

\bibitem{krieger_time_1986}
{\sc J.~B. Krieger and G.~J. Iafrate}, {\em Time evolution of {Bloch} electrons
  in a homogeneous electric field}, Physical Review B, 33 (1986),
  pp.~5494--5500, \url{https://doi.org/10.1103/PhysRevB.33.5494},
  \url{https://link.aps.org/doi/10.1103/PhysRevB.33.5494} (accessed
  2023-02-22).

\bibitem{krieger_quantum_1987}
{\sc J.~B. Krieger and G.~J. Iafrate}, {\em Quantum transport for {Bloch}
  electrons in a spatially homogeneous electric field}, Physical Review B, 35
  (1987), pp.~9644--9658, \url{https://doi.org/10.1103/PhysRevB.35.9644},
  \url{https://link.aps.org/doi/10.1103/PhysRevB.35.9644} (accessed
  2023-02-22).

\bibitem{lu_bloch_2022}
{\sc J.~Lu, Z.~Zhang, and Z.~Zhou}, {\em Bloch dynamics with second order
  {Berry} phase correction}, Asymptotic Analysis, 128 (2022), pp.~55--84.
\newblock Publisher: IOS Press.

\bibitem{markowich_wignerfunction_1994}
{\sc P.~A. Markowich, N.~J. Mauser, and F.~Poupaud}, {\em A {Wigner}‐function
  approach to (semi)classical limits: {Electrons} in a periodic potential},
  Journal of Mathematical Physics, 35 (1994), pp.~1066--1094,
  \url{https://doi.org/10.1063/1.530629},
  \url{http://aip.scitation.org/doi/10.1063/1.530629} (accessed 2023-02-21).

\bibitem{nenciu_dynamics_1991}
{\sc G.~Nenciu}, {\em Dynamics of band electrons in electric and magnetic
  fields: rigorous justification of the effective {Hamiltonians}}, Reviews of
  Modern Physics, 63 (1991), p.~91.
\newblock Publisher: APS.

\bibitem{peierls_zur_1929}
{\sc R.~Peierls}, {\em Zur theorie der galvanomagnetischen effekte},
  Zeitschrift für Physik, 53 (1929), pp.~255--266.
\newblock Publisher: Springer.

\bibitem{reed_analysis_1978}
{\sc M.~Reed}, {\em Analysis of operators}, Method of Modern Mathematical
  Physics, 4 (1978).
\newblock Publisher: Academic press.

\bibitem{sie_valley-selective_2018}
{\sc E.~J. Sie and E.~J. Sie}, {\em Valley-selective optical {Stark} effect in
  monolayer {WS} 2}, Coherent Light-Matter Interactions in Monolayer
  Transition-Metal Dichalcogenides,  (2018), pp.~37--57.
\newblock Publisher: Springer.

\bibitem{sipe_second-order_2000}
{\sc J.~E. Sipe and A.~I. Shkrebtii}, {\em Second-order optical response in
  semiconductors}, Physical Review B, 61 (2000), pp.~5337--5352,
  \url{https://doi.org/10.1103/PhysRevB.61.5337},
  \url{https://link.aps.org/doi/10.1103/PhysRevB.61.5337} (accessed
  2023-02-22).

\bibitem{sparber_wigner_2003}
{\sc C.~Sparber, P.~A. Markowich, and N.~J. Mauser}, {\em Wigner functions
  versus {WKB}‐methods in multivalued geometrical optics}, Asymptotic
  Analysis, 33 (2003), pp.~153--187.
\newblock Publisher: IOS Press.

\bibitem{spohn_adiabatic_2001}
{\sc H.~Spohn and S.~Teufel}, {\em Adiabatic {Decoupling} and
  {Time}-{Dependent} {Born}–{Oppenheimer} {Theory}}, Communications in
  Mathematical Physics, 224 (2001), pp.~113--132,
  \url{https://doi.org/10.1007/s002200100535},
  \url{http://link.springer.com/10.1007/s002200100535} (accessed 2023-03-01).

\bibitem{teufel_adiabatic_2003}
{\sc S.~Teufel}, {\em Adiabatic perturbation theory in quantum dynamics},
  no.~1821 in Lecture notes in mathematics, Springer, Berlin ; New York, 2003.

\bibitem{teufel_semiclassical_2002}
{\sc S.~Teufel and H.~Spohn}, {\em {SEMICLASSICAL} {MOTION} {OF} {DRESSED}
  {ELECTRONS}}, Reviews in Mathematical Physics, 14 (2002), pp.~1--28,
  \url{https://doi.org/10.1142/S0129055X02001077},
  \url{https://www.worldscientific.com/doi/abs/10.1142/S0129055X02001077}
  (accessed 2023-03-01).

\bibitem{wan2011topological}
{\sc X.~Wan, A.~M. Turner, A.~Vishwanath, and S.~Y. Savrasov}, {\em Topological
  semimetal and fermi-arc surface states in the electronic structure of
  pyrochlore iridates}, Physical Review B, 83 (2011), p.~205101.

\bibitem{wan_topological_2011}
{\sc X.~Wan, A.~M. Turner, A.~Vishwanath, and S.~Y. Savrasov}, {\em Topological
  semimetal and {Fermi}-arc surface states in the electronic structure of
  pyrochlore iridates}, Physical Review B, 83 (2011), p.~205101,
  \url{https://doi.org/10.1103/PhysRevB.83.205101},
  \url{https://link.aps.org/doi/10.1103/PhysRevB.83.205101} (accessed
  2023-02-22).

\bibitem{wannier_wave_1960}
{\sc G.~H. Wannier}, {\em Wave {Functions} and {Effective} {Hamiltonian} for
  {Bloch} {Electrons} in an {Electric} {Field}}, Physical Review, 117 (1960),
  pp.~432--439, \url{https://doi.org/10.1103/PhysRev.117.432},
  \url{https://link.aps.org/doi/10.1103/PhysRev.117.432} (accessed 2023-02-22).

\bibitem{weyl1929gravitation}
{\sc H.~Weyl}, {\em Gravitation and the electron}, Proceedings of the National
  Academy of Sciences, 15 (1929), pp.~323--334.

\bibitem{xiao_berry_2010}
{\sc D.~Xiao, M.-C. Chang, and Q.~Niu}, {\em Berry phase effects on electronic
  properties}, Reviews of Modern Physics, 82 (2010), pp.~1959--2007,
  \url{https://doi.org/10.1103/RevModPhys.82.1959},
  \url{https://link.aps.org/doi/10.1103/RevModPhys.82.1959} (accessed
  2023-01-29).

\bibitem{zhang_anomalous_2014}
{\sc H.~Zhang, J.~Yao, J.~Shao, H.~Li, S.~Li, D.~Bao, C.~Wang, and G.~Yang},
  {\em Anomalous photoelectric effect of a polycrystalline topological
  insulator film}, Scientific reports, 4 (2014), p.~5876.
\newblock Publisher: Nature Publishing Group UK London.

\end{thebibliography}
